\documentclass[10pt]{amsart}

\usepackage{a4wide}
\usepackage{graphicx}
\usepackage{multirow}
\usepackage{footnote}
\theoremstyle{plain}
\newtheorem{proposition}{Proposition}[section]
\newtheorem{theorem}[proposition]{Theorem}
\newtheorem{lemma}[proposition]{Lemma}

\theoremstyle{remark}
\newtheorem{ex}{Example}
\newtheorem{remark}[proposition]{Remark}
\newtheorem{definition}[proposition]{Definition}

\newcommand {\A}  {{\mathcal A}}
\newcommand {\M}  {{\mathbf M}}

\newcommand {\T}  {{\mathcal T}}
\newcommand {\NN}  {{\mathbb N}}
\newcommand {\RR}  {{\mathbb R}}

\newcommand {\lBF} {\mathbf{l}}
\newcommand {\uBF} {\mathbf{u}}
\newcommand {\vBF} {\mathbf{v}}
\newcommand {\hBF} {\mathbf{h}}

\numberwithin{equation}{section}

\title[Topological properties of a class of cubic Rauzy fractals]{Topological properties of a class of cubic Rauzy fractals}

\author{Beno\^it Loridant}

\address{
Montanuniversit\"at Leoben\\Lehrstuhl Mathematik \& Statistik\\ Franz Josef Strasse 18\\8700 Leoben Austria
}

\email{benoit.loridant@unileoben.ac.at}

\date{\today}

\begin{document}
\begin{abstract} We consider the substitution $\sigma_{a,b}$ defined by
$$\begin{array}{rlcl}
\sigma_{a,b}: & 1 & \mapsto & \underbrace{1\ldots 1}_{a}2 \\
& 2 & \mapsto & \underbrace{1\ldots 1}_{b}3 \\
& 3 & \mapsto & 1
\end{array}
$$ 
with $a\geq b\geq 1$. 
The shift dynamical system induced by $\sigma_{a,b}$ is measure theoretically isomorphic to an exchange of three domains on a compact tile $\T_{a,b}$ with fractal boundary. 

We prove that $\T_{a,b}$ is homeomorphic to the closed disk iff $2b-a\leq 3$. This solves a conjecture of Shigeki Akiyama posed in 1997. To this effect, we construct a H\"older continuous parametrization $C_{a,b}:\mathbb{S}^1\to\partial \T_{a,b}$ of the boundary of $\T_{a,b}$. As a by-product, this parametrization gives rise to an increasing sequence of polygonal approximations of $\partial \T_{a,b}$, whose vertices lye on $\partial \T_{a,b}$ and have algebraic pre-images in the parametrization. 
\end{abstract}
\thanks{This research was supported by the project P22-855 of the Austrian Science Fund (FWF) and by the project FAN-I1136 of the FWF and the ANR (Agence Nationale de la Recherche).}

 \maketitle

\begin{section}{Introduction}

In 1982, G. Rauzy studied the dynamical system generated by the substitution $\sigma(1)=12,\;\sigma(2)=13,\;\sigma(3)=1$ and proved that it is measure theoretically conjugate to a domain exchange on a compact subset $\T$ of the complex plane~\cite{Rauzy82}. Moreover, it has pure discrete spectrum and it is isomorphic to translation on the two dimensional torus. $\T$ has a self-similar structure and induces both a periodic and an aperiodic tiling of the plane. The results of Rauzy were generalized. A \emph{Rauzy fractal} $\T\subset\mathbb{R}^{d-1}$ can be attached to each irreducible unimodular Pisot substitution $\sigma$ on $d$ letters. The shift dynamical system generated by $\sigma$ is measure theoretically isomorphic to a domain exchange on $d$ subtiles of  $\T$, provided that $\sigma$ satisfies the combinatorial \emph{strong coincidence condition}~\cite{ArnouxIto01,CanteriniSiegel01b}. If $\sigma$ satisfies the \emph{super coincidence condition}, the shift dynamical system has even pure discrete spectrum and is measure theoretically isomorphic to a translation on the $(d-1)$ dimensional torus (\cite{ItoRao06a,BargeKwapisz06}). In this case, the tile $\T$ induces a periodic tiling and the subtiles $\T(i)$ for $i\in\{1,\ldots,d\}$ an aperiodic self-replicating tiling of $\mathbb{R}^{d-1}$ \cite{ItoRao06a}. In fact, the outstanding \emph{Pisot conjecture} states that the dynamical system generated by every irreducible unimodular Pisot substitution has pure discrete spectrum. 

There is a vast literature on Rauzy fractals, as they appear naturally in many domains. In $\beta$-numeration (\cite{Thurston89}), finiteness properties of digit representations are related to the fact that $0$ is an inner point of the Rauzy fractal, and the intersection of the Rauzy fractal with lines allows to characterize the rationals numbers with purely periodic expansion~\cite{AkiyamaBaratBertheSiegel08}. In Diophantine approximation, best simultaneous approximations are obtained by computing the size of the largest ball inside the Rauzy fractal~\cite{HubertMessaoudi06}. Rauzy fractals also play an important r\^ole in the construction of Markov partitions for toral automorphisms.  It is known that every hyperbolic automorphism of the $d$-dimensional torus admits a Markov partition~\cite{Sinai68,Bowen70}.  For $d=2$, the partition is made of rectangles~\cite{AdlerWeiss70}. However, for $d>2$, the partition can not have a smooth boundary~\cite{Bowen78}. Markov partitions for hyperbolic toral automorphisms were explicitly constructed in~\cite{Praggastis92,Praggastis99,ItoOtsuki93} using cylinders whose bases are the original subtiles of the Rauzy fractals. Whenever the Rauzy fractal is homeomorphic to the closed disk, the situation remains close to the case $d=2$, as the Markov partition consists in topological $3$-dimensional balls.

 In their monograph~\cite{SiegelThuswaldner10}, Siegel and Thuswaldner give algorithms to check topological properties such as  tiling property, connectedness or homeomorphy to the closed disk  for any given Pisot unimodular substitution. These criteria use graphs and rely on the self-similar structure of the Rauzy fractals. However, it is usually more difficult to describe the topological properties for whole families of Rauzy fractals.
 
In this paper, we consider the Rauzy fractals $\T_{a,b}$ associated with the substitutions
\[\begin{array}{rlcl}
\sigma_{a,b}: & 1 & \mapsto & \underbrace{1\ldots 1}_{a\textrm{ times}}2 \\
& 2 & \mapsto & \underbrace{1\ldots 1}_{b\textrm{ times}}3 \\
& 3 & \mapsto & 1
\end{array}\]
over the alphabet $\{1,2,3\}$, where $a\geq b\geq 1$. For every such parameters $a,b$, $\sigma_{a,b}$ is an irreducible primitive unimodular Pisot substitution. Moreover, it satisfies the super coincidence condition~\cite{BargeKwapisz06,Solomyak92}. Therefore, $\T_{a,b}$ induces a periodic tiling and its subtiles $\T_{a,b}(i)$ ($i=1,2,3$) an aperiodic self-replicating tiling of the plane.

We will show that $\T_{a,b}$ is homeomorphic to the closed disk if and only if $2b-a\leq 3$. 
This solves a conjecture of Shigeki Akiyama announced in 1997~\cite{AkiyamaNote1,AkiyamaNote2}. To this effect, we will construct a parametrization of the boundary of $\T_{a,b}$. A standard method for the boundary parametrization of self-affine tiles was proposed by Shigeki Akiyama and the author in~\cite{AkiyamaLoridant11}. We will be able to extend this construction for the boundary of our substitution tiles, as it mainly relies on the graph-directed self-similar structure of the boundary. A by-product of the parametrization is a sequence of boundary approximations whose way of generation is analogous to Dekking's recurrent set method~\cite{Dekking82b,Dekking82a}.

We mention existing results. In the case $b=1$, the tiles $\T_{a,1}$ were shown to be disk-like and the Hausdorff dimension of their boundary was computed by Messaoudi~\cite{Messaoudi00,Messaoudi06} via a boundary parametrization, but the technique used to parametrize would not generalize to the non disk-like tiles. In~\cite{ItoKimura91}, Ito and Kimura  produced the boundary of $\T_{1,1}$ by Dekking's fractal generating method, making use of higher dimensional geometric realizations of the Tribonacci substitution. This also allowed the computation of the Hausdorff dimension of the boundary. They could generalize their method in~\cite{SanoArnouxIto01}. In~\cite{Thuswaldner06}, Thuswaldner computed the so-called contact graph, related to the aperiodic tilings induced by $\T_{a,b}$, for the whole class of substitutions $\sigma_{a,b}$ and deduced the Hausdorff dimension of the boundary of $\T_{a,b}$. This graph will be of great importance in our parametrization procedure. In~\cite{LoridantMessaoudiSurerThuswaldner13}, the non-disk-likeness for the parameters satisfying $2b-a>3$ was proved. Indeed, the authors obtained a subgraph of the lattice boundary graph, associated with the periodic tiling induced by $\T_{a,b}$, for all parameters $a\geq b\geq 1$. It turned out that for $2b-a>3$, the number of states in this graph, which is also the number of neighbors of $\T_{a,b}$ in the periodic tiling, is strictly larger than 8. However, in a periodic tiling induced by a topological disk, the tiles have either $6$ or $8$ neighbors~\cite{GruenbaumShephard89}. Therefore, $\T_{a,b}$ is not homeomorphic to a disk. We will recover this result by another method based only on the contact graphs, showing that the parametrization is not injective for these parameters. The proof of the counterpart is more intricate, as it consists in showing the injectivity of the parametrization for $2b-a\leq3 $: this requires rather involved computations on B\"uchi automata. 

The paper is organized as follows. In Section~\ref{sec:mainresults}, we recall basic facts concerning our class of substitutions and formulate our main results. In Section~\ref{sec:GIFS}, we introduce two graphs that are essential in our work: the boundary graph $\mathcal{G}_{0,a,b}$, that describes the whole language of the boundary of $\T_{a,b}$, and a subgraph $G_{0,a,b}\subset\mathcal{G}_{0,a,b}$, whose language is large enough to cover the boundary. In Section~\ref{sec:boundparam}, we use the graph $G_{0,a,b}$ to construct the boundary parametrization, proving Theorem~\ref{ParamTheo}. Section~\ref{sec:prooftheo} is devoted to the proof of Theorem~\ref{DNDTheo}. If $2b-a\leq 3$, then $G_{0,a,b}=\mathcal{G}_{0,a,b}$ and we can show that the parametrization is injective. Therefore, $\partial\T_{a,b}$ is a simple closed curve and $\T_{a,b}$ is disk-like. Otherwise, the complement of $G_{0,a,b}$ in $\mathcal{G}_{0,a,b}$ is nonempty and we can find a redundant point in the parametrization.  Finally, in Section~\ref{sec:conc}, we add some comments and questions for further work.

\emph{Acknowledgements.} The author is grateful to Shigeki Akiyama and Shunji It\={o} for mentioning the conjecture and for the motivating discussions on this subject.
\end{section}

\begin{section}{Main results}\label{sec:mainresults}

We wish to study the topological properties of a class tiles arising from a family of substitutions. 
\subsection{Substitutions $\sigma_{a,b}$}
Let $\A:=\{1,2,3\}$ be the \emph{alphabet}. We denote by $\A^*$ the set of finite words over $\A$, including the empty word $\varepsilon$. For $a\geq b\geq 1$, we call $\sigma=\sigma_{a,b}:\A^*\to\A^*$ the mapping
\begin{equation}\label{DefSubst}
\begin{array}{rlcl}
\sigma: & 1 & \mapsto & \underbrace{1\ldots 1}_{a\textrm{ times}}2 \\
& 2 & \mapsto & \underbrace{1\ldots 1}_{b\textrm{ times}}3 \\
& 3 & \mapsto & 1,
\end{array}
\end{equation}
extended to $\A^*$ by concatenation. 

For a word $w\in\A^*$, we write $|w|$ its \emph{length} and  $|w|_a$ the number of occurrences of a letter $a$ in $w$. We define the \emph{abelianization mapping}
$$\lBF:w\in\A^*\mapsto \left(|w|_a\right)_{a\in\A}\in\mathbb{N}^3
$$ 
The \emph{incidence matrix} $\M$ of the substitution $\sigma$ is the $3\times 3$ matrix obtained by abelianization:
\begin{equation}\label{rel:abel}
\lBF(\sigma(w))=\M\lBF(w)
\end{equation}

for all $w\in\A^*$. Thus we have

\[\M=\left(\begin{array}{ccc} a & b & 1 \\ 1 & 0 & 0 \\ 0 & 1 & 0 \end{array}\right).\]

$\M$ is a primitive matrix, {\it i.e.}, $\M^k$ has only strictly positive entries for some power $k \in \NN$ (here, $k=3$). We denote by  $\beta$ the corresponding dominant Perron-Frobenius eigenvalue, satisfying $\beta^3=a\beta^2+b\beta+1$.
The substitution $\sigma$ has the following properties. It is
\begin{itemize}
\item  {\it primitive}: the incidence matrix $\M$ is a primitive matrix;
\item {\it unimodular}: $\beta$ is an algebraic unit;
\item {\it irreducible}: the algebraic degree of $\beta$ is exactly $|{\A}|=3$;
\item {\it Pisot}: the Galois conjugates $\alpha_1,\alpha_2$ of $\beta$ satisfy $|\alpha_1|,|\alpha_2|<1$ (see~\cite{Brauer51}).
\end{itemize}

\subsection{Associated Rauzy fractals $\T_{a,b}$}\label{subsec:RF} We turn to the construction of the Rauzy fractals associated with the substitution $\sigma$. 

Let ${\bf v}_\beta$ be a strictly positive left eigenvector of
$\M$ for the dominant eigenvalue $\beta$  and $\uBF_\beta$ a strictly positive right eigenvector with coordinates in $\mathbb{Z}[\beta]$, satisfying $\langle\uBF_\beta,\vBF_\beta\rangle =1$.
Moreover, let  $\uBF_{\alpha_i}$ be the eigenvectors for the Galois conjugates obtained by replacing $\beta$ by $\alpha_i$ in the coordinates of the vector  $\uBF_\beta$. We obtain the decomposition
\[\mathbb{R}^3=\mathbb{H}_e\oplus\mathbb{H}_c,\]
where
\begin{itemize}
\item $\mathbb{H}_e$ is the \emph{expanding line}, generated by ${\bf u}_{\beta}$,
\item $\mathbb{H}_c$ is the \emph{contracting plane}, generated by ${\bf u}_{\alpha_1},{\bf u}_{\alpha_2}$ (or by $\Re({\bf u}_{\alpha_1}),\Im({\bf u}_{\alpha_1})$ whenever $\alpha_1,\alpha_2$ are complex conjugates).
\end{itemize}
We denote by $\pi: \RR^3 \rightarrow \mathbb{H}_c$ the projection onto $\mathbb{H}_c$ along $\mathbb{H}_e$ and by $\hBF$ the restriction of $\M$ on the contractive plane $\mathbb{H}_c$. Note that if we define the norm 
$$||{\bf x}||=\max\left\{|\langle{\bf x},{\bf v}_{\alpha_1}\rangle|,|\langle{\bf x},{\bf v}_{\alpha_2}\rangle|\right\},
$$
then $\hBF$ is a contraction with $||{\bf hx} ||\leq {\rm max}\{|\alpha_1|,|\alpha_2|\}|{|\bf x}||$ for all ${\bf x} \in \mathbb{H}_c$.

Furthermore, we have
\begin{equation}\label{fc}
\forall w \in \A^*,\quad \hBF(\pi({\bf l}(w)))= \pi(\M{\bf l}(w)) = \pi({\bf l}(\sigma(w))).
\end{equation}

The fixed point $w=w_0w_1w_2\cdots=\lim_{k\to\infty}\sigma^k(1)\in\A^{\mathbb{N}}$ imbeds into $\mathbb{R}^3$ as a discrete line with vertices $\{\lBF(w_0\cdots w_n);n\in\mathbb{N}\}$. The assumption that $\sigma$ is a Pisot substitution implies that this broken line remains at a bounded distance of the expanding line. Projecting the vertices of the discrete line on the contracting plane, we obtain the Rauzy fractal of $\sigma$ (see \cite{ArnouxIto01}):
$$\begin{array}{rcl}
\mathcal{T}=\T_{a,b}&=&\overline{\{\pi\circ \lBF(w_0w_1\ldots w_{n-1});n\in\mathbb{N}\}},\\\\
\forall i\in\A,\;\mathcal{T}(i)=\T_{a,b}(i)&=&\overline{\{\pi\circ \lBF(w_0w_1\ldots w_{n-1});w_n=i,n\in\mathbb{N}\}}.
\end{array}
$$
For our purpose, we will need to view the Rauzy fractals as solution of a \emph{graph directed iteration function system} (\emph{GIFS}, see~\cite{MauldinWilliams88}). The appropriate graph is the \emph{prefix-suffix graph}, defined as in~\cite{CanteriniSiegel01a}: 
\begin{itemize}
\item vertices: the letters of $\A$;
\item edges: $i\xrightarrow{p}j$ if and only if $\sigma(j)=pis$ for some $s\in\A^*$.
\end{itemize}
The prefix-suffix graph $\Gamma=\Gamma_{a,b}$ of $\sigma$ is depicted on Figure~\ref{PrefSuff}.
\begin{figure}
\begin{center}
\includegraphics[width=105mm,height=35mm]{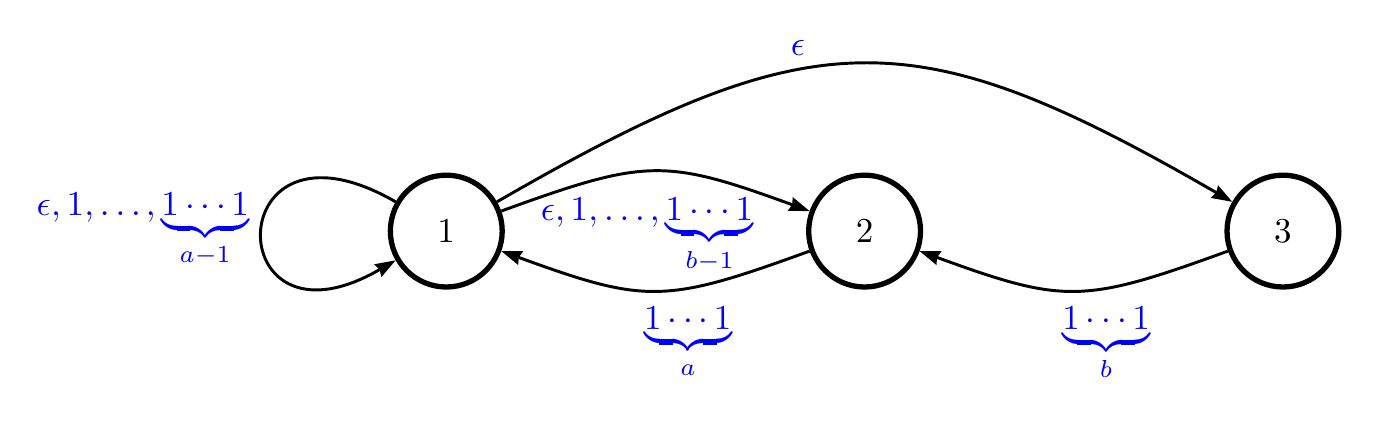}
\end{center}
 \caption{Prefix-suffix graph: $i\xrightarrow{p}j\in\Gamma\iff \sigma(j)=pis$.}\label{PrefSuff}
\end{figure}

Since $\sigma$ is a primitive unimodular Pisot substitution, $\T$ is the attractor of the GIFS  defined by the prefix-suffix graph (see for example~\cite{BertheSiegel05}):
\begin{equation}\label{TileGIFS}
\begin{array}{c}
\forall i\in\mathcal{A},\;\mathcal{T}(i) =  \bigcup_{i\xrightarrow{p}j} \mathbf{h}\mathcal{T}(j) + \pi \lBF ( p),\\\\
\T=\bigcup_{i=1}^3\T(i).
\end{array}
\end{equation}

From this GIFS structure we deduce that the Rauzy fractal and its subtiles are a geometric representation of the language of the prefix-suffix graph~\cite{CanteriniSiegel01b}:
\[\T= \left\{\sum_{k \geq 0}\hBF^k \pi(\lBF(p_k));\, i_0\xrightarrow{p_0}i_1\xrightarrow{p_1}i_{2}\xrightarrow{p_2}  \ldots\in \Gamma\right\}\]
and for $i\in\A$
\begin{equation}\label{deftiles}
\T(i)= \left\{\sum_{k \geq 0}\hBF^k \pi(\lBF(p_k));\,  i_0=i\xrightarrow{p_0}i_1\xrightarrow{p_1}i_{2}\xrightarrow{p_2}  \ldots\in \Gamma\right\}.
\end{equation}

There are other equivalent constructions of the Rauzy fractal. An overview of the different methods can be found in~\cite{BertheRigo10}. 

Fundamental topological properties of these Rauzy fractals can be found in the literature. 
\begin{itemize}
\item[(1)] $\mathcal{T}$ is a compact set and $\mathcal{T}=\overline{{\mathcal{T}}^o}$.
\item[(2)]  For $i=1,2,3$,  the subtile $\mathcal{T}(i)$ is a compact set and $\mathcal{T}(i)=\overline{{\mathcal{T}(i)}^o}$.
\item[(3)] The subtiles induce an aperiodic tiling of the contracting plane. Let $(\mathbf{e}_1,\mathbf{e}_2,\mathbf{e}_3)$ be the canonical basis of $\mathbb{R}^3$. The tiling set is 
$$\Gamma_{srs}:=\left\{[\pi(\mathbf{x}),i]\in\pi(\mathbb{Z}^3)\times\mathcal{A}\;;\;0\leq \langle\mathbf{x},\mathbf{v}_\beta\rangle< \langle\mathbf{e}_i,\mathbf{v}_\beta\rangle\right\}
$$
and
\begin{equation}\label{AperTiling}
\begin{array}{c}
\forall\; [\gamma,i]\ne[\gamma',j]\in\Gamma_{srs},\;(\T(i)+\gamma)^o\cap(\T(j)+\gamma')^o=\emptyset,\\\\
\mathbb{H}_c=\bigcup_{[\gamma,i]\in\Gamma_{srs}}\mathcal{T}(i)+\gamma.
\end{array}
\end{equation}

\end{itemize} 
(1) and (2) hold because $\sigma$ is a primitive unimodular Pisot substitution~\cite{SirventWang02}. (3) is a consequence of the combinatorial \emph{super coincidence condition} satisfied by $\sigma$. Indeed, Solomyak~ \cite{Solomyak92} proved in 1992 that the associated dynamical system has pure discrete spectrum, and Barge and Kwapisz~\cite{BargeKwapisz06} showed in 2006 that this is equivalent to the super coincidence condition for the substitution. By \cite{ItoRao06a}, the subtiles $\T(i)$ ($i=1,2,3$) induce the aperiodic tiling of the plane (\ref{AperTiling}). This tiling is also self-replicating (see~\cite[Chapter 3]{SiegelThuswaldner10}). Examples are depicted in Figure~\ref{AperTilingEx}.

\begin{figure}
\begin{center}
\begin{tabular}{cc}
\includegraphics[width=75mm,height=45mm]{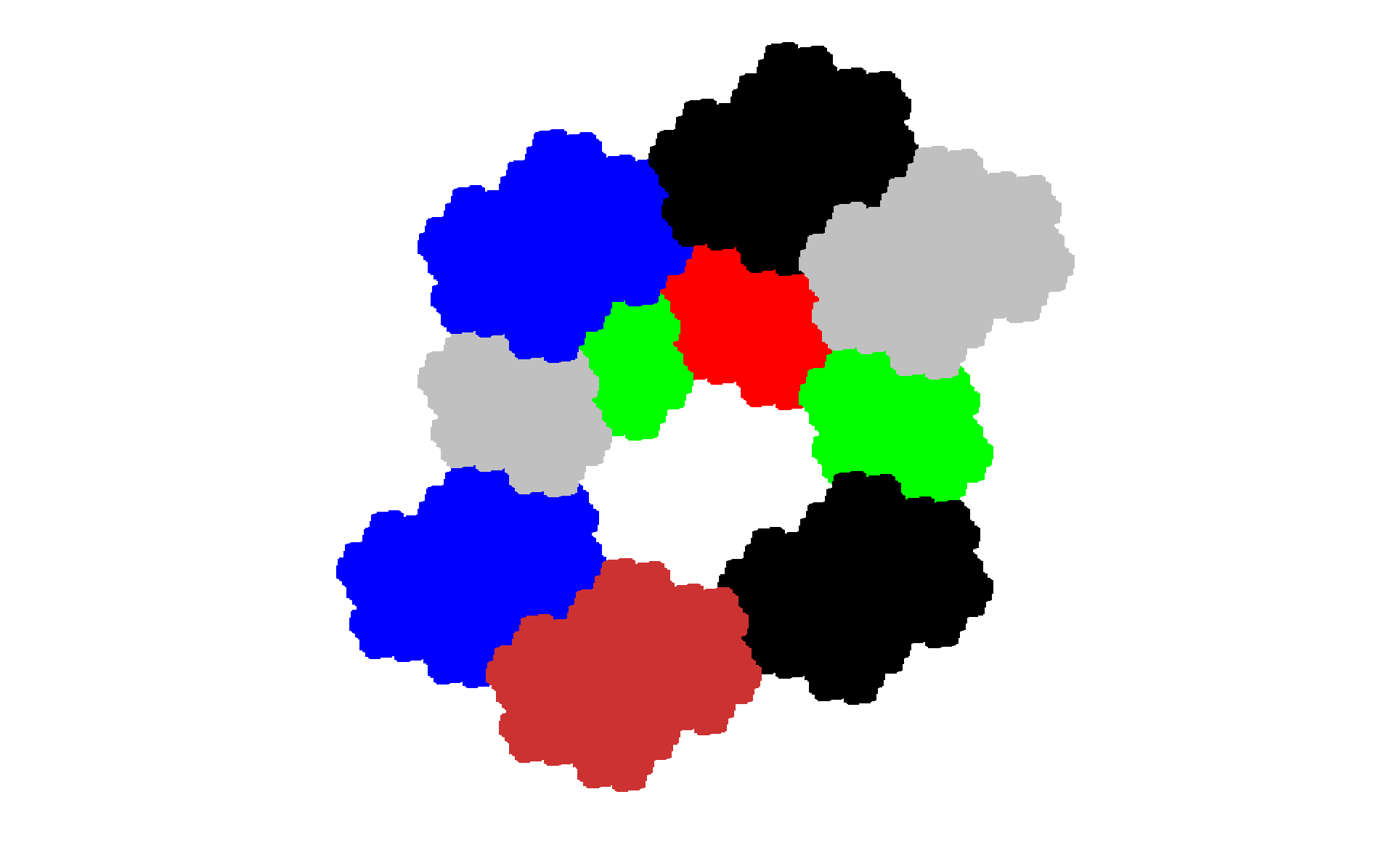}
&
\includegraphics[width=75mm,height=45mm]{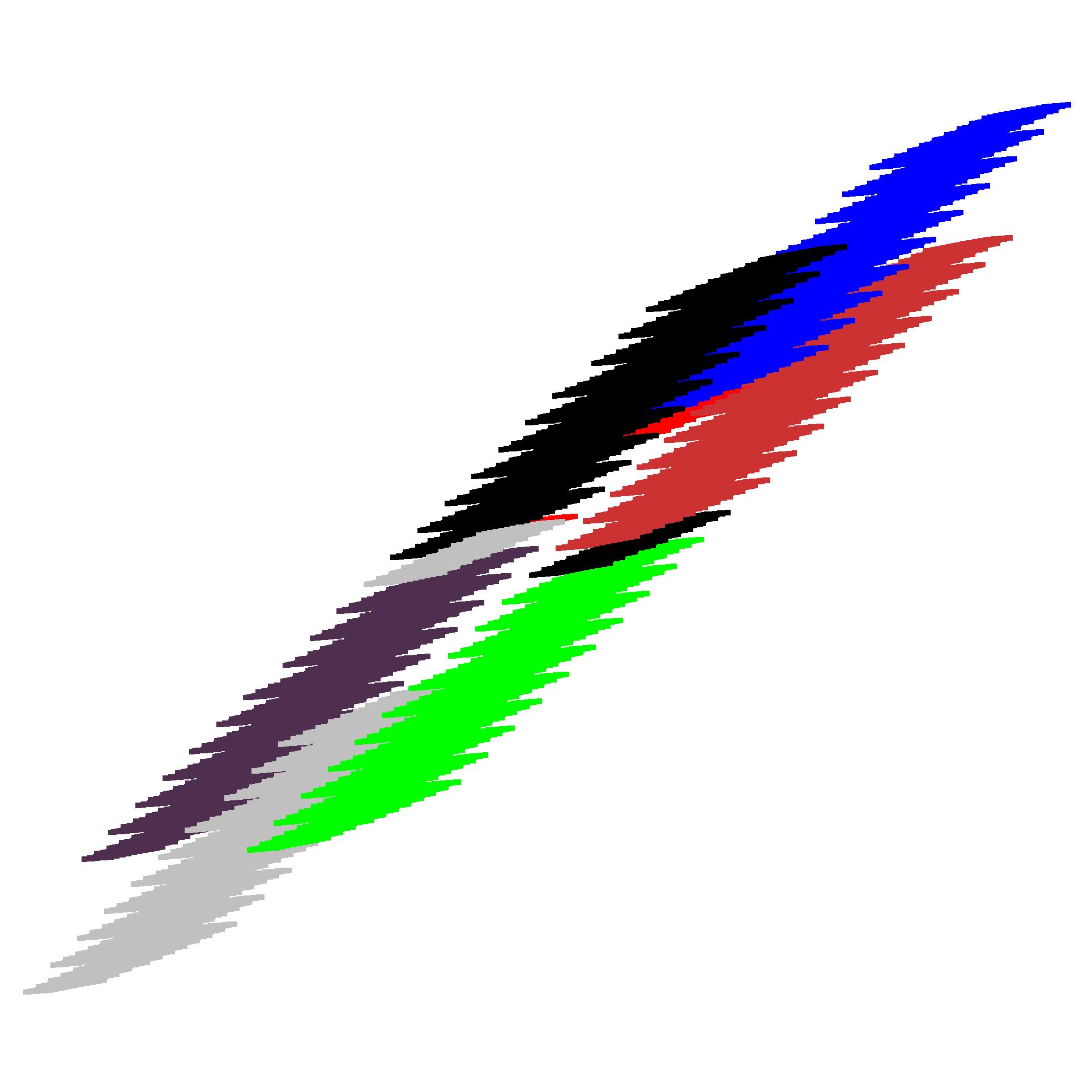}
\\
Tribonacci substitution
&
Substitution $\sigma_{7,10}$
\end{tabular}
\caption{Aperiodic self-replicating tilings of the contracting plane}\label{AperTilingEx}
\end{center}
\end{figure}

In this paper, we will prove the following theorem. 

\begin{theorem}\label{DNDTheo}
Consider the substitution $\sigma_{a,b}$ ($a\geq b\geq 1$) defined in~(\ref{DefSubst}) and let $\T_{a,b}$  be its Rauzy fractal. Then 
$$\T_{a,b} \textrm{ is homeomorphic to a closed disk } \iff 2b-a\leq 3.
$$ 
\end{theorem}

\begin{figure}
\begin{center}
\begin{tabular}{cc}
\includegraphics[width=60mm,height=50mm]{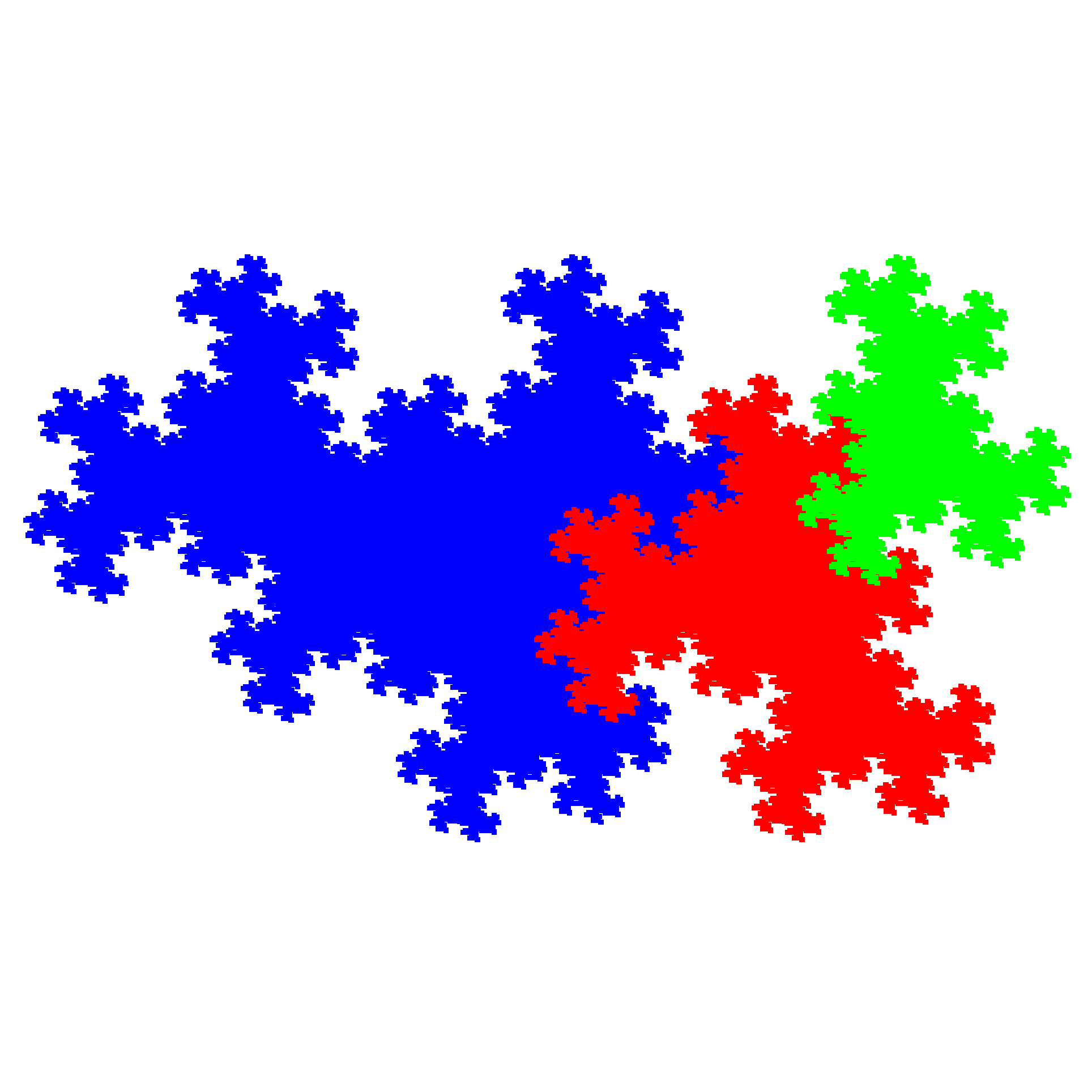}&\includegraphics[width=60mm,height=50mm]{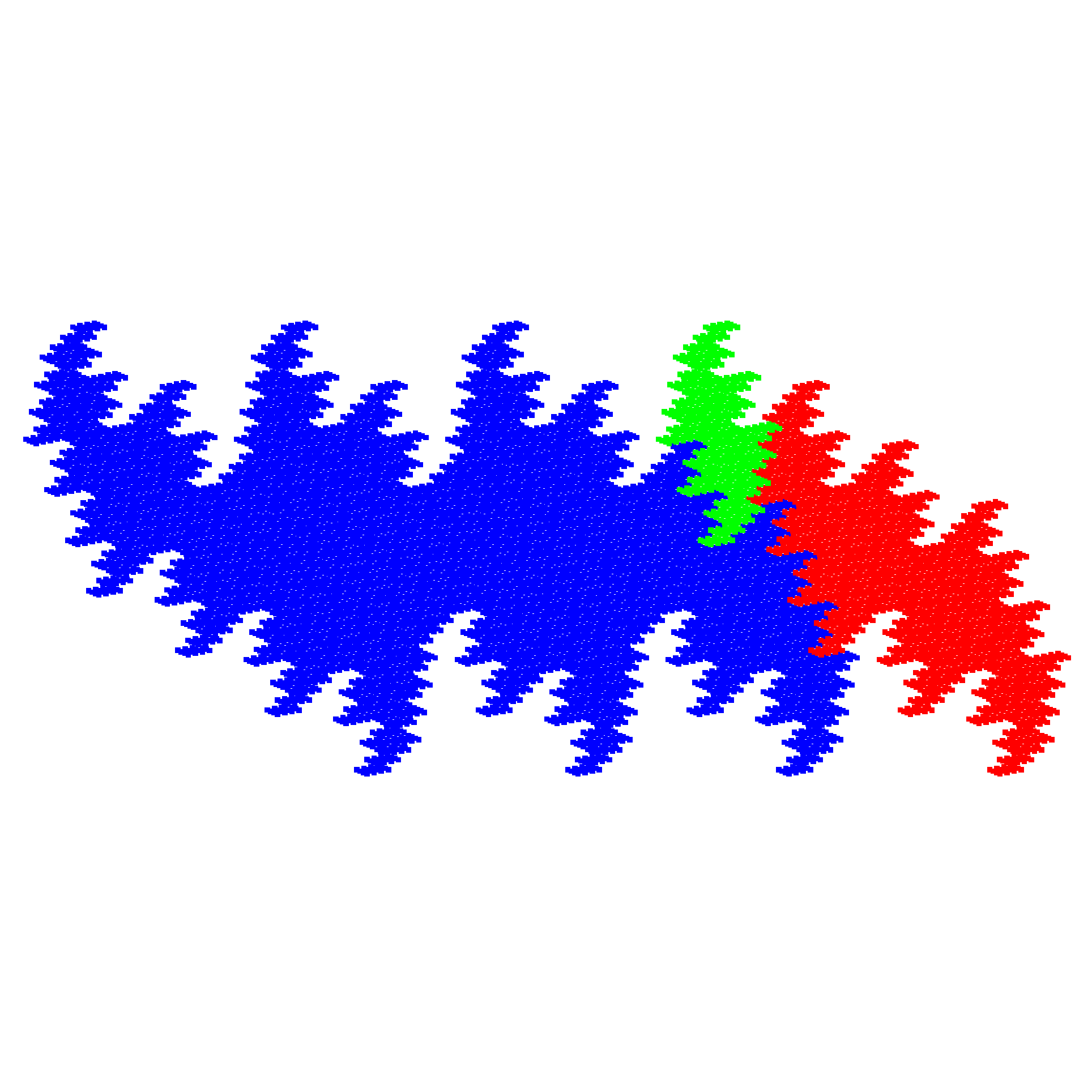}\\
$\sigma_{1,2}$&$\sigma_{3,3}$\\
\includegraphics[width=60mm,height=50mm]{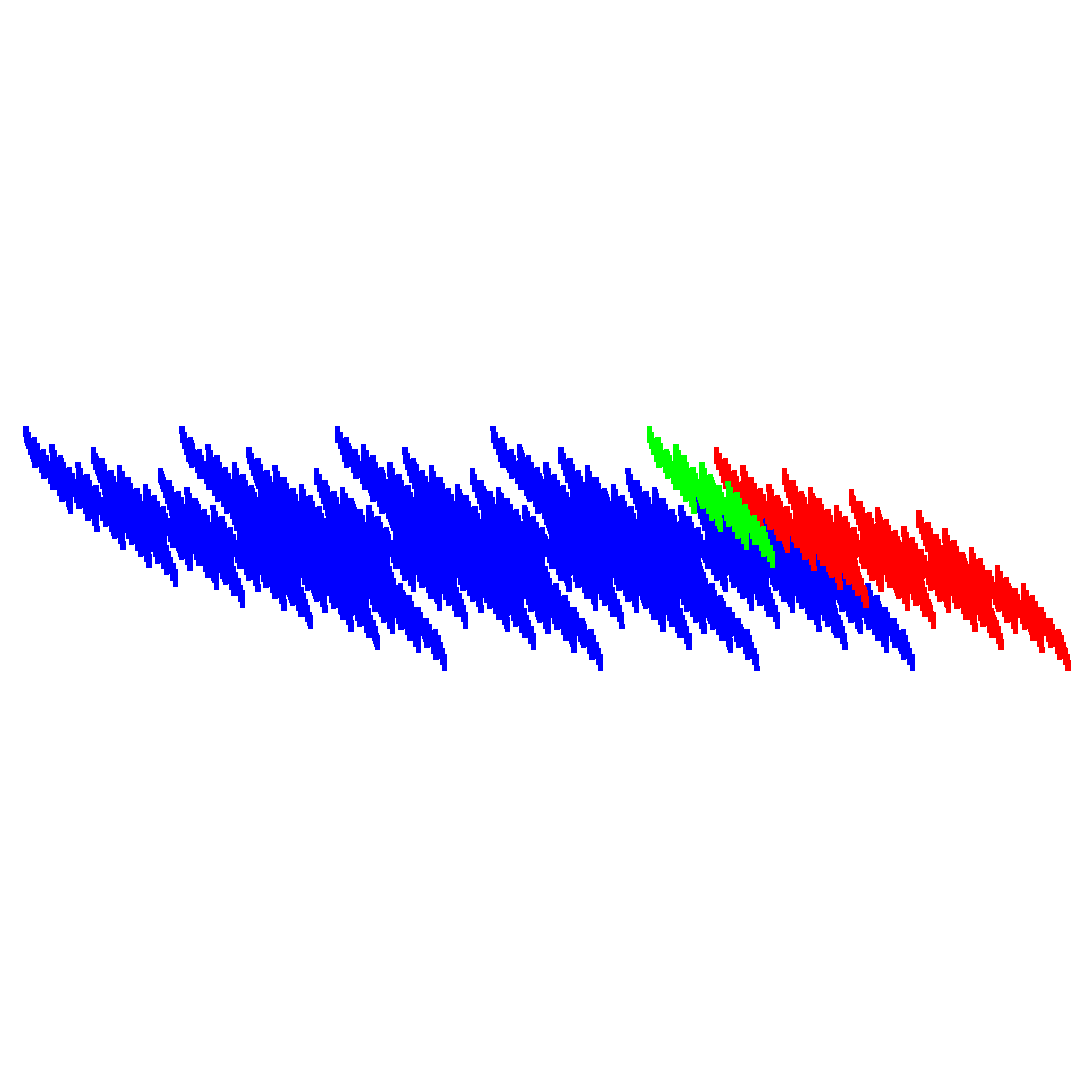}&\includegraphics[width=60mm,height=50mm]{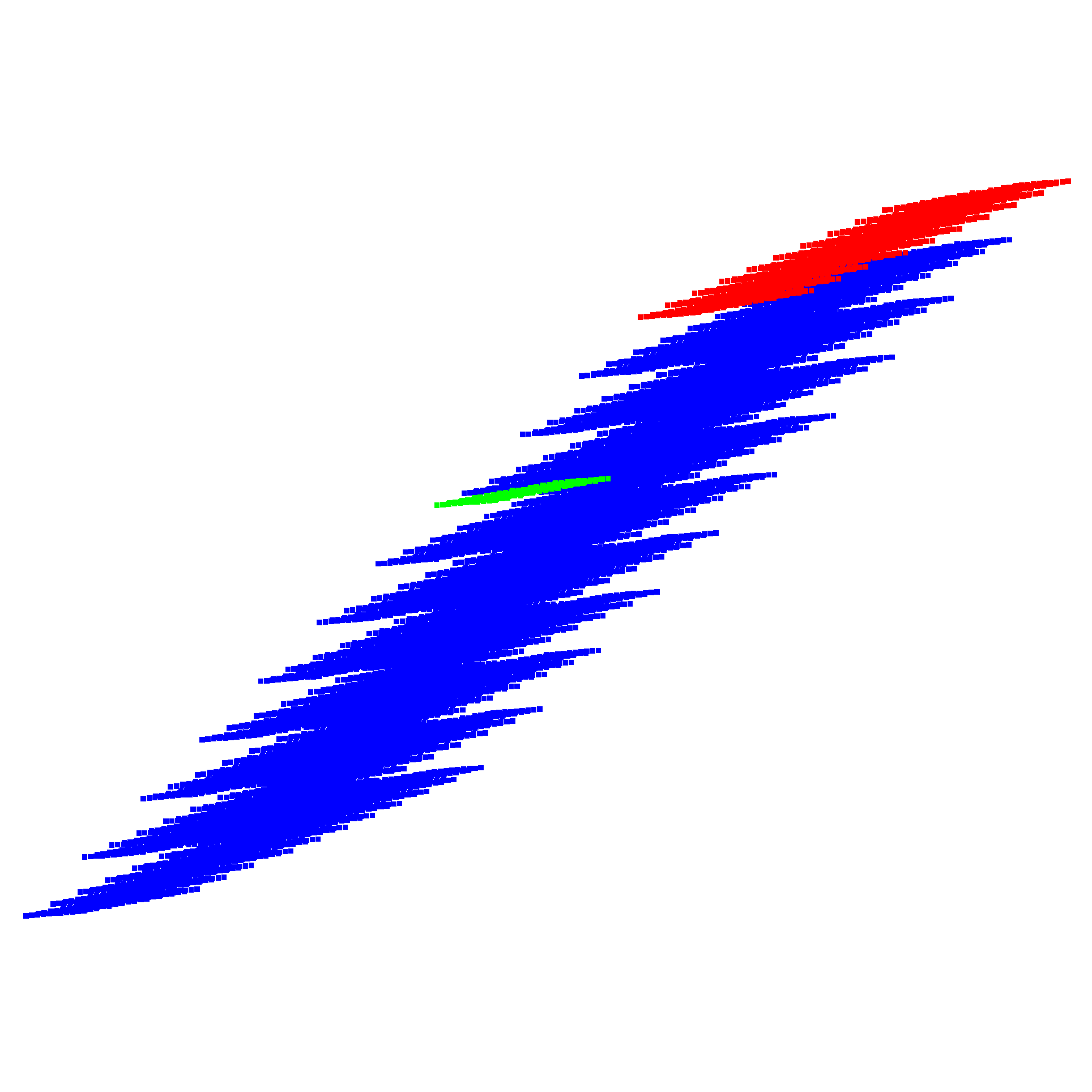}\\
$\sigma_{4,4}$&$\sigma_{7,10}$
\end{tabular}
\caption{Disk-like (above) and non disk-like (below) cubic Rauzy fractals}\label{ExDND}
\end{center}
\end{figure}

Some examples can be seen on Figure~\ref{ExDND}. The cases $a=b=1$ and $a\geq b=1$ were treated in~\cite{Messaoudi00,Messaoudi06}, where it was shown that the Rauzy fractals are quasi-circles. Also, it was proved in~\cite{LoridantMessaoudiSurerThuswaldner13} that $\T_{a,b}$ can not be homeomorphic to a closed disk as soon as $2b-a>3$.  We will recover all these results by another method. Indeed, in order to prove Theorem~\ref{DNDTheo}, we will construct a parametrization of the boundary of $\T$. This parametrization will have the following properties.

\begin{theorem}\label{ParamTheo}
Consider the substitution $\sigma=\sigma_{a,b}$ ($a\geq b\geq 1$) defined in~(\ref{DefSubst}) and let $\T$  be its Rauzy fractal. Let $\lambda$ be the largest root of 
$$x^4+(1-b)x^3+(b-a)x^2-(a+1)x-1.$$ 
Then there exists a surjective H\"older continuous mapping $C:[0,1]\to\partial\mathcal{T}$ with $C(0)=C(1)$ 
and a sequence of polygonal curves $(\Delta_n)_{n\geq 0}$ such that 
\begin{itemize}
\item $\lim_{n\to\infty}\Delta_n=\partial \mathcal{T}$ (Hausdorff metric).
\item Denote by $V_n$ the set of vertices of  $\Delta_n$. Then 
$$V_n\subset V_{n+1}\subset C(\mathbb{Q}(\lambda)\cap[0,1]).$$
\end{itemize}
The H\"older exponent is $\displaystyle s=-\frac{\log |\alpha|}{\log |\lambda|}$, where $|\alpha|=\max\{|\alpha_1|,|\alpha_2|\}$.
\end{theorem}
\begin{remark}
In the case $\alpha_2=\overline{\alpha_1}$, the H\"older exponent is
$$s=\frac{1}{\textrm{dim}_H\partial \mathcal{T}}.
$$
\end{remark}

The construction of the boundary parametrization $C$ in Theorem~\ref{ParamTheo} roughly reads as follows. The tile $\T=\T_{a,b}$ is the attractor of the graph directed construction~(\ref{TileGIFS}). The labels of the infinite walks in the associated prefix-suffix graph $\Gamma=\Gamma_{a,b}$ build up the language of the tile. The boundary $\partial \T$ happens to be also the attractor of a graph directed construction. A finite graph $G$ with a bigger number of states than $\Gamma$ describes the corresponding sublanguage of the language of $\T$.  This graph induces a Dumont-Thomas numeration system~\cite{DumontThomas89}, leading to the parametrization schematically represented below: 
$$\begin{array}{cccc}C:&[0,1]&\longrightarrow&\partial \mathcal{T}\\ &\searrow&&\nearrow\\ &&G&\end{array}$$ 
with $C(0)=C(1)$. To prove Theorem~\ref{DNDTheo}, we will investigate the injectivity of $C$ on $[0,1[$. Indeed, whenever $C$ is injective, $\partial \T$ is a simple closed curve and $\T$ is homeomorphic to a closed disk by a theorem of Sch\"onflies - a strengthened form of Jordan's curve theorem, see~\cite{Whyburn79}. 

\end{section}

\begin{section}{GIFS for the boundary of $\T_{a,b}$}\label{sec:GIFS}
In this section, we introduce two graphs that describe the boundary of the Rauzy fractals $\T=\T_{a,b}$ associated to the substitutions $\sigma=\sigma_{a,b}$. First, we will focus on the boundary graph $\mathcal{G}_{0,a,b}$, that describes the whole language of the boundary of $\T_{a,b}$. Second, we will present a subgraph $G_{0,a,b}\subset\mathcal{G}_{0,a,b}$, whose language is large enough to cover the boundary. The latter graph will be strongly connected (see Lemma~\ref{PosEVContact}), unlike the boundary graph, and this property will allow us to perform the boundary parametrization. Both graphs will be of importance to distinguish the disk-like tiles from the non-disk-like tiles. Roughly speaking, whenever the languages of these graphs are equal, the parametrization is injective and the boundary is a simple closed curve, otherwise the parametrization fails to be injective. For our class of substitutions, slight different versions of these graphs were computed in 2006 \cite{Thuswaldner06} and in 2013 \cite{LoridantMessaoudiSurerThuswaldner13} (see Remarks~\ref{rem:2typesedges} and~\ref{rem:G0vsCont}). A crucial result will be Lemma~\ref{CharacBound}, characterizing the boundary points of the tiles. Indeed, the ``if part'' will be used to prove the continuity of the parametrization $C$ in Theorem~\ref{ParamTheo} for all parameters $a,b$, the ``only if part'' to prove its injectivity whenever $2b-a\leq 3$.

By the tiling property~(\ref{AperTiling}), 
$$\partial \T=\bigcup_{i=1}^3\bigcup_{[\gamma,j]\in\Gamma_{srs},\gamma\ne 0}\T(i)\cap(\T(j)+\gamma).
$$
The subtiles $\T(i)$ satisfy the equations~(\ref{TileGIFS}). This allows to write the boundary $\partial \T$ itself as the attractor of a graph directed function system (\emph{GIFS}).

\subsection{The boundary graph: the boundary language}\label{sec:BoundaryLanguage}

\begin{definition}\label{DefBoundGraph} The \emph{boundary graph} $\mathcal{G}_0=\mathcal{G}_{0,a,b}$ is the largest graph satisfying the following conditions.
\begin{itemize}
\item[$(i)$] A triple $[i,\gamma,j]\in\mathcal{A}\times\pi(\mathbb{Z}^3)\times\mathcal{A}$ is a vertex of $\mathcal{G}_0$ if 
\begin{equation}\label{BoundBound}
||\gamma||\leq 2\frac{\max\{||\pi\lBF( p)||;p\textrm{ label of }\Gamma\}}{1-\max\{|\alpha_1|,|\alpha_2|\}}.
\end{equation}
\item[$(ii)$] There is an edge  $[i,\gamma,j]\xrightarrow{p|p'}[i_1,\gamma_1,j_1]$ iff $i\xrightarrow{p}i_1\in\Gamma$, $j\xrightarrow{p'}j_1\in\Gamma$ and 
$$\mathbf{h}\gamma_1=\gamma+\pi (\lBF(p')-\lBF( p)).
$$
\item[$(iii)$] Each vertex belongs to an infinite walk starting from a vertex $[i,\gamma,j]$ with $[\gamma,j]\in\Gamma_{srs}$ and ($\gamma\ne 0$ or $i<j$).
\end{itemize}
The set of vertices of $\mathcal{G}_0$ is denoted by $\mathcal{S}_0$.
\end{definition}
An analogous definition can be found in~\cite[Definition 5.4]{SiegelThuswaldner10}. Note that~(\ref{BoundBound}) is an upper bound for the diameter of $\T$. 

For a given substitution, the computation of $\mathcal{G}_0$ is algorithmic. There are finitely many triples satisfying~(\ref{BoundBound}). $\mathcal{G}_0$  is obtained after checking the algebraic relation of $(ii)$ between all pairs of triples and erasing the vertices that do not fulfill $(iii)$. See also~\cite{SiegelThuswaldner10}. 
\begin{ex}$\mathcal{G}_0$ is depicted on Figure~\ref{fig:BoundGraphTribo} for $a=b=1$. 
See Table~\ref{SRBG1} for the vertices associated to the letters in this graph. Here, if $S=[i,\gamma,j]$, then $S^-:=[j,-\gamma,i]$. The colored states stand for triples $[i,\gamma,j]$ with $[\gamma,j]\in\Gamma_{srs}$. The labels just indicate the number of $1$'s in $p_1,p_2$ ($0$ for the prefix $\epsilon$).
\end{ex}

\begin{figure}
\begin{center}
\includegraphics[width=150mm,height=141mm]{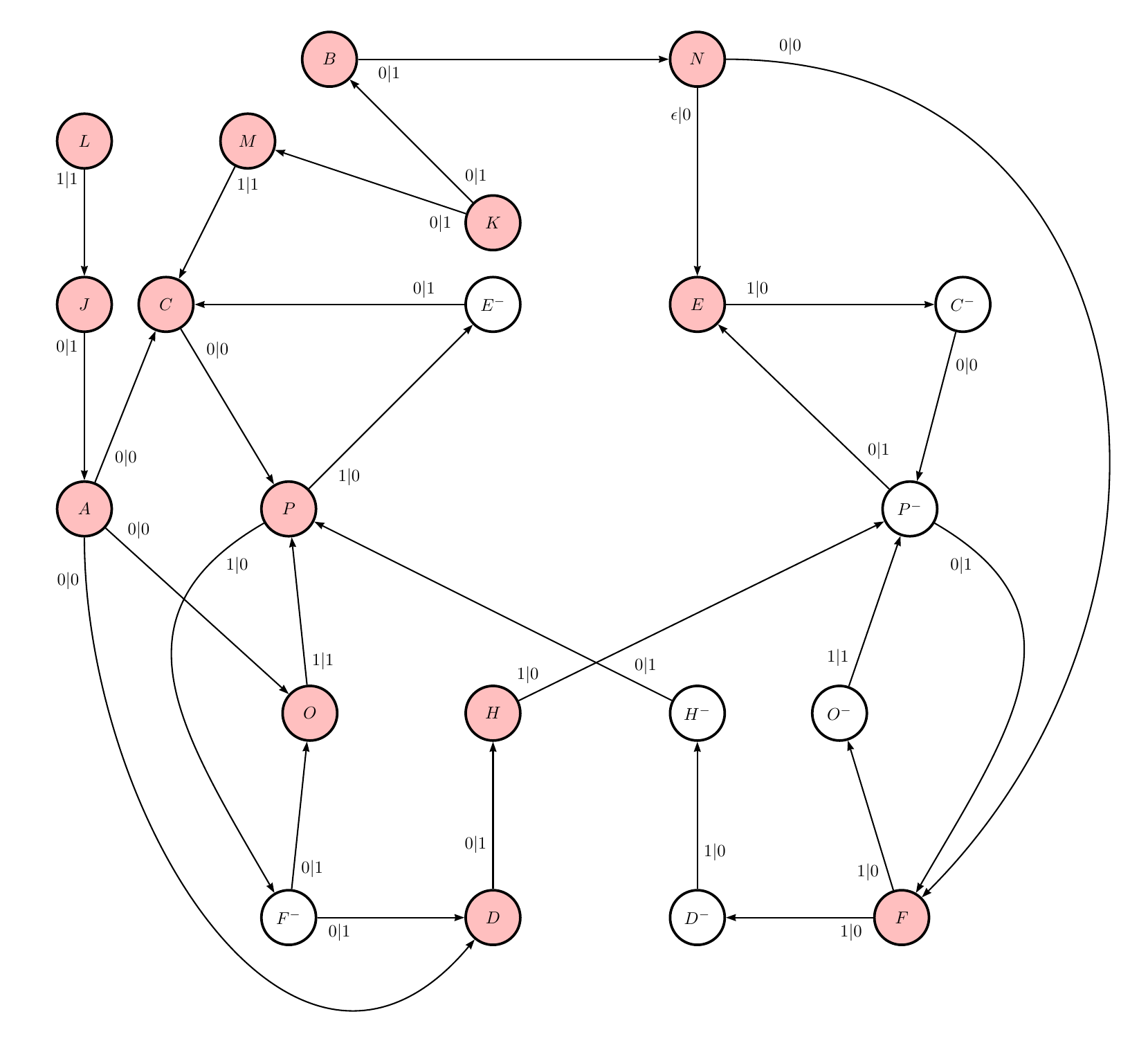}\end{center}
\caption{Boundary graph of the Tribonacci substitution ($a=b=1$)}\label{fig:BoundGraphTribo}
\end{figure}

Boundary points are characterized as follows.
\begin{lemma}\label{CharacBound} Let $(p_k)_{k\geq 0}$ and $(p_k')_{k\geq0}$ be the labels of infinite walks in the prefix-suffix graph $\Gamma$ starting from $i\in\mathcal{A}$ and $j\in\mathcal{A}$ respectively. Let $\gamma\in\pi(\mathbb{Z}^3)$ such that $[\gamma,j]\in\Gamma_{srs}$ and $(\gamma\ne 0 \textrm{ or }i<j)$. Then 
$$\sum_{k\geq 0} \mathbf{h}^k\pi \lBF(p_k)=\gamma+\sum_{k\geq 0} \mathbf{h}^k\pi \lBF(p_k')=:x$$ if and only if there is an infinite walk
$$[i,\gamma,j]\xrightarrow{p_0|p_0'}[i_1,\gamma_1,j_1]\xrightarrow{p_1|p_1'}\ldots\in\mathcal{G}_0.
$$
In this case, $x\in \mathcal{T}(i)\cap (\mathcal{T}(j)+\gamma)$.
\end{lemma}
\begin{proof} We mainly use arguments of~\cite[Proof of Theorem 5.6]{SiegelThuswaldner10}. If the above infinite walk exists in $\mathcal{G}_0$, then using the definition of the edges one can write for all $n\geq 0$:
$$\hBF^{n+1}\gamma_{n+1}+\sum_{k=0}^n\pi\lBF(p_k)=\gamma+\sum_{k=0}^n\pi\lBF(p_k').
$$  
As $\hBF$ is contracting and $(\gamma_{n})_{n\geq 0}$ is a bounded sequence, letting $n\to\infty$ gives the required equality.

We now construct the walk by assuming the equality of the two infinite expansions. Note that $\gamma$ satisfies~(\ref{BoundBound}), and by assumption there exist edges $i\xrightarrow{p_0}i_1$ and $j\xrightarrow{p_0'}j_1$ in $\Gamma$. Let 
$$\gamma_1=\sum_{k=0}^\infty\pi\lBF(p_{k+1})-\sum_{k=0}^\infty\pi\lBF(p_{k+1}')=\hBF^{-1}(\gamma+\pi\lBF(p_{0}')-\pi\lBF(p_{0})).
$$
Then again $\gamma_1$ satisfies~(\ref{BoundBound}) and $\hBF\gamma_1=\gamma+\pi(\lBF(p_{0}')-\lBF(p_{0}))$. Moreover, choosing $x\in\mathbb{Z}^3$ satisfying $\pi(x)=\gamma$, we can define 
$$x_1=\M^{-1}(x+\lBF(p_{0}')-\lBF(p_{0}))\in\mathbb{Z}^3,
$$
that is, $\gamma_1\in\pi(\mathbb{Z}^3)$. Therefore, the edge $[i,\gamma,j]\xrightarrow{p_0|p_0'}[i_1,\gamma_1,j_1]$ fulfills $(ii)$ of Definition~\ref{DefBoundGraph}. The infinite sequence of edges $[i,\gamma,j]\xrightarrow{p_0|p_0'}[i_1,\gamma_1,j_1]\xrightarrow{p_1|p_1'}\ldots$ satisfying $(i)$ and $(ii)$ of Definition~\ref{DefBoundGraph} is constructed iteratively in the above way. It satisfies also $(iii)$, since $[\gamma,j]\in \Gamma_{srs}$ and $(\gamma\ne 0 \textrm{ or }i<j)$. Therefore, it is an infinite walk in $\mathcal{G}_0$.
\end{proof}

\begin{lemma}Let $[i,\gamma,j]\in\mathcal{S}_0$. Then 
either $[\gamma,j]$ or $[-\gamma,i]$ belongs to $ \Gamma_{srs} $.
\end{lemma}
\begin{proof} Note that $[0,i]\in\Gamma_{srs}$ for all $i\in\mathcal{A}$. By definition, a vertex of $\mathcal{G}_0$ belongs to an infinite walk starting from a vertex $[i,\gamma,j]$ with $[\gamma,j]\in\Gamma_{srs}$.  Thus we assume that a given vertex $[i,\gamma,j]$ of $\mathcal{G}_0$ satisfies $[\gamma,j]\in\Gamma_{srs}$ or $[-\gamma,i]\in\Gamma_{srs}$,  and check that as soon as there is an edge $[i,\gamma,j]\xrightarrow{p|p'}[i_1,\gamma_1,j_1] $ in $\mathcal{G}_0$, then either $[\gamma_1,j_1]$ or $[-\gamma_1,i_1]$ belongs to $\Gamma_{srs}$. Indeed, let $x\in\mathbb{Z}^3$ such that $\pi(x)=\gamma$. Then the existence of such an edge insures that 
$$\gamma_1=\pi(x_1)=\pi(\M^{-1}(x+\lBF(p')-\lBF( p)))
$$
for some $x_1\in\mathbb{Z}^3$. Therefore,
$$\langle x_1,\mathbf{v}_{\beta}\rangle=\langle\M^{-1}(x+\lBF(p')-\lBF( p)),\mathbf{v}_{\beta}\rangle=\frac{1}{\beta}\langle x+\lBF(p')-\lBF( p),\mathbf{v}_{\beta}\rangle.
$$
If $[\gamma,j]\in\Gamma_{srs}$, then 
$0\leq\langle\mathbf{x},\mathbf{v}_{\beta}\rangle< \langle\mathbf{e}_j,\mathbf{v}_{\beta}\rangle$ implies that
$$-\beta^{-1}\langle \lBF( p),\mathbf{v}_{\beta}\rangle
\leq\langle x_1,\mathbf{v}_{\beta}\rangle
< \beta^{-1}\langle \mathbf{e}_j+\lBF(p'),\mathbf{v}_{\beta}\rangle.
$$
Using the fact that $\sigma(i_1)=pis$ and $\sigma(j_1)=p'js'$ for some $s,s'\in\mathcal{A}^*$, we obtain
$$-\langle \mathbf{e}_{i_1},\mathbf{v}_{\beta}\rangle
<\langle x_1,\mathbf{v}_{\beta}\rangle
<\langle \mathbf{e}_{j_1},\mathbf{v}_{\beta}\rangle,
$$
hence $[\gamma_1,j_1]$ or $[-\gamma_1,i_1]$ belongs to $\Gamma_{srs}$. A similar computation holds if $[-\gamma,i]\in\Gamma_{srs}$. See also~\cite[Proof of Theorem 5.6]{SiegelThuswaldner10}.
\end{proof}

\begin{remark}\label{rem:2typesedges} In~\cite{SiegelThuswaldner10,Thuswaldner06}, all the vertices $[i,\gamma,j]$ of the boundary graph satisfy $[\gamma,j]\in\Gamma_{srs}$, but two types of edges are used. In the present article, we do not introduce two types of edges. In this way, the labels of infinite walks in $\mathcal{G}_0$ are sequences of prefixes that also occur as labels of infinite walks in the prefix-suffix graph. In other words, the language of the boundary of $\T$ is directly visualized as a sublanguage of $\T$. This will be important for the proof of our main results, that requires to find out the infinite sequences of prefixes $(p_k)_{k\geq 0},(p_k')_{k\geq 0}$ satisfying $\sum_{k\geq 0} \mathbf{h}^k\pi \lBF(p_k)=\sum_{k\geq 0} \mathbf{h}^k\pi \lBF(p_k')$. We explain in the core of the proof of Proposition~\ref{prop:ContactGIFS} how to get rid off the two types of edges from the boundary graphs of~\cite{SiegelThuswaldner10,Thuswaldner06} in order to derive our boundary graph $\mathcal{G}_0$.
 \end{remark}

We call
$$\mathcal{S}=\{[i,\gamma,j]\in\mathcal{S}_0;\gamma\ne 0,[\gamma,j]\in\Gamma_{srs}\}$$
 the \emph{set of neighbors} of $\T$ in the tiling (\ref{AperTiling}).

This gives us the first boundary GIFS. 
\begin{proposition}\label{GSGIFS} Let $B[i,\gamma,j]$ the non-empty compact sets solutions of the GIFS 
\begin{equation}\label{NeighborGIFS}
\forall [i,\gamma,j]\in\mathcal{S}_0, \;B[i,\gamma,j]=\bigcup_{[i,\gamma,j]\xrightarrow{p|p'}[i_1,\gamma_1,j_1]\in\mathcal{G}_0}\hBF B[i_1,\gamma_1,j_1]+\pi\lBF( p).
\end{equation}
Then $B[i,\gamma,j]=\T(i)\cap(\T(j)+\gamma)$ and 
$\partial \T=\bigcup_{[i,\gamma,j]\in\mathcal{S}}B[i,\gamma,j].
$
\end{proposition}
\begin{proof} The proof follows~\cite[Proof of Theorem 5.7]{SiegelThuswaldner10}. The set 
$$\left\{x\mapsto \displaystyle\hBF x+\pi\lBF( p)\right\}_{[i,\gamma,j]\xrightarrow{p|p'}[i_1,\gamma_,j_1]\in\mathcal{G}_0}$$ 
is a graph iterated function system, since $\hBF$ is a contraction. By a result of Mauldin and Williams~\cite{MauldinWilliams88}, there is a unique sequence of non-empty compact sets $\left(B[i,\gamma,j]\right)_{[i,\gamma,j]\in\mathcal{S}_0}$ which is the attractor of this GIFS. 

We now show that the sequence of sets $\left(\T(i)\cap(\T(j)+\gamma)\right)_{[i,\gamma,j]\in\mathcal{S}_0}$ also satisfies the set equations of the above GIFS and then use the unicity of the attractor.

Let $[i,\gamma,j]$ be a vertex of $\mathcal{G}_0$. Using (\ref{TileGIFS}), we can subdivide each intersection of tiles as follows:
\begin{equation}\label{SubdBoundPart}
\begin{array}{c}\T(i)\cap (\T(j)+\gamma)\\\\
=\bigcup_{i\xrightarrow{p}i_1\in\Gamma,j\xrightarrow{p'}j_1\in\Gamma}\pi\lBF( p)+\hBF\left[\T(i_1)\cap\left(\T(j_1)+\underbrace{\hBF^{-1}\pi(\lBF(p')-\lBF( p)+\hBF^{-1}\gamma }_{\displaystyle=:\gamma_1}\right)\right].
\end{array}
\end{equation}
Let $[i_1,\gamma_1,j_1]$ be as in the above union. If it is a vertex of $\mathcal{G}_0$, then by a similar computation as in the first part of the proof of Lemma~\ref{CharacBound}, one obtains a point in $\T(i_1)\cap(\T(j_1)+\gamma_1)$, thus this intersection is non-empty.

On the contrary, suppose $\T(i_1)\cap(\T(j_1)+\gamma_1)\ne\emptyset$. We wish to show that  $[i,\gamma,j]\xrightarrow{p|p'}[i_1,\gamma_1,j_1]\in\mathcal{G}_0$. First, since $[i,\gamma,j]$ is a vertex of $\mathcal{G}_0$, we can write $\gamma=\pi(x)$ for some $x\in\mathbb{Z}^3$ and $\gamma_1=\pi(\M^{-1}(x+\lBF(p')-\lBF( p)))\in\pi(\mathbb{Z}^3)$. Also, since $\T(i_1)\cap(\T(j_1)+\gamma_1)\ne\emptyset$, there are $(p_k)_{k\geq 0}$ and $(p_k')_{k\geq0}$ labels of infinite walks of $\Gamma$ starting from $i_1$ and $j_1$ respectively such that  
$$\sum_{k\geq 0} \mathbf{h}^k\pi \lBF(p_k)=\gamma+\sum_{k\geq 0} \mathbf{h}^k\pi \lBF(p_k').
$$
Consequently, $\gamma_1$ is bounded as in~(\ref{BoundBound}).  Hence the edge $[i,\gamma,j]\xrightarrow{p|p'}[i_1,\gamma_1,j_1]$ satisfies $(i)$, as well as $(ii)$ of Definition~\ref{DefBoundGraph}. Moreover, from the above equality of expansions, one can construct as in the proof of Lemma~\ref{CharacBound} an infinite sequence of edges
starting from $[i_1,\gamma_1,j_1]$ and satisfying $(i)$ and $(ii)$ of Definition~\ref{DefBoundGraph}. Lastly, by assumption on $[i,\gamma,j]$, one can find a walk $[i_0,\gamma_0,j_0]\xrightarrow{q_0|q_0'}\cdots\xrightarrow{q_{l}|q_{l}'}[i,\gamma,j]$ in $\mathcal{G}_0$ with $[\gamma_0,j_0]\in\Gamma_{srs}$ and $(\gamma_0\ne0\textrm{ or }i<j)$. Altogether,  we have found an infinite sequence of edges satisfying $(i)$ and $(ii)$  and including the edge $[i,\gamma,j]\xrightarrow{p|p'}[i_1,\gamma_1,j_1]$. Therefore, $[i_1,\gamma_1,j_1]$ fulfills $(iii)$ of Definition~\ref{DefBoundGraph} and $[i,\gamma,j]\xrightarrow{p|p'}[i_1,\gamma_1,j_1]$ belongs to $\mathcal{G}_0$.

It follows that (\ref{SubdBoundPart}) can be re-written as
\begin{equation}\label{SubdBoundPart2}\begin{array}{c}\T(i)\cap (\T(j)+\gamma)\\\\
=\bigcup_{[i,\gamma,j]\xrightarrow{p|p'}[i_1,\gamma_1,j_1]\in\mathcal{G}_0}\pi\lBF( p)+\hBF\left[\T(i_1)\cap\left(\T(j_1)+\underbrace{\hBF^{-1}\pi(\lBF(p')-\lBF( p)+\hBF^{-1}\gamma }_{=:\gamma_1}\right)\right].
\end{array}
\end{equation}
By unicity of the GIFS-attractor, we conclude that $B[i,\gamma,j]=\T(i)\cap(\T(j)+\gamma)$ for all $[i,\gamma,j]\in\mathcal{S}_0$. 

The second equality is a consequence of the tiling property and the definition of $\mathcal{S}$:
$$\partial \T=\bigcup_{i=1}^3\bigcup_{[\gamma,j]\in\Gamma_{srs},\gamma\ne 0}\T(i)\cap(\T(j)+\gamma)=\bigcup_{[i,\gamma,j]\in\mathcal{S}}\T(i)\cap(\T(j)+\gamma).
$$
\end{proof}

Therefore, $\partial T$ is the attractor of a graph directed self-affine system. To proceed to the boundary parametrization, the natural idea would be to order the vertices and edges of the graph and use the induced Dumont-Thomas numeration system~\cite{DumontThomas89}. Geometrically, this corresponds to an ordering of the boundary parts and their subdivisions clockwise or counterclockwise along the boundary. This method requires the strongly connectedness of the graph, or at least the existence of a positive dominant eigenvector for its incidence matrix. However, in general, the above boundary graph does not have this property. Roughly speaking, there may be many redundances in the boundary language given by the boundary graph: the mapping
$$[i,\gamma,j]\xrightarrow{p_0|p_0'}[i_1,\gamma_1,j_1]\xrightarrow{p_1|p_1'}\ldots\in\mathcal{G}_0\;\;\mapsto\;\;\sum_{k\geq 0} \mathbf{h}^k\pi {\bf l}(p_k)\in\partial\T
$$
sending an infinite walk in the boundary graph to a boundary point may be highly not injective. The level of non-injectivity reflects the complexity of the topology of $\T$. For example, many neighbors (that is, many states in the automaton) suggest an intricate topological structure. 

In fact, if an intersection $\T(i)\cap(\T(j)+\gamma)$ is a point, or has a Hausdorff dimension smaller than that of the boundary,  it shall be redundant (contained in other intersections), thus not essential. In the next subsection, we introduce a subgraph of the boundary graph that will be more appropriate.

\subsection{The graph $G_0$}
In 2006, J\"org Thuswaldner defined a graph which is in general smaller than the boundary graph but always contains enough information to describe the whole boundary~\cite{Thuswaldner06}. As an example, he computed this graph for our class of substitutions.

\begin{definition}\label{def:G0} Let $a\geq b\geq 1$. Let $G_0=G_{0,a,b}$ be the graph with 
\begin{itemize}
\item Vertices: 
$$\begin{array}{c}
R_0=R_{0,a,b}\\\\=\{A,B,C,C^-,D,D^-,E,E^-,F,F^-,G,G^-,H,H^-,I,I^-,J,K,N,N^-,O,O^-,P,P^-\}\\\\
\cup\;\{M\}\setminus\{I,I^-\},\textrm{ if }a\geq 2,b=1\\\\
\cup\;\{L\}\setminus\{G,G^-,N^-\},\textrm{ if }a= b\geq 2\\\\
\cup\;\{L,M\}\setminus\{G,G^-,I,I^-,N^-\},\textrm{ if }a=b=1
\end{array}
$$
 as in Table~\ref{SRBG1}. Here, if $S=[i,\gamma,j]$, then $S^-:=[j,-\gamma,i]$.
\item Edges\footnote{\label{note1}As the prefixes $p_1,p_2$ belong to $\{\epsilon, 1,11,\ldots,\underbrace{11\cdots1}_{a}\}$, the labels just indicate the number of $1$'s in $p_1,p_2$.}: in addition to the edges of Table~\ref{SRBG1}, we have
$$S^-\xrightarrow{p_2|p_1}T^-\in G_0 \iff S\xrightarrow{p_1|p_2}T\in G_0,
$$ 
and 
$$S^-\xrightarrow{p_2|p_1}T\in G_0 \iff S\xrightarrow{p_1|p_2}T^-\in G_0
$$ 
(as long as $S^-,T^-$ belong to $R_0$ defined above).
\end{itemize}
\end{definition}

\begin{table}[ht]\scriptsize\begin{center}
\begin{tabular}{|c|c|c|c|c|c|c|}
\hline
\multicolumn{3}{|c}{Vertex} & \multicolumn{3}{|c|}{Edge(s)} \\
\hline
 \# & Name & Condition& to & Label $p_1|p_2$& Condition\\
\hline
\multirow{4}{*}{$A$} & \multirow{4}{*}{$[1,\pi(0,0,1),1]$} & & $C$ & $k|b-1+k,\;0 \leq k \leq a-b$  &  \\
& & & $D$ & $0|b-1$  & \\
 & & & $O$ & $0|b-1$  & \\
 & & & $N$ & $k|b+k,\; 0 \leq k \leq a-b-1$  & $a \not=b$ \\
\hline
\multirow{2}{*}{$B$} & \multirow{2}{*}{$[1,\pi(0,0,1),2]$} & & $N$ & $a-b|a$  & \\
 & & & $C$ & $a-b+1|a$  & $b \geq 2$ \\
\hline
\multirow{3}{*}{$C$} & \multirow{3}{*}{$[1,\pi(0,1,-1),1]$} & & $P$ & $ k|a-b+k,\;0 \leq k \leq b-1$  &  \\
 & & & $H$ & $k|a-b+1+k,\;0 \leq k \leq b-2$ & $b \geq 2$ \\
 & & & $I$ & $k|a-b+1+k,\; 0 \leq k \leq b-2$  & $b \geq 2$ \\
\hline
\multirow{2}{*}{$D$} & \multirow{2}{*}{$[1,\pi(0,1,-1),2]$} & & $H$ & $b-1|a$  & \\
 & & & $I$ & $b-1|a$  & $b \geq 2$\\
\hline
\multirow{2}{*}{$E$} & \multirow{2}{*}{$[2,\pi(1,0,-1),1]$} & & $C^-$ & $a|a-b$  & \\
 & & & $N^-$ & $a|a-b-1$ & $a \not= b$ \\
\hline
 \multirow{2}{*}{$F$} & \multirow{2}{*}{$[3,\pi(1,0,-1),1]$} & & $D^-$ & $b|0$  & \\
 & & & $O^-$ & $b|0$  & \\
\hline
\multirow{2}{*}{$G$} & \multirow{2}{*}{$[1,\pi(1,0,-1),1]$} & $a \not= b$ & $C^-$ & $a-1-k|a-b-1-k,\;  0\leq k  \leq a-b-1$  & $a \not=b$ \\
 & &  & $N^-$ & $a-1-k|a-b-2-k,\;  0\leq k  \leq a-b-2$  & $a \geq b+2$ \\
\hline
\multirow{3}{*}{$H$} & \multirow{3}{*}{$[2,\pi(1,-1,1),1]$} & & $P^-$ & $a|b-1$  & \\
& & & $H^-$ & $a|b-2$  & $b \geq 2$ \\
& & & $I^-$ & $a|b-2$ & $b \geq 2$ \\
\hline
\multirow{3}{*}{$I$} & \multirow{3}{*}{$[1,\pi(1,-1,1),1]$} & \multirow{3}{*}{$b \geq 2$} & $P^-$ & $a-1-k|b-2-k,\;0 \leq k\leq b-2$  & $b \geq 2$ \\
 & & & $H^-$ & $a-1-k|b-3-k,\;0 \leq k\leq b-3$ & $b \geq 3$ \\
 & & & $I^-$ & $a-1-k|b-3-k,\;0 \leq k\leq b-3$  & $b \geq 3$ \\
\hline
$J$ & $[1,\pi(0,0,0),2]$ & & $A$ & $a-1|a$  & \\
\hline
\multirow{3}{*}{$K$} & \multirow{3}{*}{$[1,\pi(0,0,0),3]$} & & $B$ & $b-1|b$  & \\
 & & & $J$ & $b|b$ & $a \not= b$\\
 & & & $M$ & $b-1|b$& $b=1$\\
\hline
$L$ & $[2,\pi(0,0,0),3]$ & $a=b$ & $J$ & $a|a$  & $a=b$ \\
\hline
$M$ & $[2,\pi(0,0,1),2]$ & $b=1$ & $C$ & $a|a$  & $b=1$ \\
\hline
\multirow{3}{*}{$N$} & \multirow{3}{*}{$[1,\pi(0,1,0),1]$} & & $E$ & $0|a-1$  & \\
 & & & $F$ & $0|a-1$  & \\
 & & & $G$ & $0|a-1$  & $a \not=b$ \\
\hline
$O$ & $[3,\pi(0,1,-1),2]$ & & $P$ & $b|a$ &  \\\hline
\multirow{3}{*}{$P$} & \multirow{3}{*}{$[2,\pi(1,-1,0),1]$} & & $E^-$ & $a|0$  & \\
 & & & $F^-$ & $a|0$  & \\
 & & & $G^-$ & $a|0$  & $a \not=b$  \\
\hline 
\end{tabular}
\caption{The subgraph $G_0$ of the self-replicating boundary graph.}\label{SRBG1}
\end{center}
\end{table}

\begin{remark} The states $A,B,C,D,\ldots,P$ correspond to the vertices $[i,\gamma,j]$ with $[\gamma,j]\in\Gamma_{srs}$. 
\end{remark}
One can check that $G_0$ satisfies the conditions $(i),(ii)$ and $(iii)$ of the definition of the boundary graph (Definition~\ref{DefBoundGraph}). Therefore, the following lemma holds.
\begin{lemma}\label{G0subsetBoundGraph} For all $a\geq b\geq 1$,
$$G_{0,a,b}\subset\mathcal{G}_{0,a,b}.
$$
\end{lemma}
\begin{remark}\label{rem:G0vsCont}The graph $G_0$ is related to the \emph{contact graph} defined in~\cite{Thuswaldner06} or~\cite{SiegelThuswaldner10}. This notion of contact graph was first introduced by Gr\"ochenig and Haas~\cite{GroechenigHaas94} in the context of self-affine tiles (see also~\cite{ScheicherThuswaldner01}). For substitution tiles, the contact graph is obtained from a sequence of polygonal approximations of the Rauzy fractal constructed via the dual substitutions on the stepped surface (see~\cite{ArnouxIto01}). Each approximation gives rise to a polygonal tiling of the stepped surface. In these tilings, the structure of the \emph{adjacent neighbors} (neighbors whose intersection with the approximating central tile has non-zero 1-dimensional Lebesgue measure) stabilizes after finitely many steps. The collection of adjacent neighbors of a good enough polygonal approximation of the Rauzy fractal results in the set $R_0$ of Definition~\ref{def:G0}. 
\end{remark}

\begin{proposition}[{\cite[Theorem 4.3]{Thuswaldner06}}]
\label{prop:ContactGIFS}
Let $a\geq b\geq 1$ and $\sigma=\sigma_{a,b}$ the substitution defined in~\ref{DefSubst}. Consider the graph $G_0=G_{0,a,b}$ of Definition~\ref{def:G0}. We denote by $R_{a,b}=R\subset R_0$ the set
$$\begin{array}{c}
R=\{A,B,C,D,E,F,G,H,I,N,O,P\}\\\\
\cup\;\{M\}\setminus\{I\},\textrm{ if }a\geq 2,b=1\\\\
\setminus\{G\},\textrm{ if }a= b\geq 2\\\\
\cup\;\{M\}\setminus\{G,I\},\textrm{ if }a=b=1.
\end{array}
$$ 
Then
$$\begin{array}{c}
\partial \T=\bigcup_{[i,\gamma,j]\in R} C[i,\gamma,j],
\end{array}
$$
where the sets $C[i,\gamma,j]$ $([i,\gamma,j]\in R_0)$ are the solutions of the GIFS directed by $G_0$, \emph{i.e.},
\begin{equation}\label{ContactGIFS}
\begin{array}{rcl}
\forall [i,\gamma,j]\in R_0, \;C[i,\gamma,j]&=&\bigcup_{[i,\gamma,j]\xrightarrow{p|p'}[i_1,\gamma_1,j_1]\in G_0}\hBF C[i_1,\gamma_1,j_1]+\pi\lBF( p)\\\\
&\subset&\T( i )\cap(\T ( j )+\gamma)
\end{array}
\end{equation}   

\end{proposition}

\begin{proof} The last inclusion is an easy consequence of Lemma~\ref{G0subsetBoundGraph} and Proposition~\ref{GSGIFS}.

The lengthy proof is given in~\cite[Section 6]{Thuswaldner06}. However, in that article, two types of edges are used and the Rauzy fractals are defined in terms of suffixes instead of prefixes. We refer to Remark~\ref{rem:2typesedges} and to~\cite[Section 4.3]{Thuswaldner06} as well as~\cite{ArnouxIto01}. The correspondence with our setting reads as follows.

 Let $\mathcal{C}=\mathcal{C}_{a,b}$ be the graph as in~\cite[Theorem 6.2]{Thuswaldner06}, depicted in Figures 9 and 10 within this reference, and $\mathcal{C}_\partial$ the subgraph obtained from $\mathcal{C}$ after successively deleting the states having no outgoing edges (as in \cite[Definition 4.5]{Thuswaldner06}). For a state $S=[(0,0,0),i],[\gamma,j]$ occurring in \cite[Figures 9-10]{Thuswaldner06}, we shall simply write $S=[i,\gamma,j]$.
\begin{itemize}
\item[\emph{Step 1}.] The aim is to remove the two types of edges. By~\cite[Definition 3.6]{Thuswaldner06}, an edge 
$$
[i,\gamma,j]\xrightarrow{(p_1,i,s_1)|(p_2,j,s_2)}[i',\gamma',j']\in\mathcal{C}_\partial
$$
is 
\begin{itemize}
\item of type 1 if 
\begin{equation}\label{def:type1}
\sigma(i')=p_1 is_1,\sigma(j')=p_2js_2\textrm{ and }\mathbf{h}\gamma'=\gamma+\pi\mathbf{l}(s_2)-\pi\mathbf{l}(s_1)
\end{equation}
\item of type 2 if
\begin{equation}\label{def:type2}
\sigma(j')=p_1 is_1,\sigma(i')=p_2js_2,\textrm{ and }-\mathbf{h}\gamma'=\gamma+\pi\mathbf{l}(s_2)-\pi\mathbf{l}(s_1).
\end{equation}
\end{itemize}
Replace each edge 
$$S\xrightarrow{(p_1,i,s_1)|(p_2,j,s_2)}T\in\mathcal{C}_\partial$$ 
of type 1 by two edges
$$S\xrightarrow{(p_1,i,s_1)|(p_2,j,s_2)}T\;\;\textrm{ and }\;\;S^-\xrightarrow{(p_2,j,s_2)|(p_1,i,s_1)}T^-,
$$ 
and each edge 
$$S\xrightarrow{(p_1,i,s_1)|(p_2,j,s_2)}T\in\mathcal{C}_\partial$$ 
of type 2 by two edges
$$S\xrightarrow{(p_1,i,s_1)|(p_2,j,s_2)}T^-\;\;\textrm{ and }\;\;S^-\xrightarrow{(p_2,j,s_2)|(p_1,i,s_1)}T.
$$ 
Here, for $S=[i,\gamma,j]$ state of $\mathcal{C}_\partial$, we wrote $S^-:= [j,-\gamma,i]$. See also~\cite[Section 7, Proof of Theorem 5.6]{SiegelThuswaldner10}. This procedure results in a graph whose number of states has doubled. Delete successively the states $S^-$ having no incoming edges. We denote by $\mathcal{C}_\partial^1$ the remaining graph. Note that all edges in this graph now satisfy the relation~(\ref{def:type1}). 
\item[\emph{Step 2}.]  The aim is to use prefixes instead of suffixes. Note that if $X_i$ is defined as in~\cite{Thuswaldner06} by
$$X_i=\bigcup_{\sigma(j)=pis}\mathbf{h}X_j+\pi\mathbf{l}(s),
$$
then we have
$$\mathcal{T}( i )=-X_i-\pi\mathbf{l}(i). 
$$
This uses the unicity of the attractor solution of~(\ref{TileGIFS}) and the relation~(\ref{rel:abel}): 
for $\sigma(j)=pis$ in the above union, we have
$$\pi\mathbf{l}( s)=\pi\mathbf{l}(\sigma(j))-\pi\mathbf{l}( p)-\pi\mathbf{l}( i)=\mathbf{h}\pi\mathbf{l}(j)-\pi\mathbf{l}( p)-\pi\mathbf{l}( i)$$ 
Replace each edge 
$$[i,\gamma,j]\xrightarrow{(p_1,i,s_1)|(p_2,j,s_2)}[i',\gamma',j']\in\mathcal{C}_\partial^1$$  by an edge
$$[j,\gamma-\pi\mathbf{l}(j)+\pi\mathbf{l}(i),i]\xrightarrow{p_2|p_1}[j',\gamma'-\pi\mathbf{l}(j')+\pi\mathbf{l}(i'),i'].
$$
This change relies on the following computation. For an edge in $\mathcal{C}_\partial^1$ as above, we have the relation~(\ref{def:type1}). In particular,
$$\mathbf{h}\gamma'=\gamma+\pi\mathbf{l}(s_2)-\pi\mathbf{l}(s_1), 
$$
which is equivalent to 
$$\mathbf{h}\left(\gamma'-\pi\mathbf{l}(j')+\pi\mathbf{l}(i')\right)=\gamma-\pi\mathbf{l}(j)+\pi\mathbf{l}(i)+\pi\mathbf{l}(p_1)-\pi\mathbf{l}(p_2),
$$
again by using~(\ref{rel:abel}). The resulting graph is $G_0$. 
\end{itemize}

We write $X:=\bigcup_{i=1}^3X_i$. By~\cite[Theorem 4.3]{Thuswaldner06},  
\begin{equation}\label{BoundEq1}
\partial X=\bigcup_{[i,\gamma,j]\in R^1} C^1[i,\gamma,j]
\end{equation}
and 
\begin{equation}\label{BoundEq2}
\forall\;i=1,2,3,\; \partial X_i=\bigcup_{[i,\gamma,j]\in R^1} C^1[i,\gamma,j]\;\;\cup\;\;\bigcup_{[i,0,j]\in R_0^1} C^1[i,0,j],
\end{equation}
where the sets $C^1[i,\gamma,j]$ $([i,\gamma,j]\in R^1_0)$ are the solutions of the GIFS directed by $\mathcal{C}_\partial^1$, \emph{i.e.},
$$
\begin{array}{rcl}
\forall [i,\gamma,j]\in R_0^1, \;C^1[i,\gamma,j]&=&\bigcup_{[i,\gamma,j]\xrightarrow{(p_1,i,s_1)|(p_2,j,s_2)}[i_1,\gamma_1,j_1]\in \mathcal{C}_\partial^1}\hBF C^1[i_1,\gamma_1,j_1]+\pi\lBF( s_1)\\
&\subset&X_i\cap(X_j+\gamma).
\end{array}
$$ 
Here, the sets $R^1,R_0^1$ are defined for the graph $\mathcal{C}_\partial^1$ analogously to $R,R_0$. In particular,
$$R=\{[j,\gamma-\pi\mathbf{l}(j)+\pi\mathbf{l}(i),i];[i,\gamma,j]\in R^1 \},
$$
and a similar relation holds between $R_0$ and $R^1_0$.

By unicity of the attractor of the GIFS~(\ref{ContactGIFS}) directed by $G_0$, one can check that, for all $[i,\gamma,j]\in R_1^0$,
$$-C^1[i,\gamma,j]-\pi\lBF(i)=C[j,\gamma-\pi\lBF(j)+\pi\lBF(i),i].
$$
Using~(\ref{BoundEq2}), this leads to
$$\partial \T ( i )=\bigcup_{[i,\gamma,j]\in R} C[i,\gamma,j]\;\;\cup\;\;\bigcup_{[i,0,j]\in R_0\cap\{J,K,L\}} C[i,0,j] 
$$
for all $i=1,2,3$. As $\T=\bigcup_{i=1}^3\T (i )$ and $C[i,0,j]\subset \T (i )\cap \T (j )$, we finally obtain that
$$
\partial \T=\bigcup_{[i,\gamma,j]\in R} C[i,\gamma,j].
$$
\end{proof}

The following lemma is essential for the construction of the boundary parametrization in the next section. 
\begin{lemma}\label{PosEVContact} Let $a\geq b\geq 1$ and  $G_0=G_{0,a,b}$ as in Definition~\ref{def:G0}. We denote by $G=G_{a,b}$ the graph obtained from $G_0$ after deleting the states $J,K,L$ and all their in- and outcoming edges.  Let  $r=r_{a,b}$ be the number of states in $R_0\setminus\{J,K,L\}$ and $\mathbf{L}=\mathbf{L}_{a,b}$ the incidence matrix of $G$: 
$$\mathbf{L}=(l_{m,n})_{1\leq m,n\leq r}\;\;\textrm{ with }l_{m,n}=\#\{S_n\xrightarrow{p_1|p_2}S_m\in G\},
$$ 
where $\{S_1,\ldots,S_r\}=R_0\setminus\{J,K,L\}$. 
 Then there exists a strictly positive vector $\mathbf{u}=\mathbf{u}_{a,b}$ satisfying
 $$\mathbf{L}\mathbf{u}=\lambda\mathbf{u},
 $$
 
where $\lambda=\lambda_{a,b}$ is the largest root of the characteristic polynomial of $\mathbf{L}$. In particular, $\lambda$ is the largest root of
$$p_{a,b}(x)=x^4+(1-b)x^3+(b-a)x^2-(a+1)x-1.
$$
We normalize $\mathbf{u}=(u^{(1)},\ldots,u^{(r)})$ to have $u^{(1)}+\cdots+u^{(r)}=1$. 
\end{lemma}

\begin{proof} We refer to Tables~\ref{Gab},~\ref{Gab1},~\ref{Gaeqb},~\ref{Ga1b1} and the corresponding Figures~\ref{fig:Gab},~\ref{fig:Gab1},~\ref{fig:Gaeqb},~\ref{fig:Ga1b1}. Note that the restriction of the graph $G$ to the set of states 
\begin{itemize}
\item $R_0\setminus\{A,B,J,K\}$ if $a\geq b+1,b\geq 2$,
\item $R_0\setminus\{A,B,M,J,K\}$ if $a\geq 2,b=1$,
\item $R_0\setminus\{A,B,N,J,K,L\}$ if $a=b\geq 2$,
\item   $R_0\setminus\{A,B,M,N,J,K,L\}$ if $a=b=1$,
\end{itemize}
 is strongly connected. Moreover, every walk in $G$ starting from any of the remaining states $A,B,M$ or $N$ reaches this strongly connected part after at most two edges. This justifies the existence of a strictly positive eigenvector corresponding to the Perron-Frobenius eigenvalue $\lambda$ of $\mathbf{L}$, which is easily computed to be the largest root of $p_{a,b}$ for all these cases.
 \end{proof}
\begin{remark}
In general, even for the contact graph, the incidence matrix needs not have a positive dominant eigenvector. Our class of substitutions is therefore a special case.
\end{remark}

\end{section}

\begin{section}{Boundary parametrization}\label{sec:boundparam}
 Throughout this section, we fix $a\geq b\geq 1$. We will prove Theorem~\ref{ParamTheo}, that includes a parametrization of the boundary of $\T=\T_{a,b}$ based on the graph $G=G_{a.b}$. In Definition~\ref{def:GPlus}, we order the states and edges of the graph $G$. This ordering seems to be arbitrary, but it has a geometrical interpretation: it corresponds to an ordering of the boundary pieces and subpieces in the GIFS~(\ref{prop:ContactGIFS}) counterclockwise around the boundary of the Rauzy fractal $\T$. This choice of ordering will insure the left continuity of our parametrization (see the proof of Proposition~\ref{prop:CHoelder}).

\begin{definition}\label{def:GPlus} We call $G_{a,b}^+=G^+$ the ordered graph obtained from $G_{a,b}=G$ after ordering the states and edges as listed in Tables~\ref{Gab},~\ref{Gab1},~\ref{Gaeqb},~\ref{Ga1b1}, according to the values of $a,b$.

Moreover, we set $S^{--}:=S$ and call ${\bf o_{\textrm{max}}}^S$ or just ${\bf o_{\textrm{max}}}$ the number of edges starting from the state $S$. We define the following edges in $G^+$. Suppose 
\begin{itemize}
\item $S\notin\{3,4\}$ if $a\geq b+1,b\geq 2$;
\item $S\notin\{3,4,5\}$ if $a\geq 2,b=1$;
\item $S\notin\{2,3,4\}$ if $a=b\geq 2$;
\item   $S\notin\{2,3,4,5\}$ if $a=b=1$.
\end{itemize}
Then 
$$S^-\xrightarrow{p_1|p_2||{\bf o}} T^-\in G^+:\iff S\xrightarrow{p_2|p_1||{\bf o_{\textrm{max}}+1-o}} T\in G^+.$$  

Finally,  we call the states $\{1,2,\ldots,12\}$ (case $a>b$) or $\{1,2,\ldots,11\}$ (case $a=b$) the \emph{starting states} of $G^+$. A finite or infinite walk in $G^+$ is \emph{admissible} if it starts from a starting state. 

The corresponding graphs are depicted in Figures~\ref{fig:Gab},~\ref{fig:Gab1},~\ref{fig:Gaeqb},~\ref{fig:Ga1b1}. The starting states are colored. For simplicity, we write $S\xrightarrow{( m )}T$ whenever there are $m$ edges from $S$ to $T$. If $m=1$, we write the complete edge $S\xrightarrow{p||{\bf o}}T$. If $m=0$, there is simply no edge between the states. 

\end{definition}

\begin{remark}All  the edges of $G^+$ are of the form $S\xrightarrow{p_1|p_2||{\bf o}} T$. Since there is no ambiguity, we may simply write $S\xrightarrow{p_1||{\bf o}} T$ or $S\xrightarrow{{\bf o}} T$ or $(S;{\bf o})$. Note that if the condition of the last column in the tables is not fulfilled, then the edges of the corresponding line do not exist and there is exactly one edge starting from the associated state.
\end{remark}
\begin{remark} We write $(S;{\bf o_1},\ldots,{\bf o_n})$ for the walk of $G^+$ starting from the state $S$ with the edges successively labelled by ${\bf o_1},\ldots,{\bf o_n}$. By the above ordering of states and edges in $G^+$, the set of admissible walks of length $n$ ($n\geq 0$)  is lexicographically ordered, from the walk $(1;\underbrace{{\bf 1},{\bf 1},\ldots,{\bf 1}}_{n\textrm{ times }})$ to the walk $(12;\underbrace{{\bf o_{\textrm{max}}},{\bf o_{\textrm{max}}},\ldots,{\bf o_{\textrm{max}}}}_{n\textrm{ times }})$ ($a>b$) or $(11;\underbrace{{\bf o_{\textrm{max}}},{\bf o_{\textrm{max}}},\ldots,{\bf o_{\textrm{max}}}}_{n\textrm{ times }})$ ($a=b$). This holds also for the infinite admissible walks. 
\end{remark}

\begin{table}[ht]\scriptsize\begin{center}
\begin{tabular}{|c|c|c|c|c|c|}
\hline
\multicolumn{2}{|c}{Vertex} & \multicolumn{4}{|c|}{Edge(s)} \\
\hline
 \# & Order& to & Label $p_1|p_2$& Order & Condition\\
\hline
\multirow{3}{*}{$C$} &\multirow{3}{*}{$1$}  & $7$ & $ k|a-b+k,\;0 \leq k \leq b-1$  & ${\bf 1+3(b-1-k)},0\leq k\leq b-1$ &\\
    & & $5$ & $k|a-b+1+k,\; 0 \leq k \leq b-2$  &${\bf 2+3(b-2-k)},0\leq k\leq b-2$ & \\
  & & $6$ & $k|a-b+1+k,\;0 \leq k \leq b-2$ & ${\bf 3+3(b-2-k)},0\leq k\leq b-2$& \\
  \hline
\multirow{3}{*}{$N$}&\multirow{3}{*}{$2$}  & $8$ & $0|a-1$ &${\bf 1}$ & \\
   & & $9$ & $0|a-1$  &${\bf 2}$&  \\
 & & $10$ & $0|a-1$  &${\bf 3}$& \\
\hline
\multirow{4}{*}{$A$} &\multirow{4}{*}{$3$}  &  $11$ & $0|b-1$  &${\bf 1}$ &\\
  & & $12$ & $0|b-1$  &${\bf 2}$& \\
&&$1$ & $k|b-1+k,\;0 \leq k \leq a-b$  &${\bf 3+2k},0\leq k\leq a-b$&  \\
  & & $2$ & $k|b+k,\; 0 \leq k \leq a-b-1$  & ${\bf 4+2k},0\leq k\leq a-b-1$& \\
\hline
\multirow{2}{*}{$B$} &\multirow{2}{*}{$4$} & $2$ & $a-b|a$  &${\bf 1}$& \\
 & & $1$ & $a-b+1|a$  &$ {\bf 2}$& \\
 \hline
\multirow{3}{*}{$I$} &\multirow{3}{*}{$5$}   & $7^-$ & $a-1-k|b-2-k,\;0 \leq k\leq b-2$  &${\bf 1+3(b-2-k)},0\leq k\leq b-2$&  \\
  & & $6^-$ & $a-1-k|b-3-k,\;0 \leq k\leq b-3$ &${\bf 2+3(b-3-k)},0\leq k\leq b-3$& $b \geq 3$ \\
  & & $5^-$ & $a-1-k|b-3-k,\;0 \leq k\leq b-3$  &${\bf 3+3(b-3-k)},0\leq k\leq b-3$& $b \geq 3$ \\
 \hline
\multirow{3}{*}{$H$} &\multirow{3}{*}{$6$} &$6^-$ & $a|b-2$  &${\bf 1}$&  \\
  & & $5^-$ & $a|b-2$ & ${\bf 2}$& \\
 && $7^-$ & $a|b-1$  &${\bf 3}$&\\
  \hline
\multirow{3}{*}{$P$}&\multirow{3}{*}{$7$}  &  $10^-$ & $a|0$  &${\bf 1}$& \\
  & & $9^-$ & $a|0$  &${\bf 2}$&  \\
  &&$8^-$ & $a|0$ & ${\bf 3}$&\\
\hline
\multirow{2}{*}{$E$} &\multirow{2}{*}{$8$}  & $1^-$ & $a|a-b$  &${\bf 1}$& \\
  & & $2^-$ & $a|a-b-1$ &${\bf 2}$&  \\
  \hline
\multirow{2}{*}{$G$}&\multirow{2}{*}{$9$}  & $1^-$ & $a-1-k|a-b-1-k,\;  0\leq k  \leq a-b-1$  &${\bf 1+2k},0\leq k\leq a-b-1$ & \\
  &  & $2^-$ & $a-1-k|a-b-2-k,\;  0\leq k  \leq a-b-2$  &${\bf 2+2k},0\leq k\leq a-b-2$& $a \geq b+2$ \\
  \hline
 \multirow{2}{*}{$F$}& \multirow{2}{*}{$10$} & $12^-$ & $b|0$  &${\bf 1}$& \\
 & & $11^-$ & $b|0$  &${\bf 2}$& \\
 \hline
$O$&$11$  & $7$ & $b|a$ &${\bf 1}$&  \\
\hline
\multirow{2}{*}{$D$} &\multirow{2}{*}{$12$}   &  $5$ & $b-1|a$  &${\bf 1}$&\\
 &&$6$ & $b-1|a$ & ${\bf 2}$& \\
\hline
$J$&  & $3$ & $a-1|a$&  & \\
\hline
\multirow{2}{*}{$K$} &\multirow{2}{*}{}  & $4$ & $b-1|b$&  & \\
  & & $J$ & $b|b$ && \\
\hline 
\end{tabular}
\caption{$G_{0,a,b}$ for $a\geq b+1,\;b\geq 2$.}\label{Gab}
\end{center}
\end{table}

\begin{table}[ht]\scriptsize\begin{center}
\begin{tabular}{|c|c|c|c|c|c|c|}
\hline
\multicolumn{2}{|c}{Vertex} & \multicolumn{4}{|c|}{Edge(s)} \\
\hline
 \# & Order& to & Label $p_1|p_2$& Order & Condition\\
 \hline
$C$&$1$ & $7$ & $ 0|a-1$  &${\bf 1}$  &\\
\hline
\multirow{3}{*}{$N$}&\multirow{3}{*}{$2$}  & $8$ & $0|a-1$ & ${\bf 1}$& \\
  & & $9$ & $0|a-1$  &${\bf 2}$& \\
  & & $10$ & $0|a-1$  &${\bf 3}$& \\
\hline
\multirow{4}{*}{$A$} &\multirow{4}{*}{$3$}  &  $11$ & $0|0$  & ${\bf 1}$&\\
  & & $12$ & $0|0$  &${\bf 2}$& \\
&& $1$ & $k|k,\;0 \leq k \leq a-1$  &${\bf 3+2k},0\leq k\leq a-1$&  \\
  & & $2$ & $k|1+k,\; 0 \leq k \leq a-2$  & ${\bf 4+2k},0\leq k\leq a-2$& \\
\hline
$B$&$4$ & $2$ & $a-1|a$  &${\bf 1}$& \\
\hline
$M$&$5$ & $1$ & $a|a$  &${\bf 1}$&\\
\hline
$H$&$6$ & $7^-$ & $a|0$  &${\bf 1}$& \\
\hline
\multirow{3}{*}{$P$}&\multirow{3}{*}{$7$}  &  $10^-$ & $a|0$  &${\bf 1}$& \\
  & & $9^-$ & $a|0$  &${\bf 2}$&  \\
&&  $8^-$ & $a|0$ & ${\bf 3}$& \\
\hline
\multirow{2}{*}{$E$} &\multirow{2}{*}{$8$}  & $1^-$ & $a|a-1$  &${\bf 1}$& \\
  & & $2^-$ & $a|a-2$ &${\bf 2}$&  \\
  \hline
\multirow{2}{*}{$G$}&\multirow{2}{*}{$9$}  & $1^-$ & $a-1-k|a-2-k,\;  0\leq k  \leq a-2$  &${\bf 1+2k},0\leq k\leq a-2$& \\
  &  & $2^-$ & $a-1-k|a-3-k,\;  0\leq k  \leq a-3$  &${\bf 2+2k},0\leq k\leq a-3$& $a \geq 3$ \\
  \hline
 \multirow{2}{*}{$F$}& \multirow{2}{*}{$10$} & $12^-$ & $1|0$  &${\bf 1}$& \\
 & & $11^-$ & $1|0$  &${\bf 2}$& \\
 \hline
$O$&$11$  & $7$ & $1|a$ &${\bf 1}$&  \\
\hline
$D$ &$12$   &$6$ & $0|a$ &${\bf 1}$ & \\
  \hline
$J$&  & $3$ & $a-1|a$&  & \\
\hline
\multirow{3}{*}{$K$} &\multirow{3}{*}{}  & $4$ & $0|1$&  & \\
  &  &$J$ & $1|1$ && \\
  &  &$5$ & $0|1$&& \\
  \hline 
\end{tabular}
\caption{$G_{0,a,b}$ for $a\geq 2,\;b=1$.}\label{Gab1}
\end{center}
\end{table}

\begin{table}[ht]\scriptsize\begin{center}
\begin{tabular}{|c|c|c|c|c|c|c|}
\hline
\multicolumn{2}{|c}{Vertex} & \multicolumn{4}{|c|}{Edge(s)} \\
\hline
 \# & Order& to & Label $p_1|p_2$& Order & Condition\\
 \hline
\multirow{3}{*}{$C$} &\multirow{3}{*}{$1$}  & $7$ & $ k|k,\;0 \leq k \leq a-1$  &${\bf 1+3(a-1-k)},0\leq k\leq a-1$  &\\
  & & $5$ & $k|1+k,\; 0 \leq k \leq a-2$  &${\bf 2+3(a-2-k)},0\leq k\leq a-2$ & \\
  & & $6$ & $k|1+k,\;0 \leq k \leq a-2$ & ${\bf 3+3(a-2-k)},0\leq k\leq a-2$& \\

\hline
\multirow{2}{*}{$N$}&\multirow{2}{*}{$2$}  & $8$ & $0|a-1$ & ${\bf 1}$& \\
  & & $9$ & $0|a-1$  &${\bf 2}$& \\
\hline
\multirow{3}{*}{$A$} &\multirow{3}{*}{$3$}  & $10$ & $0|a-1$  &${\bf 1}$ &\\
  & & $11$ & $0|a-1$  &${\bf 2}$& \\
 &&$1$ & $0|a-1$  &${\bf 3}$&  \\
\hline
\multirow{2}{*}{$B$} &\multirow{2}{*}{$4$} & $2$ & $0|a$  &${\bf 1}$& \\
 & & $1$ & $1|a$  &${\bf 2}$ & \\

\hline
\multirow{3}{*}{$I$} &\multirow{3}{*}{$5$}   & $7^-$ & $a-1-k|a-2-k,\;0 \leq k\leq a-2$  &${\bf 1+3(a-2-k)},0\leq k\leq a-2$&  \\
  & & $6^-$ & $a-1-k|a-3-k,\;0 \leq k\leq a-3$ &${\bf 2+3(a-3-k)},0\leq k\leq a-3$& $a \geq 3$\\
  & & $5^-$ & $a-1-k|a-3-k,\;0 \leq k\leq a-3$  &${\bf 3+3(a-3-k)},0\leq k\leq a-3$& $a \geq 3$ \\
\hline
\multirow{3}{*}{$H$} &\multirow{3}{*}{$6$} & $6^-$ & $a|a-2$  &${\bf 1}$& \\
  & & $5^-$ & $a|a-2$ &${\bf 2}$ & \\
&&$7^-$ & $a|a-1$  &${\bf 3}$& \\
\hline
\multirow{2}{*}{$P$}&\multirow{2}{*}{$7$}    &  $9^-$ & $a|0$  &${\bf 1}$& \\
&& $8^-$ & $a|0$ & ${\bf 2}$& \\
\hline
$E$ &$8$  & $1^-$ & $a|0$  &${\bf 1}$& \\
\hline
 \multirow{2}{*}{$F$}& \multirow{2}{*}{$9$} & $11^-$ & $a|0$  &${\bf 1}$& \\
 & & $10^-$ & $a|0$  &${\bf 2}$& \\

\hline
$O$&$10$  & $7$ & $a|a$ &${\bf 1}$&  \\
\hline
\multirow{2}{*}{$D$} &\multirow{2}{*}{$11$}    &  $5$ & $a-1|a$  &${\bf 1}$&\\
&&$6$ & $a-1|a$ &${\bf 2}$ & \\

\hline
$J$&  & $3$ & $a-1|a$&  & \\
\hline
$K$ &  & $4$ & $a-1|a$&  & \\
\hline
$L$ & & $J$ & $a|a$  &&\\

  \hline 
\end{tabular}
\caption{$G_{0,a,b}$ for $a=b\geq 2$.}\label{Gaeqb}
\end{center}
\end{table}

\begin{table}[ht]\scriptsize\begin{center}
\begin{tabular}{|c|c|c|c|c|}
\hline
\multicolumn{2}{|c}{Vertex} & \multicolumn{3}{|c|}{Edge(s)} \\
\hline
 \# &Order & to & Label $p_1|p_2$&Order\\
 \hline
$C$ &$1$  & $7$ & $ 0|0$  &{\bf 1 } \\
\hline
\multirow{2}{*}{$N$}&\multirow{2}{*}{$2$}  & $8$ & $0|0$  &{\bf 1 } \\
  & & $9$ & $0|0$  &{\bf 2} \\
\hline
\multirow{3}{*}{$A$} &\multirow{3}{*}{$3$}  & $10$ & $0|0$  &{\bf 1 }\\ 
  & & $11$ & $0|0$  &{\bf 2 } \\
&&$1$ & $0|0$  & {\bf 3 } \\
\hline
$B$ &$4$ & $2$ & $0|0$  & {\bf 1 } \\
\hline
$M$&$5$ & $1$ & $1|1$  & {\bf 1 } \\
\hline
$H$ &$6$ & $7^-$ & $1|0$  &{\bf 1 } \\
\hline
\multirow{2}{*}{$P$}&\multirow{2}{*}{$7$}    &  $9^-$ & $1|0$  &{\bf 1 } \\
&& $8^-$ & $1|0$ &{\bf 2 } \\

\hline
$E$&$8$  & $1^-$ & $1|0$  &{\bf 1 } \\
\hline
 \multirow{2}{*}{$F$}& \multirow{2}{*}{$9$} & $11^-$ & $1|0$  &{\bf 1 } \\
 & & $10^-$ & $1|0$  &{\bf 2 } \\
 \hline
$O$&$10$  & $7$ & $1|1$ & {\bf 1 } \\
\hline
$D$ &$11$  &$6$ & $0|1$  &{\bf 1 } \\
\hline
$J$&  & $3$ & $0|1$&   \\
\hline
\multirow{2}{*}{$K$} &\multirow{2}{*}{}  & $4$ & $0|1$&   \\
  & & $5$ & $0|1$& \\
\hline
$L$&  & $J$ & $1|1$  & \\
  \hline 
\end{tabular}
\caption{$G_{0,a=1,b=1}$ (Tribonacci substitution).}\label{Ga1b1}
\end{center}
\end{table}

\begin{figure}
\begin{center}
\includegraphics[width=170mm,height=170mm]{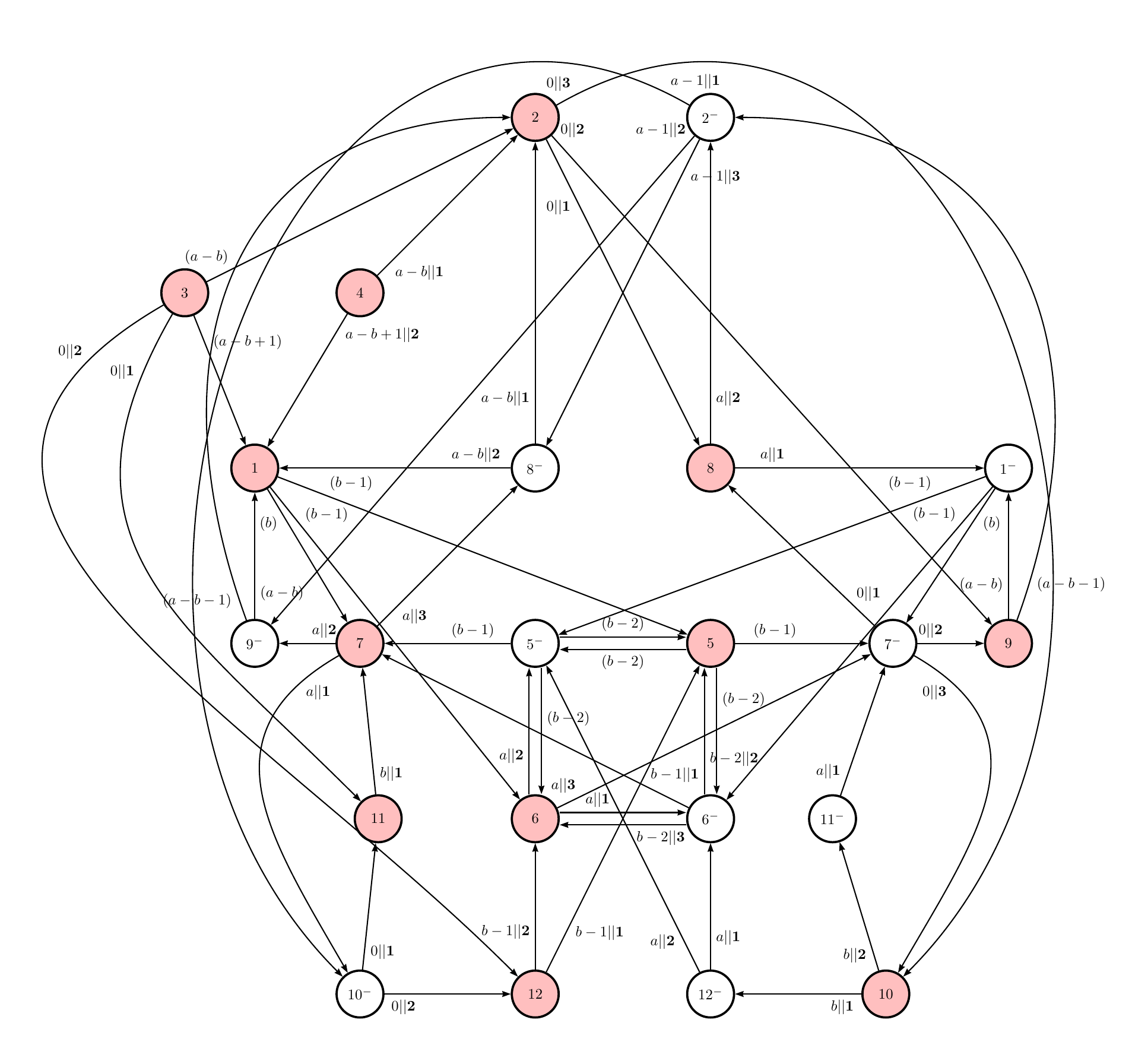}\end{center}
\caption{$G_{a,b}^+$ for $a\geq b+1,b\geq 2$.}\label{fig:Gab}
\end{figure}

\begin{figure}
\begin{center}
\includegraphics[width=170mm,height=160mm]{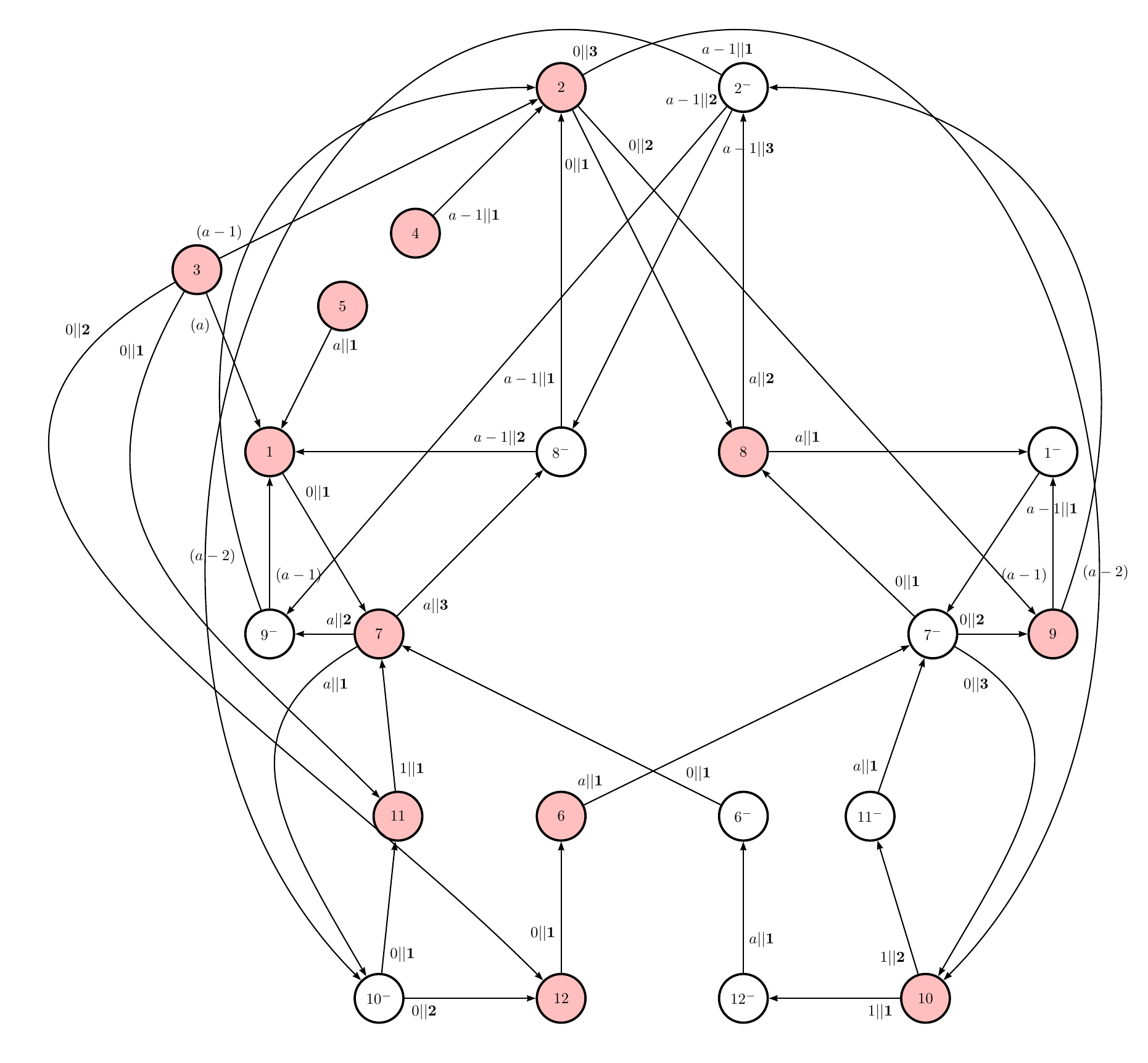}\end{center}
\caption{$G_{a,b}^+$ for $a\geq 2,b=1$.}\label{fig:Gab1}
\end{figure}

\begin{figure}
\begin{center}
\includegraphics[width=150mm,height=150mm]{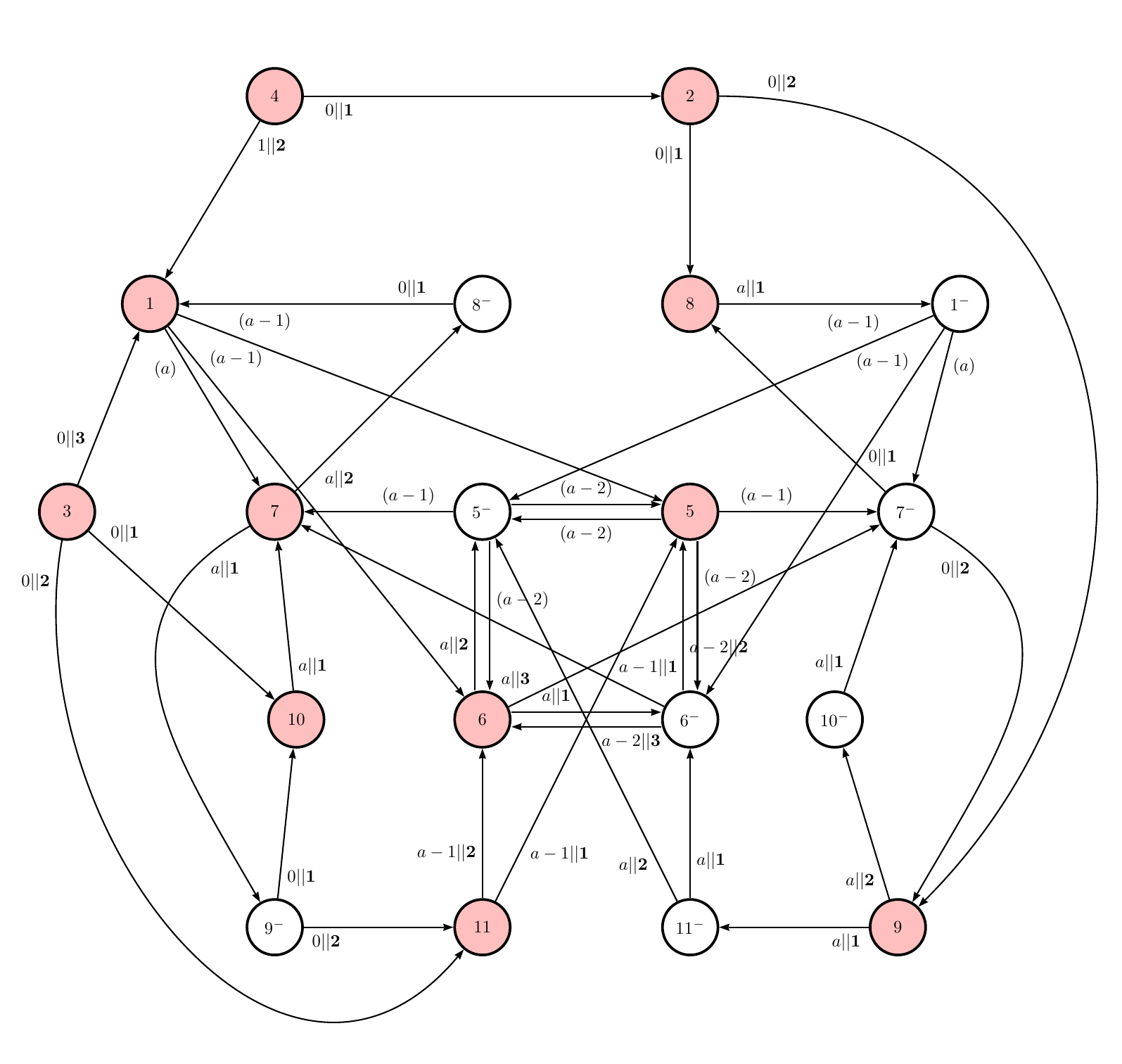}\end{center}
\caption{$G_{a,b}^+$ for $a= b\geq 2$.}\label{fig:Gaeqb}
\end{figure}

\begin{figure}
\begin{center}
\includegraphics[width=140mm,height=130mm]{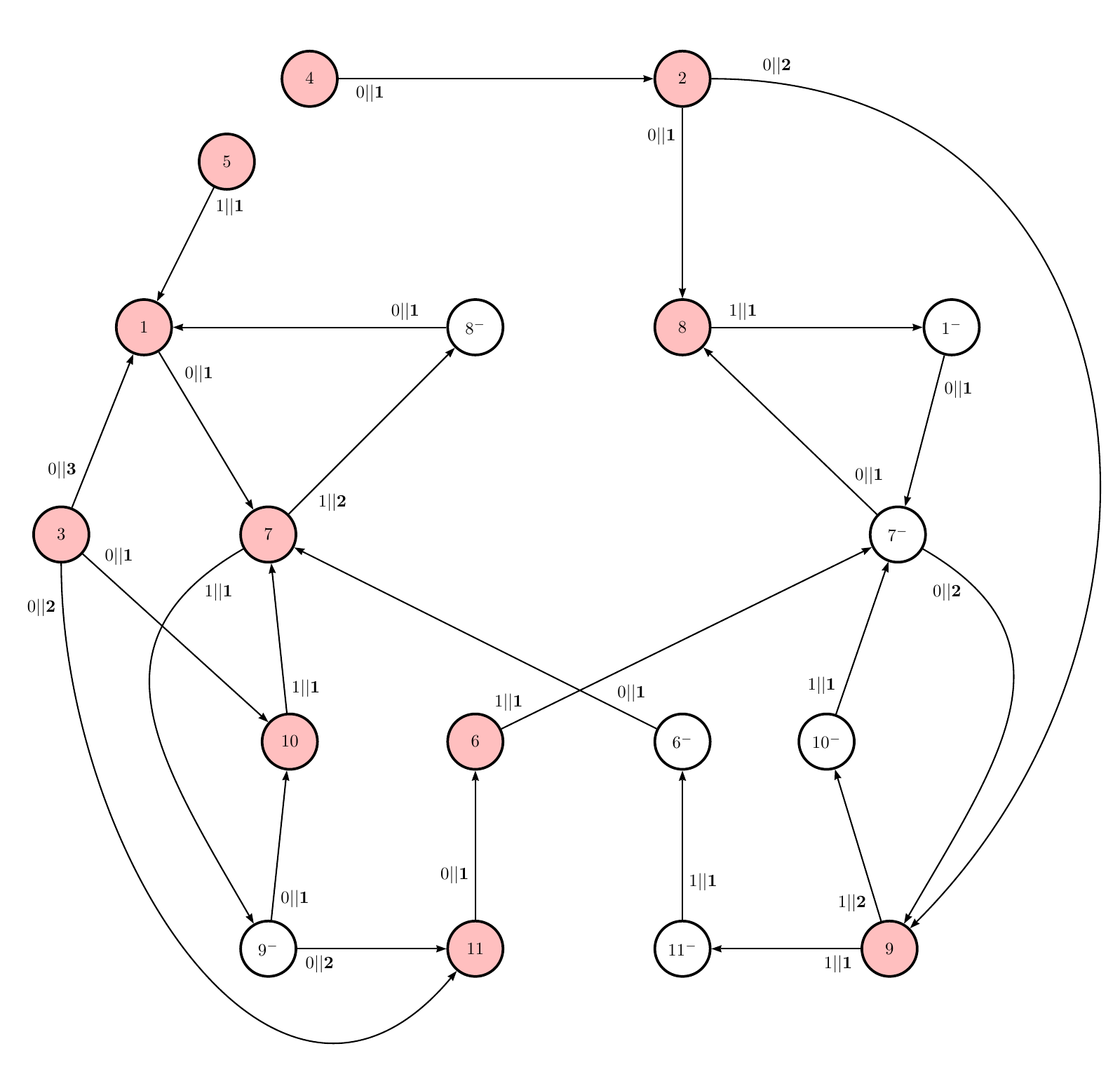}\end{center}
\caption{$G_{a,b}^+$ for $a=b=1$ (Tribonacci substitution).}\label{fig:Ga1b1}
\end{figure}

The parametrization procedure now runs along the same lines as in~\cite[Section 3]{AkiyamaLoridant11}. Let $\mathbf{L}$ be the incidence matrix of $G^+$ (or $G$) and $\mathbf{u}=(u^{(1)},\ldots,u^{(r)})$ a strictly positive eigenvector for the dominant eigenvalue $\lambda$ as in Lemma~\ref{PosEVContact}, normalized to have $u^{(1)}+\cdots+u^{(r)}=1$. The automaton $G^+$ induces a number system, also known as \emph{Dumont-Thomas numeration system}~\cite{DumontThomas89}. We map each admissible infinite walk of $G^+$ to a point in the unit interval $[0,1]$, according to the schema shown in Figure~\ref{IntervalSubdivisions}: we first subdivide $[0,1]$ in subintervals of lengths given by the coordinates of $\mathbf{u}$; we then subdivide each subinterval, for example the subinterval $[{\bf{u}}^{(1)},{\bf{u}}^{(1)}+{\bf{u}}^{(2)}]$, by using $\mathbf{L}{\bf u}=\lambda{\bf u}$; and we iterate this procedure. 
 \begin{figure}[h!]
\includegraphics[width=130mm,height=45mm]{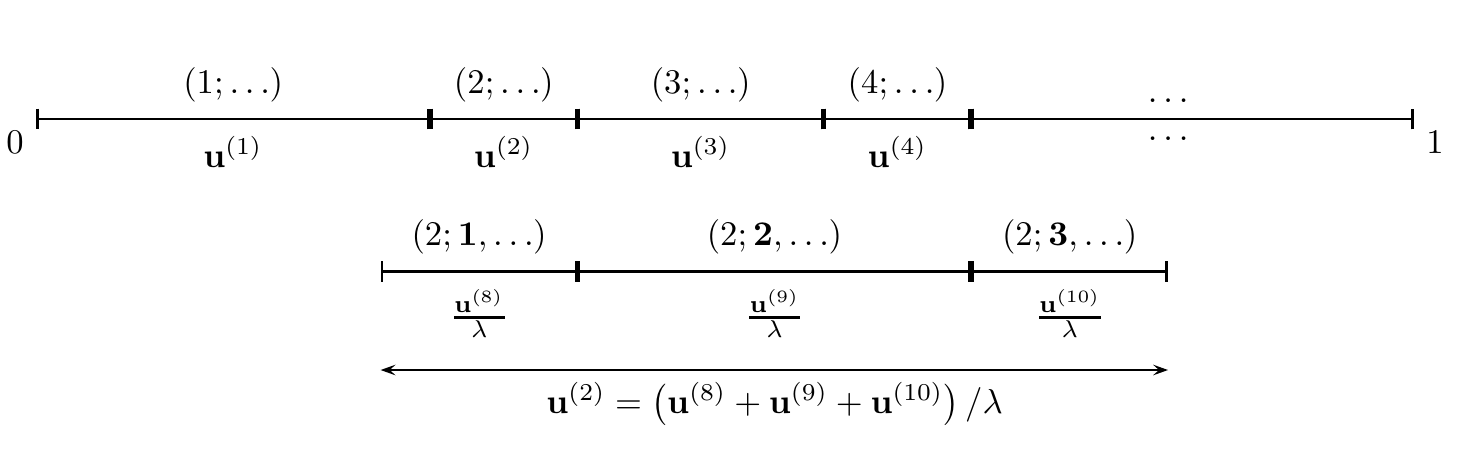}
\caption{Parametrization procedure: interval subdivisions (case $a> b$).}\label{IntervalSubdivisions}
\end{figure}
More precisely, we define a function
$$
\phi^0(S;{\bf o}) \;=\; 
\left\{\begin{array}{cc}
0, &\textrm{if }{\bf o}={\bf 1};\\
\displaystyle\sum_{\begin{array}{c}{\bf 1}\leq {\bf k}< {\bf o},\\S\xrightarrow{{\bf k}}S'\end{array}}
{\bf u}^{(S')},&\textrm{if }{\bf o}\ne {\bf 1}.
\end{array}
\right.
$$
Thus $\phi^0(S;{\bf o}) <\displaystyle\sum_{\begin{array}{c}{\bf 1}\leq {\bf k}\leq {\bf o_{\textrm{max}}},\\S\xrightarrow{{\bf k}}S'\end{array}}
{\bf u}^{(S')}=\lambda {\bf u}^{(S)}$ for each edge $(S;{\bf o})$.

We set ${\bf u}^{(0)}:=0$ and map the admissible infinite walks in $G^+$ to $[0,1]$ via
\begin{equation}\label{def:DTmap}
\begin{array}{rcl}
\phi\;:\;G^+&\xrightarrow{}&[0,1]\\\\
 w &\mapsto& \lim_{n\to\infty}
 \;(\;{\bf u}^{(0)}+{\bf u}^{(1)}+\ldots+{\bf u}^{(S-1)}
 \\&&\;\;\;+\frac{1}{\lambda}\phi^0(S;{\bf o_1})+\frac{1}{\lambda^2}\phi^0(S_1;{\bf o_2})+\ldots+\frac{1}{\lambda^{n}}\phi^0(S_{n-1};{\bf o_n})\;)
\end{array}
\end{equation}
whenever $w$ is the admissible infinite walk:
$$w:\;S\xrightarrow{{\bf o_1}}S_1\xrightarrow{{\bf o_2}}\ldots\xrightarrow{{\bf o_n}}S_{n}\xrightarrow{{\bf o_{n+1}}}\ldots
$$
This mapping $\phi$ is well-defined, increasing, onto, and it is almost 1 to 1, as identifications occur exactly on pairs of lexicographically consecutive infinite walks. Indeed, let $w\ne w'$ be admissible infinite walks in $G^+$, say for example $w >_{lex}w'$. Then $\phi(w)=\phi(w')$ iff 
\begin{equation}\label{identifpairs}
\begin{array}{lcr}
1.\left\{
\begin{array}{rcl}
w&=&(S+1;\overline{{\bf 1}})\\
w'&=&(S;\overline{{\bf o}_{\textrm{max}}})
\end{array}
\right.
&
\textrm{ or }&
2.\left\{
\begin{array}{rcl}
w&=&(S;{\bf o_1},\ldots,{\bf o_m},{\bf o+1},\overline{{\bf 1}})\\
w'&=&(S;{\bf o_1},\ldots,{\bf o_m},{\bf o},\overline{{\bf o}_{\textrm{max}}})
\end{array}
\right.
\end{array}
\end{equation}
holds for some state $S=1,\ldots,S_{\textrm{max}}$ or some prefix $(S;{\bf o_1},\ldots,{\bf o_m})$ and an order ${\bf o}$. Here, $S_{\textrm{max}}=12$ (case $a>b$) or $11$ (case $a=b$). By $\overline{{\bf o}}$, we mean the infinite repetition of the order ${\bf o}$.  We omit the proofs of these facts here, since they are similar to proofs given in~\cite{AkiyamaLoridant11}.

Consequently, if $t\in[0,1]$, then $\phi^{-1}(t)$ consists of either one or two elements. Hence an inverse of $\phi$ can be defined as
$$
\begin{array}{rcl}
\phi^{(1)}\;:\;[0,1]&\xrightarrow{} &G^+\\
 t &\mapsto& \max^{lex}{\phi^{-1}(t)},
\end{array}
$$
where $\max^{lex}$ maps a finite set of walks to its lexicographically maximal walk.

We finally denote by $P$ the natural bijection:
\begin{equation}\label{mapP}
\begin{array}{rrcl}
P:&G^+&\to&G\\
 &(S;{\bf o_1},{\bf o_2},\ldots) &\mapsto& w:\;S\xrightarrow{p_0|p_0'}S_1\xrightarrow{p_1|p_1'}\ldots
\end{array}
\end{equation}
whenever $(S;{\bf o_1},{\bf o_2},\ldots)=S\xrightarrow{p_0|p_0'||\mathbf{o_1}}S_1\xrightarrow{p_1|p_1'||\mathbf{o_2}}\ldots$ is an admissible walk of $ G^+$, and by $\psi$ the mapping
\begin{equation}\label{mapPsi}
\begin{array}{rrcl}
\psi:&G&\to&\partial \T\\
 &w= [i,\gamma,j]\xrightarrow{p_0|p_0'}[i_1,\gamma_1,j_1]\xrightarrow{p_1|p_1'}\ldots&\mapsto& \sum_{k\geq 0} \mathbf{h}^k\pi {\bf l}(p_k)
\end{array}
\end{equation} 

 This allows us to define our parametrization mapping $C$.
\begin{proposition}\label{prop:Csurj} The mapping $C:[0,1]\xrightarrow{\phi^{(1)}}G^+\xrightarrow{P}G\xrightarrow{\psi} \partial \T$ is well-defined and surjective. Furthermore, let 
$$\begin{array}{rl}
A:=\{t\in[0,1]\;;t=&{\bf u}^{(0)}+{\bf u}^{(1)}+\ldots+{\bf u}^{(S-1)}\\
&\;+\;\frac{1}{\lambda}\phi^0(S;{\bf o_1})+\frac{1}{\lambda^2}\phi^0(S_1;{\bf o_2})+\ldots+\frac{1}{\lambda^{n}}\phi^0(S_{n-1};{\bf o_n}) \\
&\textrm{ for some admissible finite walk }S\xrightarrow{p_0||{\bf o_1}}S_1\xrightarrow{p_1||{\bf o_2}}\ldots\xrightarrow{p_{n-1}||{\bf o_n}}S_n\in G^+
\}.
\end{array}
$$
Then $C$ is continuous on $[0,1]\setminus A$, and right continuous on $A$. Also, if $t\in A$, then $\lim_{t^-}C$ exists.
\end{proposition}
The proof relies on arguments of Hata~\cite{Hata85} and is given in~\cite[Proposition 3.4]{AkiyamaLoridant11}. Note that here the contractions,  associated with the prefixes $p\in\{\epsilon,1\ldots,\underbrace{1\cdots1}_{a}\}$, are affine mappings of the form $f_p({\bf x})=\hBF {\bf x}+\pi\lBF( p )$.

The above proposition means that discontinuities of $C$ may occur if $\psi$ does not identify walks that are ``trivially'' identified by the numeration system $\phi$ as in (\ref{identifpairs}). The following proposition insures the identifications for finitely many such pairs of walks. Because of  the graph directed self-similarity of the boundary, this will be sufficient to infer the continuity of $C$ on the whole interval $[0,1]$ (see Proposition~\ref{prop:CHoelder}). 

\begin{proposition}\label{prop:PsiIdentif} The following equalities hold.
\begin{eqnarray}
\label{Cond1}\psi(P(S;\overline{{\bf o_{\textrm{max}}}}))&=&\psi(P(S+1;\overline{{\bf 1}}))\;\;(1\leq S\leq S_{\textrm{max}}-1)\\
\label{Cond2}\psi(P(S_{\textrm{max}};\overline{{\bf o_{\textrm{max}}}}))&=&\psi(P(1;\overline{{\bf 1}}))\\
\label{Cond3}\psi(P(S;{\bf o},\overline{{\bf o_{\textrm{max}}}}))&=&\psi(P(S;{\bf o+1},\overline{{\bf1}}))\;\;(\forall \;S,\; 1\leq {\bf o}<{\bf o_{\textrm{max}}}),
\end{eqnarray}
where $S_{\textrm{max}}=12$ (case $a>b$) or $11$ (case $a=b$).
\end{proposition}
\begin{proof} We check~(\ref{Cond1}) in the following way. Let us consider the case $a\geq b+1,b\geq2$. We refer to Table~\ref{Gab} and Figure~\ref{fig:Gab}. We read for $S=1$
$$\begin{array}{rclcrcl}
P(1;\overline{{\bf o_{\textrm{max}}}})&=&\overline{1\xrightarrow{0}7\xrightarrow{a}8^{-}\xrightarrow{a-b}1}&\;\textrm{and }&
P(2;\overline{{\bf 1}})&=&2\xrightarrow{0}\overline{8\xrightarrow{a}1^{-}\xrightarrow{a-b}7^{-}\xrightarrow{0}8}
\end{array}
$$ 
(infinite cycles). Since the infinite sequences of prefixes are both equal to the periodic sequence $\overline{0a(a-b)}$, it follows from the definition of $\psi$ (see~(\ref{mapPsi})) that the equality $\psi(P(1;\overline{{\bf o_{\textrm{max}}}}))=\psi(P(2;\overline{{\bf 1}}))$ is trivially satisfied. 

The same happens for the other values of $S$, apart from $S=5,8,9,11$. For example, we read for $S=5$
$$\begin{array}{rcl}
P(5;\overline{{\bf o_{\textrm{max}}}})&=&5\xrightarrow{a-1}\overline{7^{-}\xrightarrow{0}10\xrightarrow{b}11^{-}\xrightarrow{a}7^-}\\
&\textrm{and }\\
P(6;\overline{{\bf 1}})&=&6\xrightarrow{a}6^-\xrightarrow{b-1}\overline{7\xrightarrow{a}10^{-}\xrightarrow{0}11\xrightarrow{b}7}.
\end{array}
$$ 
Note that this does not exclude the case $b=2$ (we then simply have ${\bf o_{\textrm{max}}}={\bf 1}$ for the state $5$). Therefore we read the infinite sequence of prefixes 
$$(a-1)\overline{0ba}\;\;\textrm{ and }\;\;a(b-1)\overline{a0b}.
$$
To prove that these sequences lead to the same boundary point, we use Lemma~\ref{CharacBound}. Indeed, there exists the pair of following infinite walks in $G_0\subset\mathcal{G}_0$:
$$\left\{\begin{array}{c}5\xrightarrow{p_0=a-1|p'_0=b-2}\overline{7^{-}\xrightarrow{p_1=0|p_1'=a}10\xrightarrow{p_2=b|p_2'=0}11^{-}\xrightarrow{p_3=a|p_3'=b}7^-}\\\\
6\xrightarrow{p_0''=a|p_0'=b-2}6^-\xrightarrow{p_1''=b-1|p_1'=a}\overline{7\xrightarrow{p_2''=a|p_2'=0}10^{-}\xrightarrow{p_3''=0|p_3'=b}11\xrightarrow{p_4''=b|p_4'=a}7},
\end{array}\right.
$$
and $5=[1,\pi(1,-1,1),1],\;6=[2,\pi(1,-1,1),1]$. Consequently, by the mentioned lemma, we have
$$\sum_{k\geq 0}{\bf h}^k\pi {\bf l}(p_k)=\pi(1,-1,1)+\sum_{k\geq 0}{\bf h}^k\pi {\bf l}(p_k')=\sum_{k\geq 0}{\bf h}^k\pi {\bf l}(p_k'').
$$

Conditions~(\ref{Cond2}) and (\ref{Cond3}) are checked similarly. The computations are carried out in Appendix~\ref{app:Continuity}. See also this appendix for the remaining values of $a,b$. 
\end{proof}
In the appendix and in the rest of this section, we use the following notation for the concatenation of walks: 
\begin{equation}\label{def:concatenation}
(S;{\bf o_1},\ldots,{\bf o_n})\& (S';{\bf o_{n+1}},\ldots):=(S,{\bf o_1},\ldots,{\bf o_n},{\bf o_{n+1}},\ldots),
\end{equation} 
where $S'$ is the ending state of the walk $(S;{\bf o_1},\ldots,{\bf o_n})$ of $G^+$.

\begin{proposition}\label{prop:CHoelder}
The mapping $C$ satisfies $C(0)=C(1)$ and it is H\"older continuous with exponent 
$s=-\frac{\log|\alpha|}{\log\lambda}$, where $|\alpha|=\max\{|\alpha_1|,|\alpha_2|\}$ ($\alpha_1,\alpha_2$ are the conjugates of the Pisot number $\beta$).
\end{proposition}
\begin{proof}
The proof mainly relies on an argument of Hata~\cite{Hata85}. By Proposition~\ref{prop:Csurj}, we just need to check the left-continuity of $C$ on the countable set $A$.  This will result from Proposition~\ref{prop:PsiIdentif}. First note that (\ref{Cond3}) means that $C(0)=C(1)$. Also, (\ref{Cond1}) and (\ref{Cond3}) mean that $C$ is left continuous at the points associated to walks of length $n=1$ in the definition of $A$. We now prove that this is sufficient for $C$ to be continuous on the whole set $A$. This follows from the definition of $\psi$. Indeed, let $t\in A$ associated to a walk of length $n\geq 2$ but not to a walk of smaller length. Thus
$$t=\phi(\underbrace{S;{\bf o_1},\ldots,{\bf o_n},\overline{{\bf 1}}}_{w})
$$ 
with ${\bf o_n}\ne {\bf 1}$.  We write $(p_0,p_1,\ldots)$ for the labeling sequence of $P(w)$.  Then,
$$\begin{array}{rcl}
C(t)&=&\psi(P(w))=\sum_{k=0}^{n-1}\hBF^k\pi\lBF(p_k)+\hBF^n\psi(P(S';{\bf o_n},\overline{{\bf1}}))\\
&=&
\sum_{k=0}^{n-1}\hBF^k\pi\lBF(p_k)+\hBF^n\psi(P(S';{\bf o_n-1},\overline{{\bf o_{\textrm{max}}}}))\;\;\textrm{(by Condition~(\ref{Cond3}))}\\
&=&C(t^-)
\end{array}
$$
(here $S'$ is the ending state of the finite admissible walk $w|_{n-1}$ in the automaton $G_0^+$). Thus $C$ is left continuous in $t$.

More details as well as the proof of the H\"older continuity can be found in~\cite[Proposition 3.5]{AkiyamaLoridant11}. The exponent is $-\frac{\log|\delta|}{\log\lambda}$,  where $\delta$ is the maximal contraction factor among all the contractions in the GIFS, \emph{i.e.}, the maximal contraction factor of the mapping $\mathbf{h}$ (see Subsection~\ref{subsec:RF}). This is exactly $|\alpha|$. 
\end{proof}

\begin{remark} If $|\alpha_1|=|\alpha_2|=|\alpha|$, \emph{i.e.}, if the contraction $\mathbf{h}$ is a similarity, then $\beta|\alpha^2|=1$, therefore the above H\"older exponent is related to the Hausdorff dimension of the boundary: 
$$s=\frac{1}{\textrm{dim}_H(\partial \T)}
$$
(see~\cite[Theorem 6.7]{Thuswaldner06} or~\cite[Theorem 4.4]{SiegelThuswaldner10}). As mentioned in~\cite[Section 6.4]{Thuswaldner06}, this is the case as soon as $D=\frac{1}{108}\left(27-4b^3+18ab-a^2b^2+4a^3\right)\geq0$. 
\end{remark}

We now give the sequence of polygonal closed curves $\Delta_n$ converging to $\partial \T$ with respect to the Hausdorff metric. For $N$ points $M_1,\ldots,M_N$ of $\mathbb{R}^2$, we denote by $\left[M_1,\ldots,M_N\right]$ the curve joining $M_1,\ldots,M_N$ in this order by straight lines. 

\begin{definition}\label{def:polygapprox} Let $w_1^{(n)},\ldots,w_{N_n}^{(n)}$ be the walks of length $n$ in the graph $G^+$, written in the lexicographical order: 
$$(1;{\bf 1},\ldots,{\bf 1})=w^{(n)}_1\leq_{lex}w^{(n)}_2\leq_{lex}\ldots\leq_{lex}w^{(n)}_{N_n}=(S_{\textrm{max}};{\bf o_{\textrm{max}}},\ldots,{\bf o_{\textrm{max}}}),
$$
 where $N_n$ is the number of these walks. For $n=0$, these are just the states $1,\ldots,S_{\textrm{max}}$. Let 
$$C_j^{(n)}:=C(\phi(w^{(n)}_j\&\overline{{\bf1}}))\;\in\partial T\;\; (1\leq j\leq N_n).
$$
Then we call 
$$\Delta_n:=\left[C_1^{(n)},C_2^{(n)},\ldots,C_{N_n}^{(n)},C_1^{(n)}\right],
$$
the $n$\emph{-th approximation of} $\partial \T_n$. 
\end{definition}
The first terms of the sequences $(\Delta_n)_{n\geq 0}$ are depicted for $a=b=1$ in Figure~\ref{BoundaTribo} and for $a=10,b=7$ in Figure~\ref{Bounda10b7}.
\begin{proposition}\label{prop:convapprox}$\Delta_n$ is a polygonal closed curve and its vertices have $\mathbb{Q}(\lambda)$-addresses in the parametrization $C$. Moreover, $(\Delta_n)_{n\geq 0}$ converges to $\partial \T$ in Hausdorff metric.
\end{proposition}
\begin{proof}
By definition, $\Delta_n$ is a polygonal closed curve with vertices on $\partial \T$. The vertices have $\mathbb{Q}(\lambda)$-addresses in the parametrization, since they correspond to parameters $t\in A$, where $A$ is the countable set defined in Proposition~\ref{prop:Csurj}. Note that  ${\bf u}$, defined in Lemma~\ref{PosEVContact}, is the dominant eigenvector of a non-negative matrix with dominant eigenvalue $\lambda$, hence its components belong to $\mathbb{Q}(\lambda)$. Finally, one can check that $\Delta_{n+1}$ is obtained from $\Delta_n$ after applying the GIFS~(\ref{ContactGIFS}). This is due to the fact that the contractions are affine mappings. Therefore, $(\Delta_n)_{n\geq0}$ converges in Hausdorff metric to the unique attractor, which is $\partial \T$. Details can be found in~\cite[Section 3]{AkiyamaLoridant11}.
\end{proof}
\begin{proof}[Proof of Theorem~\ref{ParamTheo}]
Theorem~\ref{ParamTheo} is now a consequence of Lemma~\ref{PosEVContact}, Proposition~\ref{prop:CHoelder} and Proposition~\ref{prop:convapprox}. 
\end{proof}

\begin{figure}
\begin{tabular}{cccc}
\includegraphics[width=40mm,height=45mm]{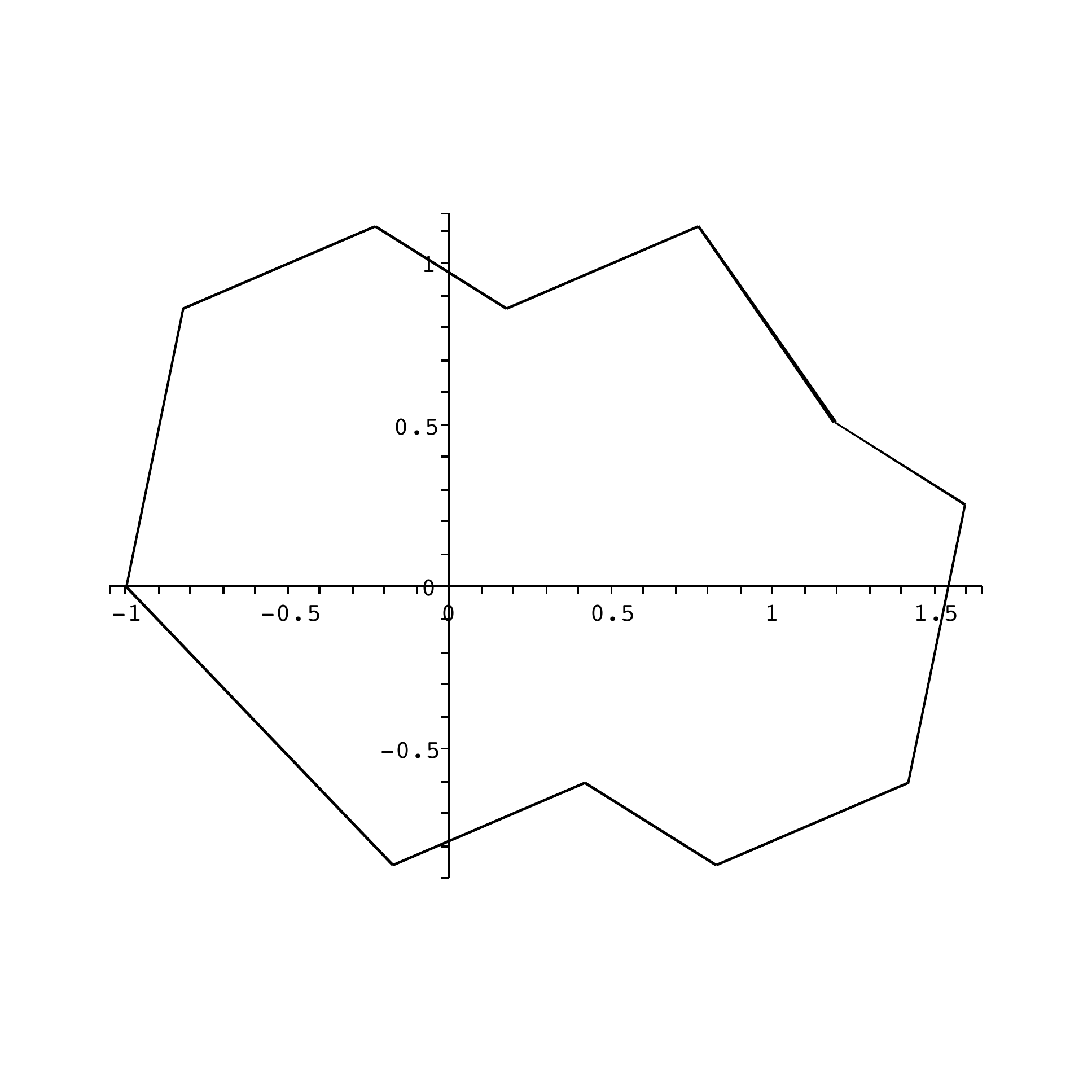}&\includegraphics[width=40mm,height=45mm]{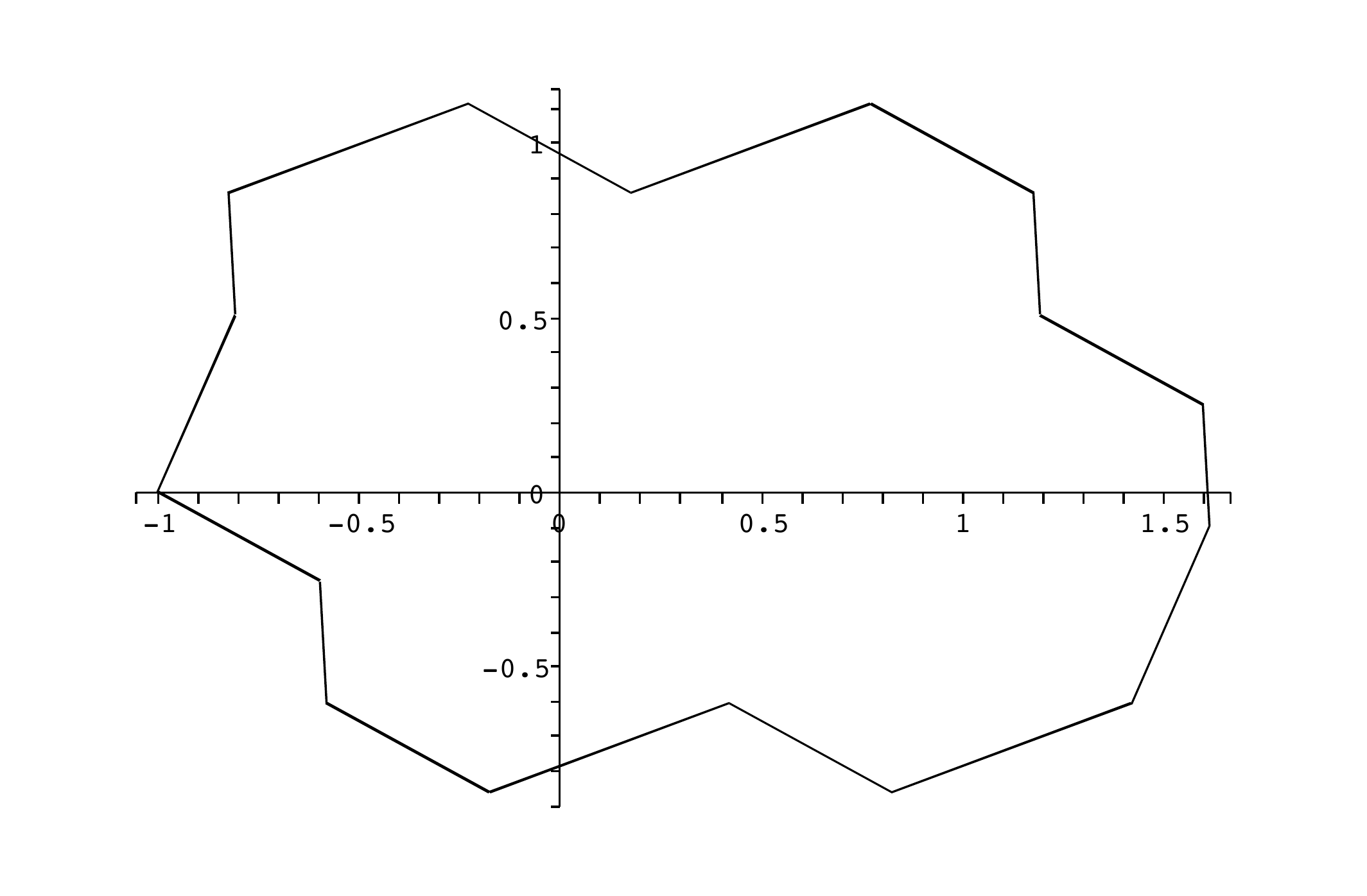}&
\includegraphics[width=40mm,height=45mm]{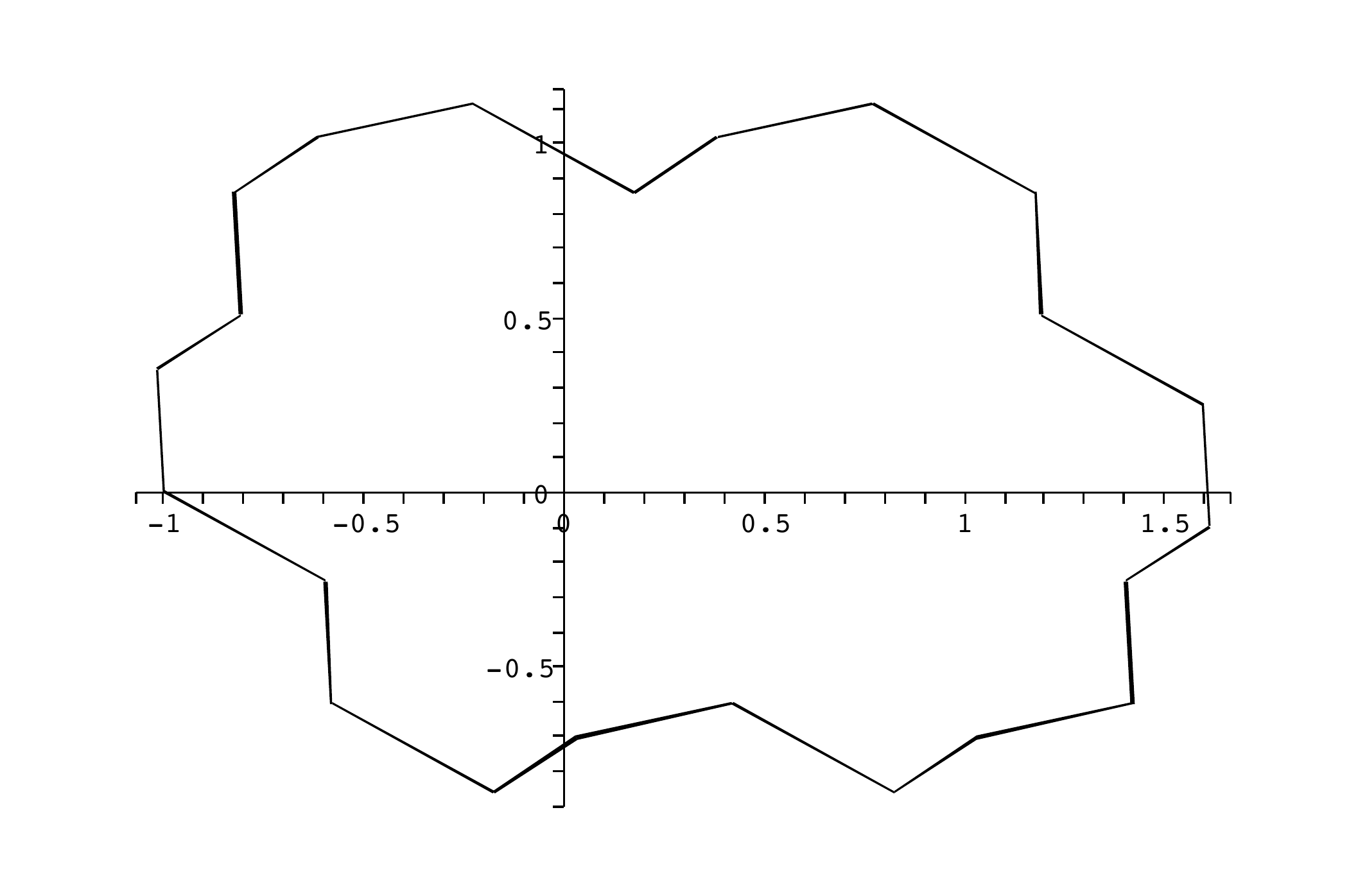}&\includegraphics[width=40mm,height=45mm]{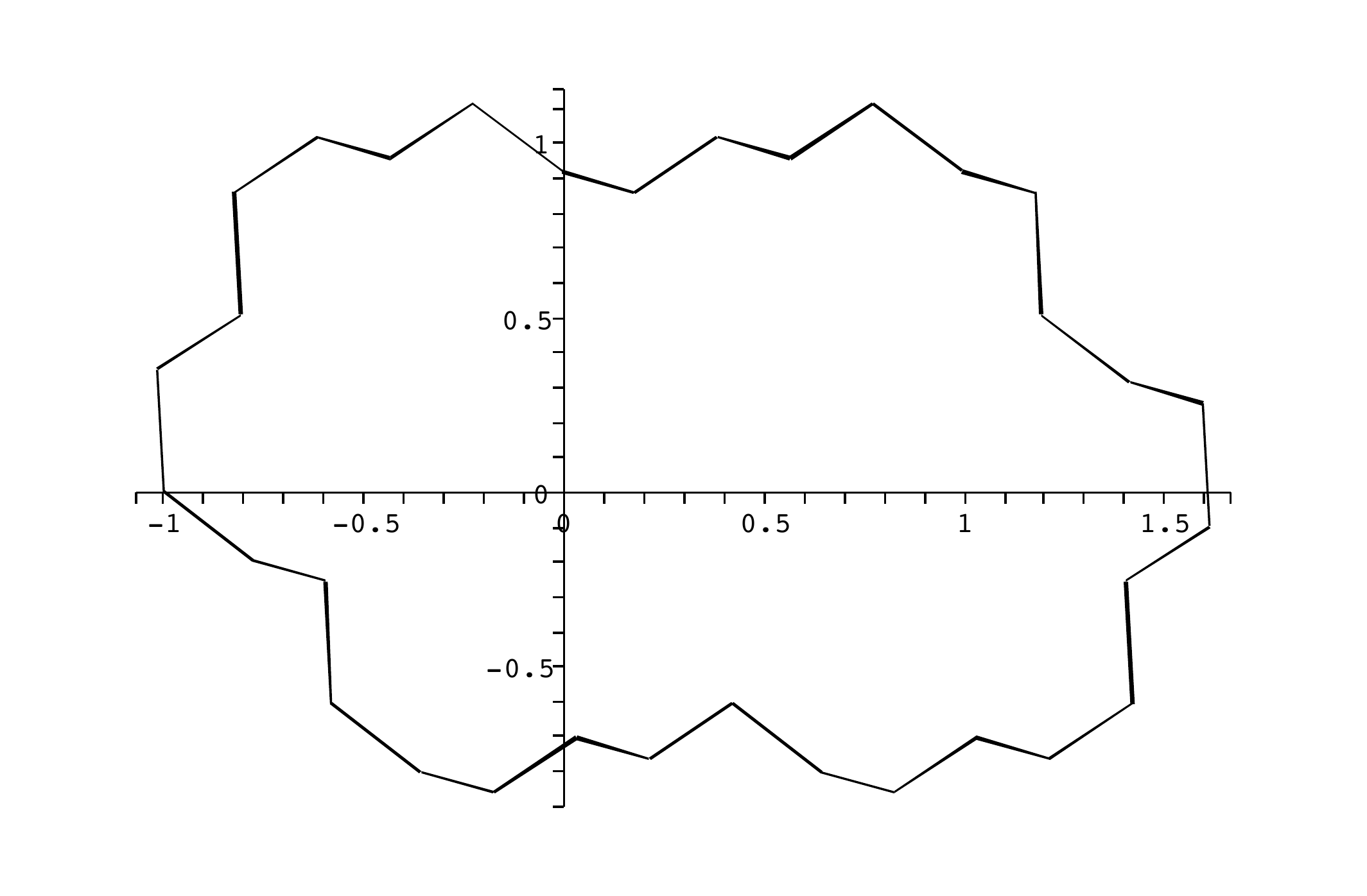}
\\
\includegraphics[width=40mm,height=45mm]{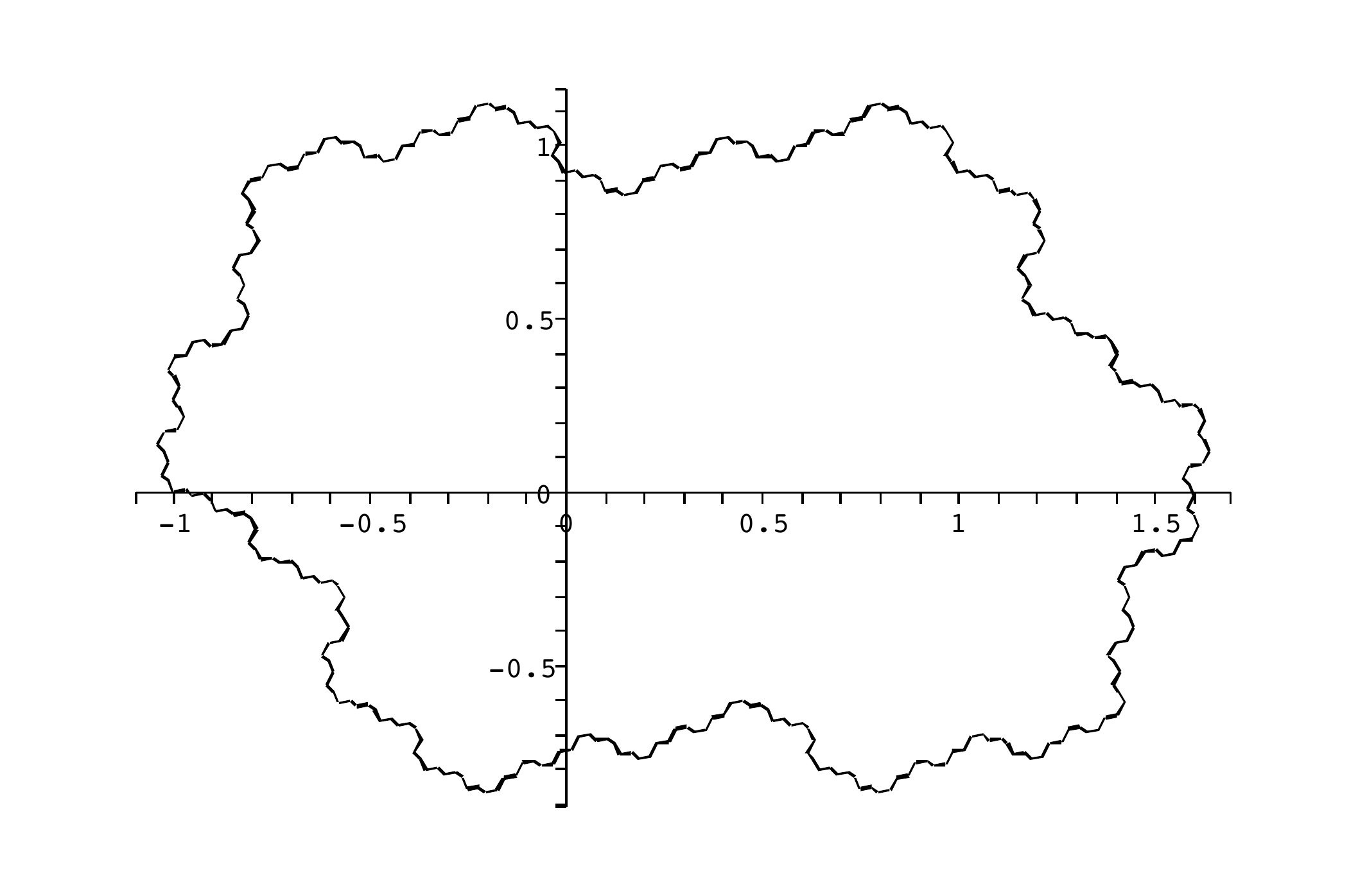}&\includegraphics[width=40mm,height=45mm]{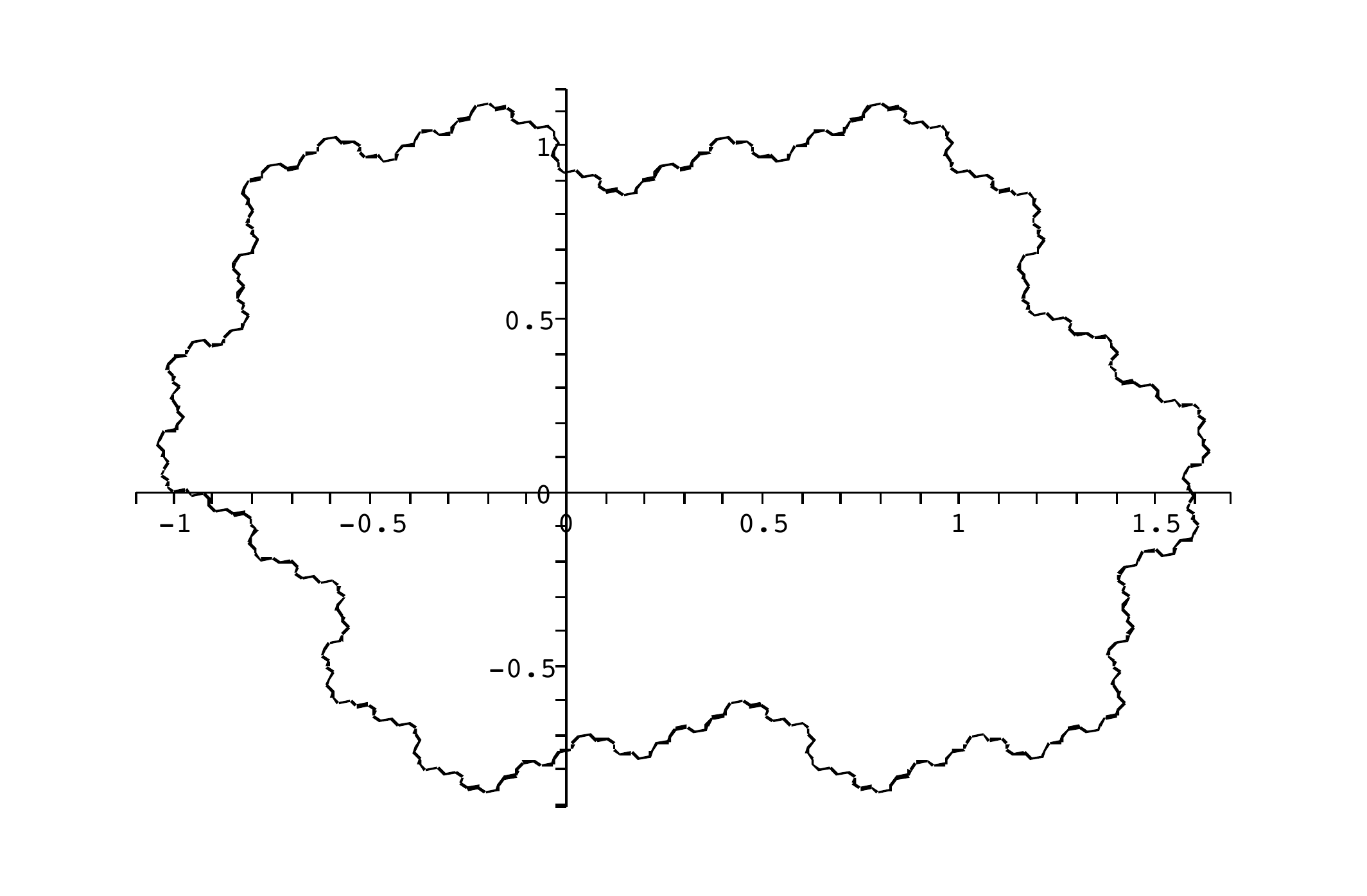}&
\includegraphics[width=40mm,height=45mm]{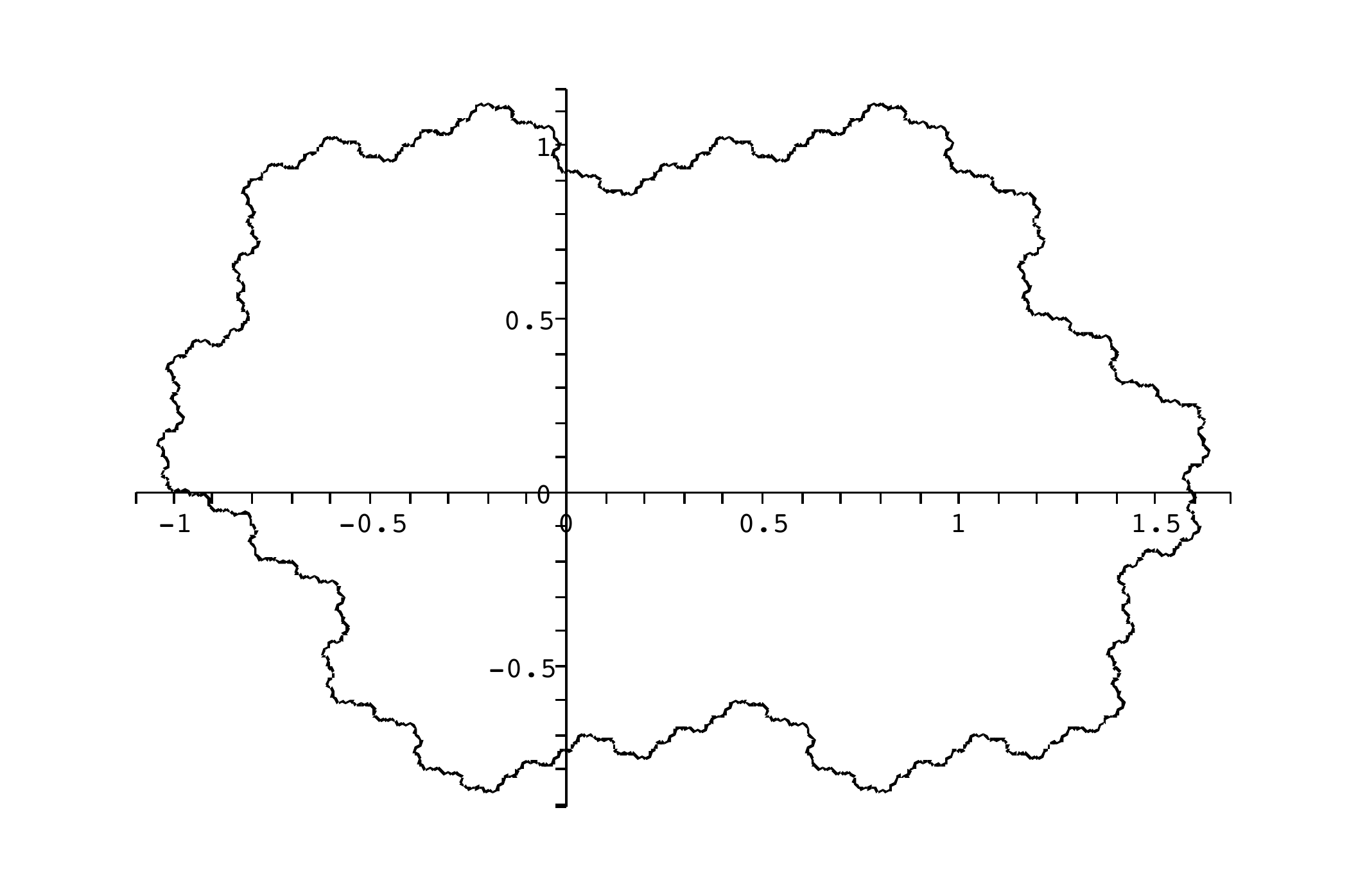}&\includegraphics[width=40mm,height=45mm]{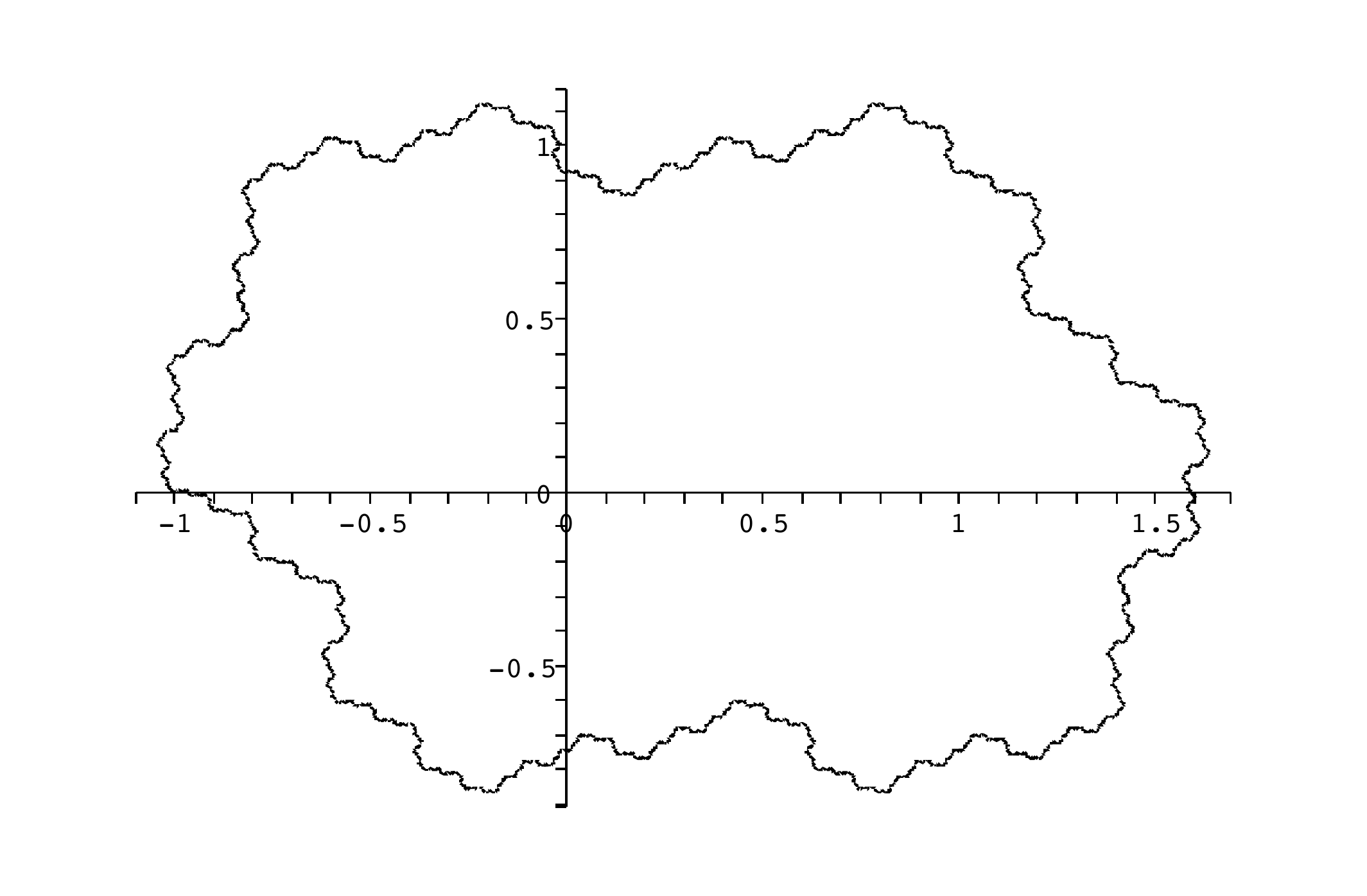}
\end{tabular}
\caption{Boundary approximations for the Rauzy fractal $\T_{1,1}$}\label{BoundaTribo}
\end{figure}

\begin{figure}
\begin{tabular}{cccc}
\includegraphics[width=40mm,height=45mm]{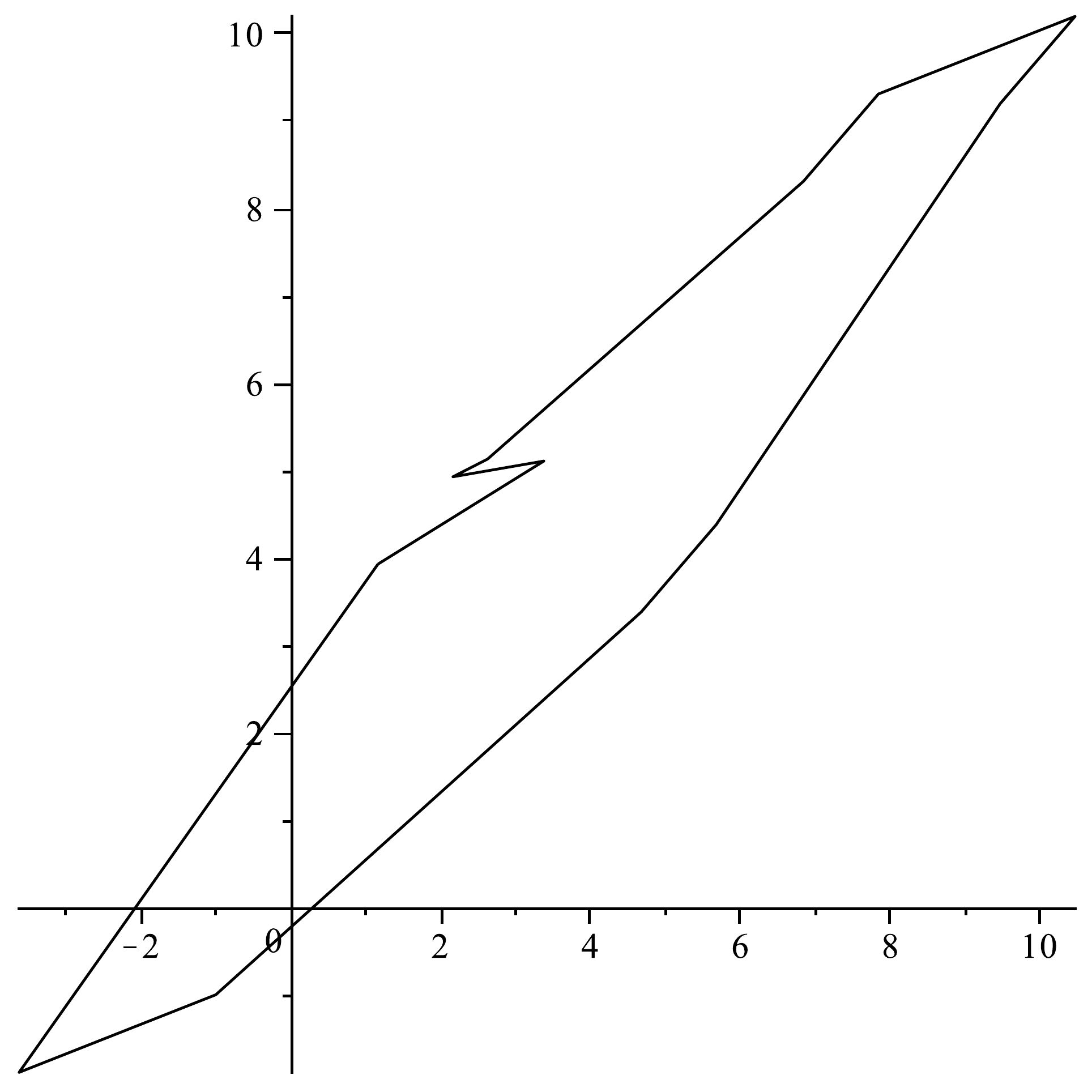}&\includegraphics[width=40mm,height=45mm]{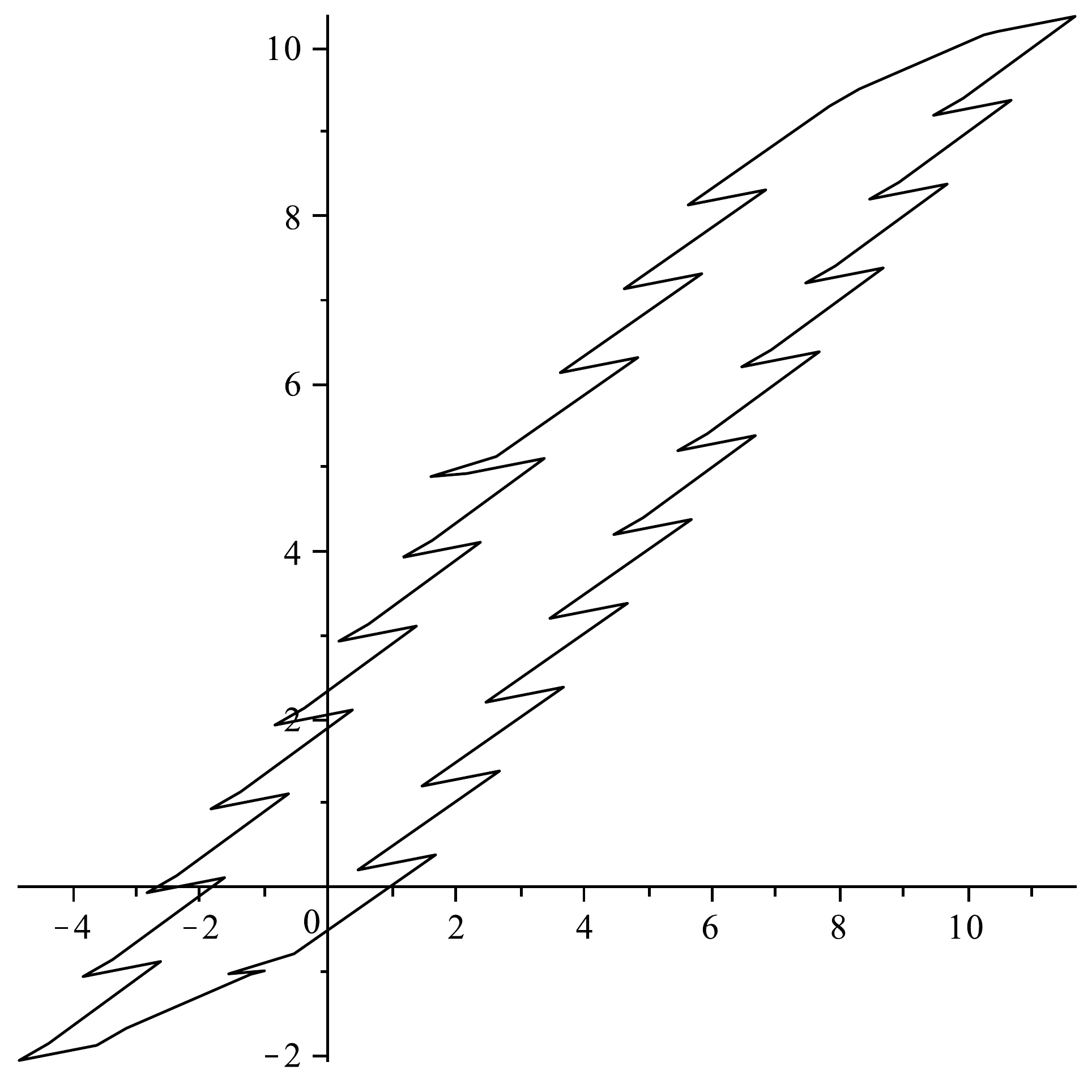}&
\includegraphics[width=40mm,height=45mm]{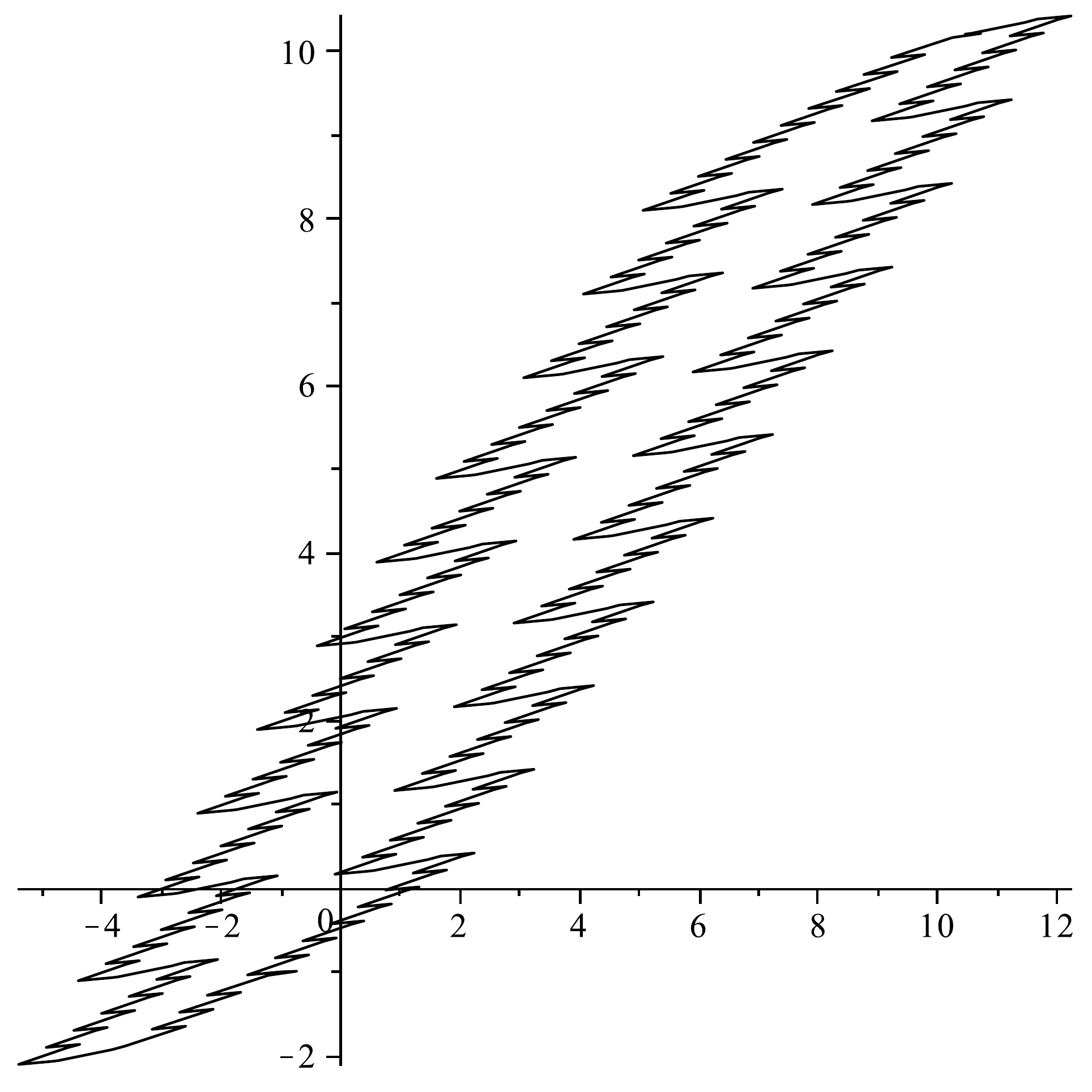}&\includegraphics[width=40mm,height=45mm]{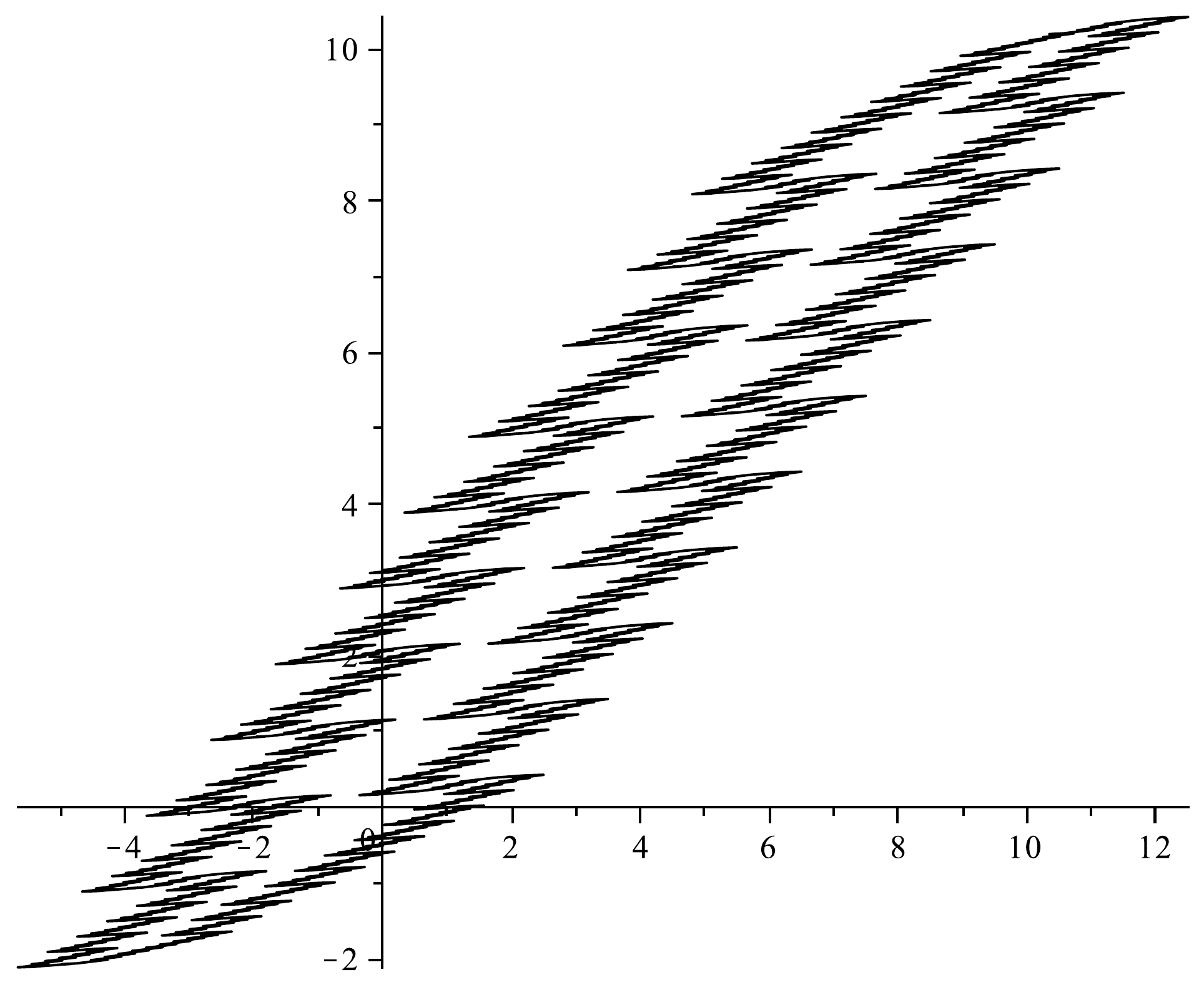}
\end{tabular}
\caption{Boundary approximations for $\T_{10,7}$}\label{Bounda10b7}
\end{figure}

\begin{remark}\label{rem:DekMeth}The way of generation of the approximations $\Delta_n$ is analogous to Dekking's recurrent set method~\cite{Dekking82b,Dekking82a}. Consider for example the Tribonacci case. The ordered automaton $G^+$ on Figure~\ref{fig:Ga1b1} gives rise to a free group endomorphism $$\Theta:\;\;1\to7,\;\;2\to8\;9,\;\;3\to10\;11\;1,\;\;\ldots$$ on the free group generated by the letters $1,2,\ldots, 11$. An edge is associated to each letter, the word $W_0=1\;2\;3\;\cdots\;11$ is mapped to the $11$-gon $\Delta_0$ depicted on Figure~\ref{BoundaTribo}. The iterations $\Theta^n(W_0)$ map to $\Delta_n$ after renormalization.
\end{remark}

\end{section}

\begin{section}{Proof of Theorem~\ref{DNDTheo}}\label{sec:prooftheo}

We recall the statement concerning non disk-like tiles. \emph{Let $\mathcal{T}_{a,b}$ be the tile associated to the substitution $\sigma_{a,b}$  ($a\geq b\geq 1$). If $2b-a> 3$, then $\mathcal{T}_{a,b}$ is not homeomorphic to a closed disk.}

\begin{proof}[Proof of Theorem~\ref{DNDTheo} (non disk-like tiles)] One proof can be found in~\cite{LoridantMessaoudiSurerThuswaldner13}, but needed the additional computation of a subgraph of the lattice boundary graph for all parameters $a,b$ - a graph that describes the boundary in the periodic tiling induced by $\T$. Here, we make no use of this periodic tiling. The proof below uses the parametrization derived from the graph $G$, already obtained by Thuswaldner in~\cite{Thuswaldner06}, or, more precisely, from our ordered version $G^+$. 

In our assumptions $a\geq b\geq 1$ and $2b-a>3$, we just need to consider two cases:
\begin{itemize}
\item[$(i)$] $b\geq 3$ and $a\geq b+1$;
\item[$(ii)$] $a=b>3$.
\end{itemize} 

In case $(i)$ we find infinite walks associated to distinct parameters $0<t<t'<1$:  
$$\left\{\begin{array}{l}t:5\xrightarrow{b-1||{\bf 2+3(2b-a-3)}}6^-\begin{array}{c}\xrightarrow{b-2||{\bf 2}}\\\xleftarrow{b-2||{\bf 2+3(2b-a-4)}}\end{array}5
\\
t':12\xrightarrow{b-1||{\bf 1}}5\begin{array}{c}\xrightarrow{b-2||{\bf 2+3(2b-a-4)}}\\\xleftarrow{b-2||{\bf 2}}\end{array}6^-
\end{array}
\right.
$$
in $G$. We refer to Table~\ref{Gab} and the corresponding Figure~\ref{fig:Gab}. 

Similarly, in case $(ii)$ we find the following infinite walks associated to distinct parameters $0<t<t'<1$:  
$$\left\{\begin{array}{l}t:5\xrightarrow{a-1||{\bf 2+3(a-3)}}6^-\begin{array}{c}\xrightarrow{a-2||{\bf 2}}\\\xleftarrow{a-2||{\bf 2+3(a-4)}}\end{array}5
\\
t':11\xrightarrow{a-1||{\bf 1}}5\begin{array}{c}\xrightarrow{a-2||{\bf 2+3(a-4)}}\\\xleftarrow{a-2||{\bf 2}}\end{array}6^-
\end{array}
\right.
$$
in $G$ (see Table~\ref{Gaeqb} and the corresponding Figure~\ref{fig:Gaeqb}).

Therefore, in both cases, we have $0<t<t'<1$ satisfying $C(t)=C(t')$, since the associated infinite walks in $G$ carry the same labels. Hence $\partial \T$ is not a simple closed curve.

\end{proof}

We now come to the characterization of the disk-like tiles. \emph{Let  $\mathcal{T}_{a,b}$ be the tile associated to the substitution $\sigma_{a,b}$  ($a\geq b\geq 1$). If $2b-a\leq 3$, then $\partial\mathcal{T}_{a,b}$ is a simple closed curve. Therefore, $\mathcal{T}_{a,b}$ is homeomorphic to a closed disk.}

We wish to show that all pairs $(t,t')\in[0,1[$ with $C(t)=C(t')$ satisfy $t=t'$. In other words, we shall show that all pairs of walks $(w,w')\in G^+$ with 
$\psi(P(w))=\psi(P(w'))$ satisfy $\phi(w)=\phi(w')$, where $\phi, P, \psi$ are defined in~(\ref{def:DTmap}),~(\ref{mapP}) and ~(\ref{mapPsi}), respectively. 

We first characterize the infinite sequences of prefixes $(p_k)_{k\geq 0},(p_k')_{k\geq 0}$ leading to the same boundary point $\sum_{k\geq 0} \mathbf{h}^k\pi \lBF(p_k)=\sum_{k\geq 0} \mathbf{h}^k\pi \lBF(p_k')$.

\begin{lemma}\label{lem:EqAd} Let $(p_k)_{k\geq 0}$ and $(p_k')_{k\geq0}$ be the labels of infinite walks  in the prefix-suffix graph $\Gamma$ starting from $i\in\mathcal{A}$ and $i'\in\mathcal{A}$, respectively. Then 
$$\sum_{k\geq 0} \mathbf{h}^k\pi \lBF(p_k)=\sum_{k\geq 0} \mathbf{h}^k\pi \lBF(p_k')=:x\in\partial \T$$

if and only if there exist $j\in\mathcal{A}$, $\gamma\in\pi(\mathbb{Z}^3)\setminus\{0\}$ with $[\gamma,j]\in\Gamma_{srs}$ and $(p_k'')_{k\geq0}$ sequence of prefixes such that
\begin{equation}\label{eq:SamePoint}\left\{\begin{array}{rcl} [i,\gamma,j]&\xrightarrow{p_0|p_0''}\cdots\xrightarrow{p_1|p_1''}\cdots&\in \mathcal{G}_0
\\
{[i',\gamma,j]}&\xrightarrow{p_0'|p_0''}\cdots\xrightarrow{p_1'|p_1''}\cdots&\in \mathcal{G}_0.
\end{array}\right.
\end{equation}

\end{lemma}
\begin{proof}
By the tiling property, a boundary point $x$ can also be written 
 $$x=\gamma+\sum_{k\geq 0} \mathbf{h}^k\pi \lBF(p_k'')$$ 
 for some $\gamma\in\pi(\mathbb{Z}^3)\setminus\{0\}$, an infinite walk $(p_k'')_{k\geq0}$ in $\Gamma$, starting from a $j\in\mathcal{A}$, with $[\gamma,j]\in\Gamma_{srs}$. Thus the lemma follows from Lemma~\ref{CharacBound}.

\end{proof}
The above characterization requires the knowledge of the boundary graph $\mathcal{G}_0$ - the subgraph $G_0$ would not be sufficient to obtain all the identifications. $\mathcal{G}_0$ is not known for our whole class of substitutions $\sigma_{a,b}$. However, in the case $2b-a\leq 3$, it was computed in our joint work~\cite{LoridantMessaoudiSurerThuswaldner13}.

\begin{proposition}[{\cite[Theorem 3.2]{LoridantMessaoudiSurerThuswaldner13}}] \label{BoundGraphProp}
Let $\sigma_{a,b}$ be the substitution~(\ref{DefSubst}), $\mathcal{G}_{0,a,b}$ the boundary graph as in Definition~\ref{DefBoundGraph} and $G_{0,a,b}$ the graph of Definition~\ref{def:G0}.  Suppose $2b-a\leq 3$. Then
$$\mathcal{G}_{0,a,b}=G_{0,a,b}.
$$
\end{proposition}

We can now characterize the disk-like tiles of our class.
\begin{proof}[Proof of Theorem~\ref{DNDTheo} (disk-like tiles)] As mentioned above, we need to check that all identifications are trivial in the parametrization, \emph{i.e.}, that infinite sequences of prefixes $(p_k)_{k\geq 0}$ and $(p_k')_{k\geq0}$ in the prefix-suffix graph $\Gamma$ satisfying 
$$\sum_{k\geq 0} \mathbf{h}^k\pi \lBF(p_k)=\sum_{k\geq 0} \mathbf{h}^k\pi \lBF(p_k')=:x\in\partial \T$$
correspond only to labels of admissible infinite walks $w$ and $w'$ in $G^+$ satisfying $\phi(w)=\phi(w')$ defined in~(\ref{def:DTmap}). The pairs of walks identified by $\phi$ are given in~(\ref{identifpairs}).

To this effect, we first look for all pairs of infinite sequences of prefixes $(p_k)\ne(p_k')$ such that  $\sum_{k\geq 0} \mathbf{h}^k\pi \lBF(p_k)=\sum_{k\geq 0} \mathbf{h}^k\pi \lBF(p_k')\in\partial \T$. This amounts to finding all the pairs of infinite admissible walks in $G_{0}=\mathcal{G}_{0}$ satisfying~(\ref{eq:SamePoint}). This can be done algorithmically by constructing an automaton $\mathcal{A}^{\psi}$, with the following states and edges.
\begin{itemize}
\item[$(i)$]  $S|S'\xrightarrow{p||{\bf o}\;|\;p||{\bf o'}}T|T'$  if and only if there is a prefix $p''$ satisfying
$$\left\{\begin{array}{rcl} S&\xrightarrow{p|p''\;||\;{\bf o}}T&\in G^+
\\
S'&\xrightarrow{p|p''\;||\;{\bf o'}}T'&\in G^+.
\end{array}\right.
$$
\item[$(ii)$]  $S|S'\xrightarrow{p||{\bf o}\;|\;p'||{\bf o'}}T||T'$ if and only if $p\ne p'$ and there is a prefix $p''$ satisfying
$$\left\{\begin{array}{rcl} S&\xrightarrow{p|p''\;||\;{\bf o}}T&\in G^+
\\
S'&\xrightarrow{p'|p''\;||\;{\bf o'}}T'&\in G^+.
\end{array}\right.
$$
\item[$(iii)$]  $S||S'\xrightarrow{p||{\bf o}\;|\;p'||{\bf o'}}T||T'$
if and only if there is a prefix $p''$ satisfying
$$\left\{\begin{array}{rcl} S&\xrightarrow{p|p''\;||\;{\bf o}}T&\in G^+
\\
S'&\xrightarrow{p'|p''\;||\;{\bf o'}}T'&\in G^+.
\end{array}\right.
$$

\end{itemize}
We call an infinite walk in $\mathcal{A}^{\psi}$ \emph{admissible} if it starts from a state $S|S'$ with $S=[i,\gamma,j]$, $S'=[i',\gamma,j]$ and $[\gamma,j]\in\Gamma_{srs}$ (possibly $S=S'$), and if it goes through at least one state of the shape $T||T'$.
Now, two sequences of the prefix-suffix automaton $\Gamma$,  $(p_k)_{k\geq 0}\ne(p_k')_{k\geq 0}$, satisfy $\sum_{k\geq 0} \mathbf{h}^k\pi \lBF(p_k)=\sum_{k\geq 0} \mathbf{h}^k\pi \lBF(p_k')\in\partial \T$ if and only if there is an admissible walk in $\mathcal{A}^{\psi}$ labelled by $(p_k||{\bf o_k}\;|\;p_k'||{\bf o_k'})_{k\geq 0}$. 

After deleting the states and edges that do not belong to an admissible walk, we get the automaton of Figure~\ref{fig:APSIab} for the case $a\geq b+1,b\geq 2$. Note that for $a=b+1$ or $b=2$, the automaton becomes lighter, as several edges disappear. The starting states for admissible walks are colored. For the sake of simplicity,  we did not depict the edges of the form $S|S\xrightarrow{p||{\bf o}\;|\;p||{\bf o}}T|T$ (particular case of $(i)$). Therefore, the states $S|S$ in these figures may be preceded by a finite walk made of such edges and ending in $S|S$. The remaining cases can be found in the Appendix~\ref{app:TheoDisk}.

Second, we look for all pairs of infinite admissible walks $w\ne w'$ of $G^+$ such that $P(w)$ and $P(w')$ carry the same infinite sequence of prefixes $(p_k)_{k\geq 0}$: the parameters $t,t'\in[0,1]$ for such walks trivially map to the same boundary point by the parametrization $C$. Again, these pairs of walks can be obtained algorithmically via an automaton $\mathcal{A}^{sl}$ with the following states and edges.
\begin{itemize}
\item[$(iv)$] $S|S\xrightarrow{p\;|\;{\bf o}||{\bf o}}T|T$  if and only if $S\xrightarrow{p\;||\;{\bf o}}T\in G^+$.
\item[$(v)$]  $S|S\xrightarrow{p\;|\;{\bf o}||{\bf o'}}T||T'$ if and only if  ${\bf o}\ne{\bf o'}$  and 
$$\left\{\begin{array}{rcl} S&\xrightarrow{p||{\bf o}}T&\in G^+
\\
S&\xrightarrow{p||{\bf o'}}T'&\in G^+.
\end{array}\right.
$$
\item[$(vi)$]   $S||S'\xrightarrow{p\;|\;{\bf o}||{\bf o'}}T||T'$ if and only if  and 
$$\left\{\begin{array}{rcl} S&\xrightarrow{p||{\bf o}}T&\in G^+
\\
S'&\xrightarrow{p||{\bf o'}}T'&\in G^+.
\end{array}\right.
$$

\end{itemize}
We call an infinite walk in $\mathcal{A}^{sl}$ \emph{admissible} if it starts from a state $S|S$ or $S||S'$ with $S=[i,\gamma,j]$, $S'=[i',\gamma,j]$ and $[\gamma,j]\in\Gamma_{srs}$ ($S\ne S'$), and if it goes through at least one state of the shape $T||T'$. Now, for two admissible infinite walks of $G^+$:
$$w=(S;{\bf o_1},{\bf o_2},\ldots)\;\ne\; w'=(S';{\bf o_1'},{\bf o_2'},\ldots),$$
$P(w)$ and $P(w')$ carry the same sequence of prefixes $(p_k)_{k\geq 0}$  if and only if there is an admissible walk in $\mathcal{A}^{sl}$ labelled by $(p_k\;|\;{\bf o_k}||{\bf o_k'})_{k\geq 0}$. 

After deleting the states and edges that do not belong to an admissible walk, we get the automaton of Figure~\ref{fig:ASLab} for the case $a\geq b+1,b\geq 2$. The starting states for admissible walks are colored. For the sake of simplicity,  we did not depict the edges of Item $(iv)$. For the remaining cases, see Appendix~\ref{app:TheoDisk}.

Note that in $\mathcal{A}^{\psi}$ as well as in $\mathcal{A}^{sl}$, no more pairs $(w,w')$ than the pairs given in~(\ref{identifpairs}) are found. Therefore, we conclude that for all $w,w'\in G^+$
$$\psi(P(w))=\psi(P(w'))\Rightarrow \phi(w)=\phi(w').
$$
Consequently, the parametrization $C:[0,1]\to\partial \T$ is injective, apart from $C(0)=C(1)$. Hence $\partial T$ is a simple closed curve and, by a theorem of Sch\"onflies~\cite{Whyburn79}, $\T$ is homeomorphic to a closed disk. 
\end{proof}

\begin{figure}
\begin{center}
\includegraphics[width=150mm,height=130mm]{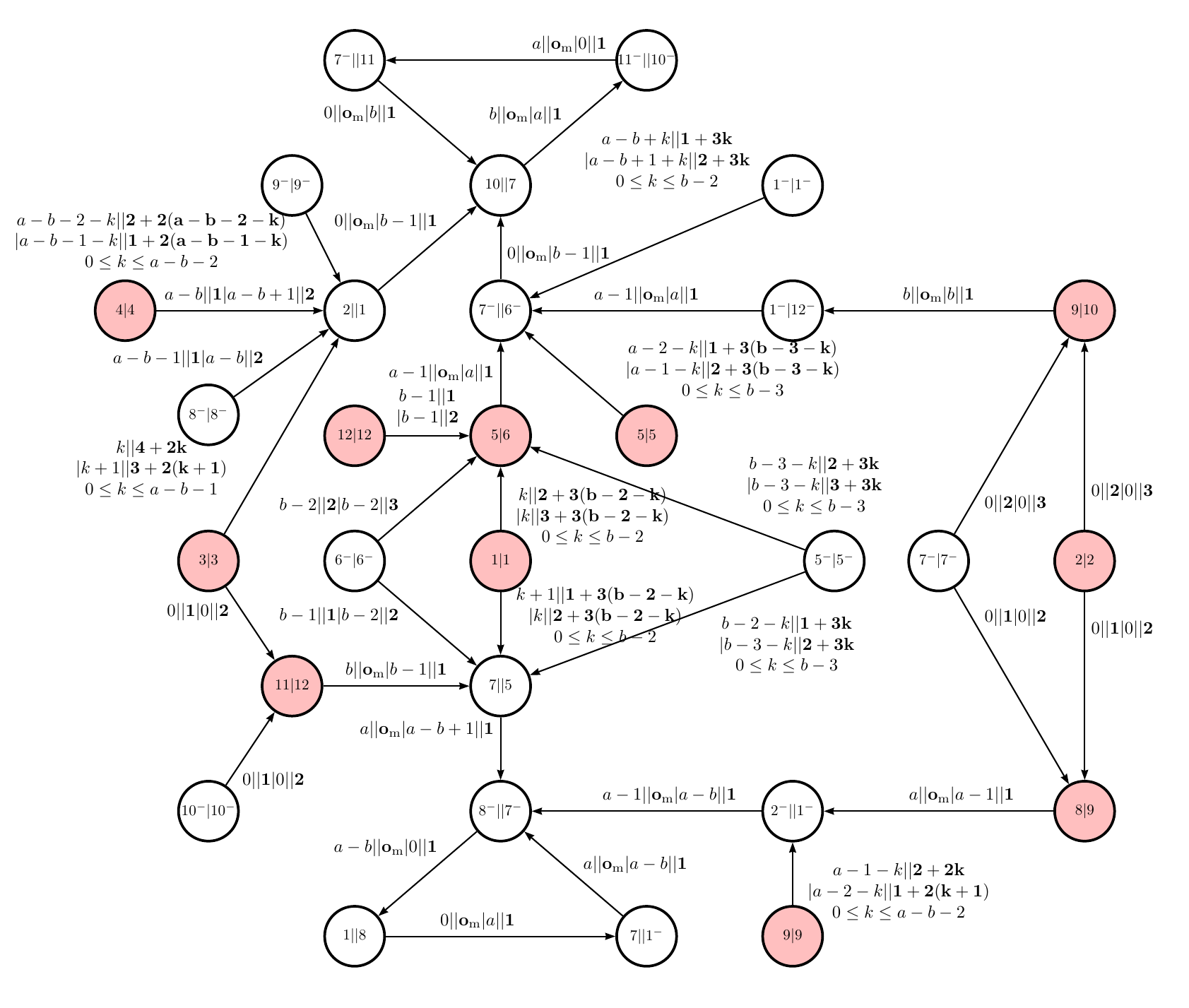}
\end{center}
\caption{$\mathcal{A}^\psi$ for $a\geq b+1,b\geq 2$ (${\bf o_{\textrm{m}}}$ stands for ${\bf o_{\textrm{max}}}$).}\label{fig:APSIab}
\end{figure}

\begin{figure}
\begin{center}
\includegraphics[width=150mm,height=120mm]{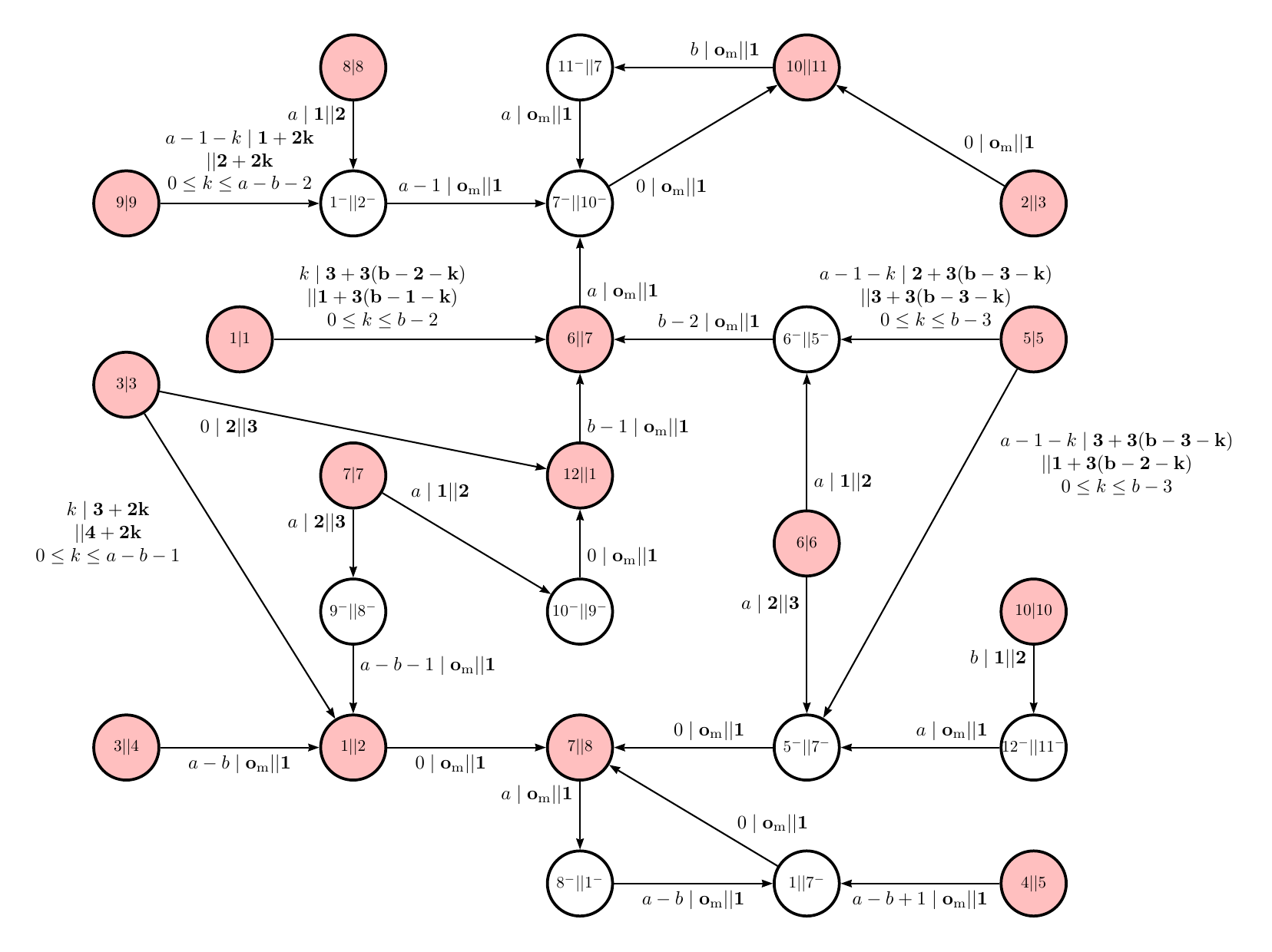}
\end{center}
\caption{$\mathcal{A}^{sl}$ for $a\geq b+1,b\geq 2$ (${\bf o_{\textrm{m}}}$ stands for ${\bf o_{\textrm{max}}}$).}\label{fig:ASLab}
\end{figure}

\end{section}

\begin{section}{Concluding remarks}\label{sec:conc}

Other projects using the parametrization method may concern the topological study of further classes of substitutions, for example families of Arnoux-Rauzy substitutions. These substitutions are of the form $\sigma=\tau_1\cdots\tau_r$, where $r\geq 3$ and $\{\tau_1,\ldots,\tau_r\}=\{\sigma_1,\sigma_2,\sigma_3\}$ ($\sigma_i$ are the Arnoux-Rauzy substitutions). For the moment, the connectedness of the associated Rauzy fractals could be obtained (see~\cite{BertheJolivetSiegel13}), but the classification disk-like/non-disk-like is still outstanding.

Another challenge is the study of non-disk-like tiles, which happens to be rather difficult. A criterion~\cite{SiegelThuswaldner10} allows to decide whether  the fundamental group is  trivial or uncountable, but more precise descriptions are not known. For given examples of self-affine tiles, the description of cut points and of connected components could be achieved (see~\cite{NgaiNguyen03,BernatLoridantThuswaldner10}). We can understand the degree of difficulty of these studies via the following considerations. In our framework, non-disk-likeness implies non-trivial identifications in the parametrization and requires to  find out non-injective points of the parametrization. To speak roughly, we need the computation of the language of $\mathcal{G}_0\setminus G_0$. Therefore, this is related to the complementation of B\"uchi automata, which is known to be a difficult problem (\cite{Thomas90,PerrinPin04}).  Note that we have here the tools to complete such a study. Similarly to~\cite[Section 4]{AkiyamaLoridant11} and as in the above proof of Theorem~\ref{DNDTheo} (disk-like tiles), we can define three automata whose edges take the form 
$$S|S'\xrightarrow{p||{\bf o}\;|\;p'||{\bf o'}}T|T',
$$
where $S\xrightarrow{p||{\bf o}}T$ and $S'\xrightarrow{p'||{\bf o'}}T'$ are edges of $\mathcal{G}_0$. A first automaton $\mathcal{A}^\phi$ gives the walks identified by the Dumont-Thomas numeration system $\phi$, \emph{i.e.},  the pairs $(w,w')\in G^+$ given in~(\ref{identifpairs}). In the disk-like case, no other walks are identified. The second automaton $\mathcal{A}^\psi$ gives pairs of walks $(w,w')$ identified by $\psi$ and is computed via Lemma~\ref{lem:EqAd}. The third automaton  $\mathcal{A}^{sl}$ gives the pairs of walks $(w,w')$ carrying the same sequence of prefixes. Topological information might be read off from the automaton $\mathcal{A}^\psi\cup \mathcal{A}^{sl}\;\setminus\;  \mathcal{A}^\phi$.

\end{section}

\appendix
\section{Details for the proof of Proposition~\ref{prop:PsiIdentif}.}\label{app:Continuity}
We check Conditions~(\ref{Cond1}),~(\ref{Cond2}) and~(\ref{Cond3}) of Proposition~\ref{prop:PsiIdentif}.  The ideas were given in the proof after the statement of this proposition. We sum up the computations in Tables~\ref{table:abCont} to~\ref{table:a1b1Cont}. As mentioned, some computations require the use of Lemma~\ref{CharacBound}. The last column refers to the items below for these special cases. In the cases below, we give the pairs of  infinite walks in $G_0\subset\mathcal{G}_0$:
$$\left\{\begin{array}{c}S=[i,\gamma,j]\xrightarrow{p_0|p'_0}S_1\xrightarrow{p_1|p_1'}S_2\xrightarrow{p_2|p_2'}\cdots\\\textrm{ and }\\
S'=[i',\gamma,j']\xrightarrow{p_0''|p_0'}S_1'\xrightarrow{p_1''|p_1'}S_2'\xrightarrow{p_2''|p_2'}\cdots,
\end{array}\right.
$$
 for which holds $$\sum_{k\geq 0}{\bf h}^k\pi {\bf l}(p_k)=\gamma+\sum_{k\geq 0}{\bf h}^k\pi {\bf l}(p_k')=\sum_{k\geq 0}{\bf h}^k\pi {\bf l}(p_k'')
$$
by Lemma~\ref{CharacBound}. The concatenation of walks, using the symbol $\&$, was defined in~(\ref{def:concatenation}).

\textbf{Case} $a\geq b+1,\;b\geq 2$. Note that the states $S=11,11^-$ have only one outgoing edge, thus it does not show up in the checking of Conditon~(\ref{Cond3}). This happens also with the states $S=5,5^-$ whenever $b=2$, and $S=9,9^-$ whenever $a=b+1$.
\begin{itemize}
\item[(1)] See proof of Proposition~\ref{prop:PsiIdentif}.
\item[(2)] $$\left\{\begin{array}{c}8\xrightarrow{p_0=a|p'_0=a-b-1}2^-\xrightarrow{p_1=a-1|p'_1=0}\overline{8^-\xrightarrow{p_2=a-b|p_2'=a}1\xrightarrow{p_3=0|p_3'=a-b}7\xrightarrow{p_4=a|p_4'=0}8^-}\\\\
9\xrightarrow{p_0''=a-1|p_0'=a-b-1}\overline{1^-\xrightarrow{p_1''=a-b|p_1'=0}7^-\xrightarrow{p_2''=0|p_2'=a}8\xrightarrow{p_3''=a|p_3'=a-b}1^-}.
\end{array}\right.
$$
\item[(3)] $$\left\{\begin{array}{c}11\xrightarrow{p_0=b|p'_0=a}\overline{7\xrightarrow{p_1=a|p_1'=0}8^-\xrightarrow{p_2=a-b|p_2'=a}1\xrightarrow{p_3=0|p_3'=a-b}7}\\\\
12\xrightarrow{p_0''=b-1|p_0'=a}5\xrightarrow{p_1=a-b+1|p'_1=0}\overline{7^-\xrightarrow{p_2''=0|p_2'=a}8\xrightarrow{p_3''=a|p_3'=a-b}1^-\xrightarrow{p_4''=a-b|p_4'=0}7^-}.
\end{array}\right.
$$

\item[(4)] For $0\leq k\leq b-2$,
$$\left\{\begin{array}{c}1\xrightarrow{p_0=b-1-k|p'_0=a-1-k}\overline{7\xrightarrow{p_1=a|p_1'=0}8^-\xrightarrow{p_2=(a-b)|p_2'=a}1\xrightarrow{p_3=0|p_3'=a-b}7}\\\\
1\xrightarrow{p_0''=b-2-k|p_0'=a-1-k}5\xrightarrow{p_1=a-b+1|p'_1=0}\overline{7^-\xrightarrow{p_2''=0|p_2'=a}8\xrightarrow{p_3''=a|p_3'=a-b}1^-\xrightarrow{p_4''=a-b|p_4'=0}7^-}.
\end{array}\right.
$$
\item[(5)]For $0\leq k\leq a-b-1$,
$$\left\{\begin{array}{c}
3\xrightarrow{p_0''=k|p_0'=b+k}2\xrightarrow{p_1=0|p'_1=a-1}\overline{10\xrightarrow{p_2''=b|p_2'=0}11^-\xrightarrow{p_3''=a|p_3'=b}7^-\xrightarrow{p_4''=0|p_4'=a}10}\\\\
3\xrightarrow{p_0=k+1|p'_0=b+k}1\xrightarrow{p_1=b-1|p'_1=a-1}\overline{7\xrightarrow{p_2=a|p_2'=0}10^-\xrightarrow{p_3=0|p_3'=b}11\xrightarrow{p_4=b|p_4'=a}7}.
\end{array}\right.
$$
\item[(6)] $$\left\{\begin{array}{c}4\xrightarrow{p_0=a-b|p'_0=a}2\xrightarrow{p_1=0|p'_1=a-1}\overline{10\xrightarrow{p_2=b|p_2'=0}11^-\xrightarrow{p_3=a|p_3'=b}7^-\xrightarrow{p_4=0|p_4'=a}10}\\\\
4\xrightarrow{p_0''=a-b+1|p_0'=a}1\xrightarrow{p_1=b-1|p'_1=a-1}\overline{7\xrightarrow{p_2''=a|p_2'=0}10^-\xrightarrow{p_3''=0|p_3'=b}11\xrightarrow{p_4''=b|p_4'=a}7}.
\end{array}\right.
$$
\item[(7)] For $0\leq k\leq a-b-2$,
$$\left\{\begin{array}{c}9\xrightarrow{p_0=a-1-k|p'_0=a-b-2-k}2^-\xrightarrow{p_1=a-1|p'_1=0}\overline{8^-\xrightarrow{p_2=a-b|p_2'=a}1\xrightarrow{p_3=0|p_3'=a-b}7\xrightarrow{p_4=a|p_4'=0}8^-}\\\\
9\xrightarrow{p_0''=a-2-k|p_0'=a-b-2-k}\overline{1^-\xrightarrow{p_1=a-b|p'_1=0}7^-\xrightarrow{p_2''=0|p_2'=a}8\xrightarrow{p_3''=a|p_3'=a-b}1^-}.
\end{array}\right.
$$

\item[(8)]For $0\leq k\leq b-2$,
$$\left\{\begin{array}{c}1^-\xrightarrow{p_0=a-b+k|p'_0=k}\overline{7^-\xrightarrow{p_1=0|p'_1=a}10\xrightarrow{p_2=b|p_2'=0}11^-\xrightarrow{p_3=a|p_3'=b}7^-}\\\\
1^-\xrightarrow{p_0''=a-b+k+1|p_0'=k}6^-\xrightarrow{p_1=b-1|p'_1=a}\overline{7\xrightarrow{p_2''=a|p_2'=0}10^-\xrightarrow{p_3''=0|p_3'=b}11\xrightarrow{p_4''=b|p_4'=a}7}.
\end{array}\right.
$$

\item[(9)] For $0\leq k\leq b-3$,
$$\left\{\begin{array}{c}5^-\xrightarrow{p_0=b-2-k|p'_0=a-1-k}\overline{7\xrightarrow{p_1=a|p'_1=0}8^-\xrightarrow{p_2=a-b|p_2'=a}1\xrightarrow{p_3=0|p_3'=a-b}7}\\\\
5^-\xrightarrow{p_0''=b-3-k|p_0'=a-1-k}5\xrightarrow{p_1=a-b+1|p'_1=0}\overline{7^-\xrightarrow{p_2''=0|p_2'=a}8\xrightarrow{p_3''=a|p_3'=a-b}1^-\xrightarrow{p_4''=a-b|p_4'=0}7^-}.
\end{array}\right.
$$

\item[(10)] 
$$\left\{\begin{array}{c}6^-\xrightarrow{p_0=b-1|p'_0=a}\overline{7\xrightarrow{p_1=a|p'_1=0}8^-\xrightarrow{p_2=a-b|p_2'=a}1\xrightarrow{p_3=0|p_3'=a-b}7}\\\\
6^-\xrightarrow{p_0''=b-2|p_0'=a}5\xrightarrow{p_1=a-b+1|p'_1=0}\overline{7^-\xrightarrow{p_2''=0|p_2'=a}8\xrightarrow{p_3''=a|p_3'=a-b}1^-\xrightarrow{p_4''=a-b|p_4'=0}7^-}.
\end{array}\right.
$$

\item[(11)] $$\left\{\begin{array}{c}8^-\xrightarrow{p_0=a-b-1|p'_0=a}2\xrightarrow{p_1=0|p'_1=a-1}\overline{10\xrightarrow{p_2=b|p'_2=0}11^-\xrightarrow{p_3=a|p_3'=b}7^-\xrightarrow{p_4=0|p_4'=a}10}\\\\
8^-\xrightarrow{p_0''=a-b|p_0'=a}1\xrightarrow{p_1=b-1|p'_1=a-1}\overline{7\xrightarrow{p_2''=a|p_2'=0}10^-\xrightarrow{p_3''=0|p_3'=b}11\xrightarrow{p_4''=b|p_4'=a}7}.
\end{array}\right.
$$

\end{itemize}

\textbf{Case} $a\geq 2,\;b=1$. We take advantage of the similarities with the preceding case (compare the graph  of Figure~\ref{fig:Gab1}, Table~\ref{Gab1} with the graph of Figure~\ref{fig:Gab}, Table~\ref{Gab} when taking $b=1$). 

Conditions~(\ref{Cond1}) is checked as in Table~\ref{table:abCont} for $S=1,2,3,6,7,8,10$ and the referred items by taking $b=1$. 

Condition~(\ref{Cond3}) is checked as in Tables~\ref{table:abCont2A} and~\ref{table:abCont2B} for 
$$\begin{array}{rcl}
(S;{\bf o})&\in&\left\{(2;{\bf 1}),(3;{\bf 3+2k}),(3;{\bf 4+2k}),(7;{\bf 1}),(7;{\bf 2}),(8;{\bf 1}),(9;{\bf 1}),(9;{\bf 2}),\right.\\
&&\left.(2^-;{\bf 1}),(2^-;{\bf 2}),(7^-;{\bf 1}),(9^-;{\bf 1+2k}),(9^-;{\bf 2+2k})\right\}.
\end{array}
$$ 
by taking $b=1$. Note that the states $S=1,4,5,6,11,12,1^-,6^-,11^-,12^-$ have only one outgoing edge, that is why they do not show up in the checking of Conditon~(\ref{Cond3}). This also happens for $S=9,9^-$, whenever $a=2$.

The remaining cases and Condition~(\ref{Cond2}) are presented in Table~\ref{table:ab1Cont}, with references to the items below when the use of Lemma~\ref{CharacBound} is required.

\begin{itemize}
\item[(12)] $$\left\{\begin{array}{c}4\xrightarrow{p_0=a-1|p'_0=a}2\xrightarrow{p_1=0|p'_1=a-1}\overline{10\xrightarrow{p_2=1|p'_2=0}11^-\xrightarrow{p_3=a|p_3'=1}7^-\xrightarrow{p_4=0|p_4'=a}10}\\\\
5\xrightarrow{p_0''=a|p_0'=a}1\xrightarrow{p_1=0|p'_1=a-1}\overline{7\xrightarrow{p_2''=a|p_2'=0}10^-\xrightarrow{p_3''=0|p_3'=1}11\xrightarrow{p_4''=1|p_4'=a}7}.
\end{array}\right.
$$
\item[(13)] $$\left\{\begin{array}{c}9\xrightarrow{p_0=1|p'_0=0}1^-\xrightarrow{p_1=a-1|p'_1=0}\overline{7^-\xrightarrow{p_2=0|p'_2=a}10\xrightarrow{p_3=1|p_3'=0}11^-\xrightarrow{p_4=a|p_4'=1}7^-}\\\\
10\xrightarrow{p_0''=1|p_0'=0}12^-\xrightarrow{p_1=a|p'_1=0}6^-\xrightarrow{p_2''=0|p_2'=a}\overline{7\xrightarrow{p_3''=a|p_3'=0}10^-\xrightarrow{p_4''=0|p_4'=1}11\xrightarrow{p_5''=1|p_5'=a}7}.
\end{array}\right.
$$

\item[(14)]  $$\left\{\begin{array}{c}11\xrightarrow{p_0=1|p'_0=a}\overline{7\xrightarrow{p_1=a|p_1'=0}8^-\xrightarrow{p_2=a-1|p_2'=a}1\xrightarrow{p_3=0|p_3'=a-1}7}\\\\
12\xrightarrow{p_0''=0|p_0'=a}6\xrightarrow{p_1=a|p'_1=0}\overline{7^-\xrightarrow{p_2''=0|p_2'=a}8\xrightarrow{p_3''=a|p_3'=a-1}1^-\xrightarrow{p_4''=a-1|p_4'=0}7^-}.
\end{array}\right.
$$

\end{itemize}

\textbf{Case} $a=b\geq 2$.  This case has similarities with the case $a\geq b+1,b\geq 2$. However, the number of starting states reduces to $11$. We present the results in Tables~\ref{table:aeqbCont} to~\ref{table:aeqbCont2B}. Note that the states $S=8,10,8^-,10^-$ have only one outgoing edge, thus they do not show up in the checking of Conditon~(\ref{Cond3}). This also happens for $S=5,5^-$, whenever $a=2$.
\begin{itemize}
\item[(15)] $$\left\{\begin{array}{c}5\xrightarrow{p_0=a-1|p'_0=a-2}\overline{7^{-}\xrightarrow{p_1=0|p_1'=a}9\xrightarrow{p_2=a|p_2'=0}11^{-}\xrightarrow{p_3=a|p_3'=a}7^-}\\\\
6\xrightarrow{p_0''=a|p_0'=a-2}6^-\xrightarrow{p_1''=a-1|p_1'=a}\overline{7\xrightarrow{p_2''=a|p_2'=0}9^{-}\xrightarrow{p_3''=0|p_3'=a}10\xrightarrow{p_4''=a|p_4'=a}7},
\end{array}\right.
$$

\item[(16)]  $$\left\{\begin{array}{c}10\xrightarrow{p_0=a|p'_0=a}\overline{7\xrightarrow{p_1=a|p_1'=0}8^-\xrightarrow{p_2=0|p_2'=a}1\xrightarrow{p_3=0|p_3'=0}7}\\\\
11\xrightarrow{p_0''=a-1|p_0'=a}5\xrightarrow{p_1=1|p'_1=0}\overline{7^-\xrightarrow{p_2''=0|p_2'=a}8\xrightarrow{p_3''=a|p_3'=0}1^-\xrightarrow{p_4''=0|p_4'=0}7^-}.
\end{array}\right.
$$

\item[(17)]For $0\leq k\leq a-2$,
$$\left\{\begin{array}{c}1\xrightarrow{p_0=a-1-k|p'_0=a-1-k}\overline{7\xrightarrow{p_1=a|p_1'=0}8^-\xrightarrow{p_2=0|p_2'=a}1\xrightarrow{p_3=0|p_3'=0}7}\\\\
1\xrightarrow{p_0''=a-2-k|p_0'=a-1-k}5\xrightarrow{p_1=1|p'_1=0}\overline{7^-\xrightarrow{p_2''=0|p_2'=a}8\xrightarrow{p_3''=a|p_3'=0}1^-\xrightarrow{p_4''=0|p_4'=0}7^-}.
\end{array}\right.
$$

\item[(18)] $$\left\{\begin{array}{c}4\xrightarrow{p_0=0|p'_0=a}2\xrightarrow{p_1=0|p'_1=a-1}\overline{8\xrightarrow{p_2=a|p_2'=0}9^-\xrightarrow{p_3=a|p_3'=a}7^-\xrightarrow{p_4=0|p_4'=a}9}\\\\
4\xrightarrow{p_0''=1|p_0'=a}1\xrightarrow{p_1=a-1|p'_1=a-1}\overline{7\xrightarrow{p_2''=a|p_2'=0}9^-\xrightarrow{p_3''=0|p_3'=a}10\xrightarrow{p_4''=a|p_4'=a}7}.
\end{array}\right.
$$

\item[(19)] For $0\leq k\leq a-2$,
$$\left\{\begin{array}{c}1^-\xrightarrow{p_0=k|p'_0=k}\overline{7^-\xrightarrow{p_1=0|p'_1=a}9\xrightarrow{p_2=a|p_2'=0}10^-\xrightarrow{p_3=a|p_3'=a}7^-}\\\\
1^-\xrightarrow{p_0''=k+1|p_0'=k}6^-\xrightarrow{p_1=a-1|p'_1=a}\overline{7\xrightarrow{p_2''=a|p_2'=0}9^-\xrightarrow{p_3''=0|p_3'=a}10\xrightarrow{p_4''=a|p_4'=a}7}.
\end{array}\right.
$$

\item[(20)] For $0\leq k\leq a-3$,
$$\left\{\begin{array}{c}5^-\xrightarrow{p_0=a-2-k|p'_0=a-1-k}\overline{7\xrightarrow{p_1=a|p'_1=0}8^-\xrightarrow{p_2=0|p_2'=a}1\xrightarrow{p_3=0|p_3'=0}7}\\\\
5^-\xrightarrow{p_0''=a-3-k|p_0'=a-1-k}5\xrightarrow{p_1=1|p'_1=0}\overline{7^-\xrightarrow{p_2''=0|p_2'=a}8\xrightarrow{p_3''=a|p_3'=0}1^-\xrightarrow{p_4''=0|p_4'=0}7^-}.
\end{array}\right.
$$

\item[(21)] $$\left\{\begin{array}{c}6^-\xrightarrow{p_0=a-1|p'_0=a}\overline{7\xrightarrow{p_1=a|p'_1=0}8^-\xrightarrow{p_2=0|p_2'=a}1\xrightarrow{p_3=0|p_3'=0}7}\\\\
6^-\xrightarrow{p_0''=a-2|p_0'=a}5\xrightarrow{p_1=1|p'_1=0}\overline{7^-\xrightarrow{p_2''=0|p_2'=a}8\xrightarrow{p_3''=a|p_3'=0}1^-\xrightarrow{p_4''=0|p_4'=0}7^-}.
\end{array}\right.
$$

\end{itemize}

\textbf{Case}  $a=b=1$. There are similarities with the preceding case $a=b\geq 2$ (compare the graph  of Figure~\ref{fig:Gaeqb}, Table~\ref{Gaeqb} with the graph of Figure~\ref{fig:Ga1b1}, Table~\ref{Ga1b1} when taking $a=1$). 

Conditions~(\ref{Cond1}) and Condition~(\ref{Cond2}) are checked as in Table~\ref{table:aeqbCont} for $S=1,2,3,6,7,8,9,11$ and the referred items by taking $a=1$. 

Condition~(\ref{Cond3}) is checked as in Tables~\ref{table:aeqbCont2A} and~\ref{table:aeqbCont2B} for 
$$\begin{array}{rcl}
(S;{\bf o})&\in&\left\{(2;{\bf 1}),(3;{\bf 2}),(7;{\bf 1}),(7^-;{\bf 1})\right\}.
\end{array}
$$ 
by taking $a=1$. Note that the states $S=1,4,5,6,8,10,11,1^-,6^-,8^-,10^-,11^-$ have only one outgoing edge, that is why they do not show up in the checking of Conditon~(\ref{Cond3}). 

The remaining cases are presented in Table~\ref{table:a1b1Cont}, with references to the items below when the use of Lemma~\ref{CharacBound} is required.
\begin{itemize}
\item[(22)]
$$\left\{\begin{array}{c}4\xrightarrow{p_0=0|p'_0=1}2\xrightarrow{p_1=0|p'_1=0}\overline{9\xrightarrow{p_2=1|p'_2=0}10^-\xrightarrow{p_3=1|p_3'=1}7^-\xrightarrow{p_4=0|p_4'=1}9}\\\\
5\xrightarrow{p_0''=1|p_0'=1}1\xrightarrow{p_1=0|p'_1=0}\overline{7\xrightarrow{p_2''=1|p_2'=0}9^-\xrightarrow{p_3''=0|p_3'=1}10\xrightarrow{p_4''=1|p_4'=1}7}.
\end{array}\right.
$$
\item[(23)]
$$\left\{\begin{array}{c}10\xrightarrow{p_0=1|p'_0=1}\overline{7\xrightarrow{p_1=1|p_1'=0}8^-\xrightarrow{p_2=0|p_2'=1}1\xrightarrow{p_3=0|p_3'=0}7}\\\\
11\xrightarrow{p_0''=0|p_0'=1}6\xrightarrow{p_1=1|p'_1=0}\overline{7^-\xrightarrow{p_2''=0|p_2'=1}8\xrightarrow{p_3''=1|p_3'=0}1^-\xrightarrow{p_4''=0|p_4'=0}7^-}.
\end{array}\right.
$$

\end{itemize}

\begin{table}[h!]\scriptsize\begin{center}
\begin{tabular}{|c|c||c|c|c|}
\hline
 Walk 1& Walk 2& Sequence 1& Sequence 2 & Checking via Lemma~\ref{CharacBound}: see Item\ldots\\
\hline &&&&\\
$(1;\overline{{\bf o_{\textrm{max}}}})$&$(2;\overline{{\bf 1}})$&$\overline{0a(a-b)}$&$0\overline{a(a-b)0}$&
\\\hline&&&&\\
$(2;\overline{{\bf o_{\textrm{max}}}})$&$(3;\overline{{\bf 1}})$&$0\overline{ba0}$&$0\overline{ba0}$&
\\\hline&&&&\\
$(3;\overline{{\bf o_{\textrm{max}}}})$&$(4;\overline{{\bf 1}})$&$(a-b)\overline{0a(a-b)}$&$(a-b)0\overline{a(a-b)0}$&
\\\hline&&&&\\
$(4;\overline{{\bf o_{\textrm{max}}}})$&$(5;\overline{{\bf 1}})$&$(a-b+1)\overline{0a(a-b)}$&$(a-b+1)\overline{0a(a-b)}$&
\\\hline&&&&\\
$(5;\overline{{\bf o_{\textrm{max}}}})$&$(6;\overline{{\bf 1}})$&$(a-1)\overline{0ba}$&$a(b-1)\overline{a0b}$& (1)
\\\hline&&&&\\
$(6;\overline{{\bf o_{\textrm{max}}}})$&$(7;\overline{{\bf 1}})$&$a\overline{0ba}$&$\overline{a0b}$&
\\\hline&&&&\\
$(7;\overline{{\bf o_{\textrm{max}}}})$&$(8;\overline{{\bf 1}})$&$\overline{a(a-b)0}$&$\overline{a(a-b)0}$&
\\\hline&&&&\\
$(8;\overline{{\bf o_{\textrm{max}}}})$&$(9;\overline{{\bf 1}})$&$a(a-1)\overline{(a-b)0a}$&$(a-1)\overline{(a-b)0a}$& (2)
\\\hline&&&&\\
$(9;\overline{{\bf o_{\textrm{max}}}})$&$(10;\overline{{\bf 1}})$&$b(a-1)\overline{0ba}$&$ba(b-1)\overline{a0b}$&(1)
\\\hline&&&&\\
$(10;\overline{{\bf o_{\textrm{max}}}})$&$(11;\overline{{\bf 1}})$&$\overline{ba0}$&$\overline{ba0}$&
\\\hline&&&&\\
$(11;\overline{{\bf o_{\textrm{max}}}})$&$(12;\overline{{\bf 1}})$&$b\overline{a(a-b)0}$&$(b-1)(a-b+1)\overline{0a(a-b)}$& (3)
\\\hline&&&&\\
$(12;\overline{{\bf o_{\textrm{max}}}})$&$(1;\overline{{\bf 1}})$&$(b-1)a\overline{0ba}$&$(b-1)\overline{(a0b)}$&
\\\hline
\end{tabular}
\caption{Case $a\geq b+1,\;b\geq 2$, Conditions~(\ref{Cond1}) and~(\ref{Cond2}).}\label{table:abCont}
\end{center}
\end{table}

\begin{table}[h!]\scriptsize\begin{center}
\begin{tabular}{|c|c||c|c|c|}
\hline
 Walk 1& Walk 2& Sequence 1& Sequence 2 & Item\\
\hline &&&&\\
$\begin{array}{c}(1;{\bf 1+3k},\overline{{\bf o_{\textrm{max}}}})\\=(1;{\bf 1+3k})\&(7;\overline{{\bf o_{\textrm{max}}}}),\\k=0,\ldots,b-2\end{array}$&$\begin{array}{c}(1;{\bf 2+3k},\overline{{\bf 1}})\\=(1;{\bf 2+3k})\&(5;\overline{{\bf1}})\end{array}$&$(b-1-k)\overline{a(a-b)0}$&$(b-2-k)(a-b+1)\overline{0a(a-b)}$&(4)
\\\hline&&&&\\
$\begin{array}{c}(1;{\bf 2+3k},\overline{{\bf o_{\textrm{max}}}})\\=(1;{\bf 2+3k})\&(5;\overline{{\bf o_{\textrm{max}}}}),\\k=0,\ldots,b-2\end{array}$&$\begin{array}{c}(1;{\bf 3+3k},\overline{{\bf 1}})\\=(1;{\bf 3+3k})\&(6;\overline{{\bf1}})\end{array}$&$(b-2-k)(b-1)\overline{0ba}$&$(b-2-k)a(b-1)\overline{a0b}$& (1)
\\\hline&&&&\\
$\begin{array}{c}(1;{\bf 3+3k},\overline{{\bf o_{\textrm{max}}}})\\=(1;{\bf 3+3k})\&(6;\overline{{\bf o_{\textrm{max}}}}),\\k=0,\ldots,b-2\end{array}$&$\begin{array}{c}(1;{\bf 1+3(k+1)},\overline{{\bf 1}})\\=(1;{\bf 1+3(k+1)})\&(7;\overline{{\bf1}})\end{array}$&$(b-2-k)\overline{a0b}$&$(b-2-k)\overline{a0b}$&
\\\hline&&&&\\
$\begin{array}{c}(2;{\bf 1},\overline{{\bf o_{\textrm{max}}}})\\=(2;{\bf 1})\&(8;\overline{{\bf o_{\textrm{max}}}})\end{array}$&$\begin{array}{c}(2;{\bf 2},\overline{{\bf 1}})\\=(2;{\bf 2})\&(9;\overline{{\bf 1}})\end{array}$&$0a(a-1)\overline{(a-b)0a}$&$0(a-1)\overline{(a-b)0a}$& (2)
\\\hline&&&&\\
$\begin{array}{c}(2;{\bf 2},\overline{{\bf o_{\textrm{max}}}})\\=(2;{\bf 2})\&(9;\overline{{\bf o_{\textrm{max}}}})\end{array}$&$\begin{array}{c}(2;{\bf 3},\overline{{\bf 1}})\\=(2;{\bf 3})\&(10;\overline{{\bf 1}})\end{array}$&$0b(a-1)\overline{0ba}$&$0ba(b-1)\overline{a0b}$& (1)
\\\hline&&&&\\
$\begin{array}{c}(3;{\bf 1},\overline{{\bf o_{\textrm{max}}}})\\=(3;{\bf 1})\&(11;\overline{{\bf o_{\textrm{max}}}})\end{array}$&$\begin{array}{c}(3;{\bf 2},\overline{{\bf 1}})\\=(3;{\bf 2})\&(12;\overline{{\bf 1}})\end{array}$&$0b\overline{a(a-b)0}$&$0(b-1)(a-b+1)\overline{0a(a-b)}$& (3)
\\\hline&&&&\\
$\begin{array}{c}(3;{\bf 2},\overline{{\bf o_{\textrm{max}}}})\\=(3;{\bf 2})\&(12;\overline{{\bf o_{\textrm{max}}}})\end{array}$&$\begin{array}{c}(3;{\bf 3},\overline{{\bf 1}})\\=(3;{\bf 3})\&(1;\overline{{\bf 1}})\end{array}$&$0(b-1)a\overline{0ba}$&$0(b-1)\overline{a0b}$&
\\\hline&&&&\\
$\begin{array}{c}(3;{\bf 3+2k},\overline{{\bf o_{\textrm{max}}}})\\=(3;{\bf 3+2k})\&(1;\overline{{\bf o_{\textrm{max}}}}),\\k=0,\ldots,a-b-1\end{array}$&$\begin{array}{c}(3;{\bf 4+2k},\overline{{\bf 1}})\\=(3;{\bf 4+2k})\&(2;\overline{{\bf1}})\end{array}$&$k\overline{0a(a-b)}$&$k\overline{0a(a-b)}$&
\\\hline&&&&\\
$\begin{array}{c}(3;{\bf 4+2k},\overline{{\bf o_{\textrm{max}}}})\\=(3;{\bf 4+2k})\&(2;\overline{{\bf o_{\textrm{max}}}}),\\k=0,\ldots,a-b-1\end{array}$&$\begin{array}{c}(3;{\bf 3+2(k+1)},\overline{{\bf 1}})\\=(3;{\bf 3+2(k+1)})\&(1;\overline{{\bf1}})\end{array}$&$k0\overline{ba0}$&$(k+1)(b-1)\overline{a0b}$& (5)
\\\hline&&&&\\
$\begin{array}{c}(4;{\bf 1},\overline{{\bf o_{\textrm{max}}}})\\=(4;{\bf 1})\&(2;\overline{{\bf o_{\textrm{max}}}})\end{array}$&$\begin{array}{c}(4;{\bf 2},\overline{{\bf 1}})\\=(4;{\bf 2})\&(1;\overline{{\bf 1}})\end{array}$&$(a-b)0\overline{ba0}$&$(a-b+1)(b-1)\overline{a0b}$& (6)
\\\hline&&&&\\
$\begin{array}{c}(5;{\bf 1+3k},\overline{{\bf o_{\textrm{max}}}})\\=(5;{\bf 1+3k})\&(7^-;\overline{{\bf o_{\textrm{max}}}}),\\k=0,\ldots,b-3\end{array}$&$\begin{array}{c}(5;{\bf 2+3k},\overline{{\bf 1}})\\=(5;{\bf 2+3k})\&(6^-;\overline{{\bf1}})\end{array}$&$(a-b+1+k)\overline{0ba}$&$(a-b+2+k)(b-1)\overline{a0b}$& (6)
\\\hline&&&&\\
$\begin{array}{c}(5;{\bf 2+3k},\overline{{\bf o_{\textrm{max}}}})\\=(5;{\bf 2+3k})\&(6^-;\overline{{\bf o_{\textrm{max}}}}),\\k=0,\ldots,b-3\end{array}$&$\begin{array}{c}(5;{\bf 3+3k},\overline{{\bf 1}})\\=(5;{\bf 3+3k})\&(5^-;\overline{{\bf1}})\end{array}$&$(a-b+2+k)(b-2)\overline{a0b}$&$(a-b+2+k)(b-2)\overline{a0b}$&
\\\hline&&&&\\
$\begin{array}{c}(5;{\bf 3+3k},\overline{{\bf o_{\textrm{max}}}})\\=(5;{\bf 3+3k})\&(5^-;\overline{{\bf o_{\textrm{max}}}}),\\k=0,\ldots,b-3\end{array}$&$\begin{array}{c}(5;{\bf 1+3(k+1)},\overline{{\bf 1}})\\=(5;{\bf 1+3(k+1)})\&(7^-;\overline{{\bf1}})\end{array}$&$(a-b+2+k)\overline{0a(a-b)}$&$(a-b+2+k)\overline{0a(a-b)}$&
\\\hline&&&&\\
$\begin{array}{c}(6;{\bf 1},\overline{{\bf o_{\textrm{max}}}})\\=(6;{\bf 1})\&(6^-;\overline{{\bf o_{\textrm{max}}}})\end{array}$&$\begin{array}{c}(6;{\bf 2},\overline{{\bf 1}})\\=(6;{\bf 2})\&(5^-;\overline{{\bf 1}})\end{array}$&$a(b-2)\overline{a0b}$&$a(b-2)\overline{a0b}$&
\\\hline&&&&\\
$\begin{array}{c}(6;{\bf 2},\overline{{\bf o_{\textrm{max}}}})\\=(6;{\bf 2})\&(5^-;\overline{{\bf o_{\textrm{max}}}})\end{array}$&$\begin{array}{c}(6;{\bf 3},\overline{{\bf 1}})\\=(6;{\bf 3})\&(7^-;\overline{{\bf 1}})\end{array}$&$a\overline{0a(a-b)}$&$a\overline{0a(a-b)}$&
\\\hline&&&&\\
$\begin{array}{c}(7;{\bf 1},\overline{{\bf o_{\textrm{max}}}})\\=(7;{\bf 1})\&(10^-;\overline{{\bf o_{\textrm{max}}}})\end{array}$&$\begin{array}{c}(7;{\bf 2},\overline{{\bf 1}})\\=(7;{\bf 2})\&(9^-;\overline{{\bf 1}})\end{array}$&$a0(b-1)\overline{a0b}$&$a0(b-1)\overline{a0b}$&
\\\hline&&&&\\
$\begin{array}{c}(7;{\bf 2},\overline{{\bf o_{\textrm{max}}}})\\=(7;{\bf 2})\&(9^-;\overline{{\bf o_{\textrm{max}}}})\end{array}$&$\begin{array}{c}(7;{\bf 3},\overline{{\bf 1}})\\=(7;{\bf 3})\&(8^-;\overline{{\bf 1}})\end{array}$&$a(a-b-1)\overline{0a(a-b)}$&$a(a-b-1)\overline{0a(a-b)}$&
\\\hline&&&&\\
$\begin{array}{c}(8;{\bf 1},\overline{{\bf o_{\textrm{max}}}})\\=(8;{\bf 1})\&(1^-;\overline{{\bf o_{\textrm{max}}}})\end{array}$&$\begin{array}{c}(8;{\bf 2},\overline{{\bf 1}})\\=(8;{\bf 2})\&(2^-;\overline{{\bf 1}})\end{array}$&$a(a-1)\overline{0ba}$&$a(a-1)\overline{0ba}$&
\\\hline&&&&\\
$\begin{array}{c}(9;{\bf 1+2k},\overline{{\bf o_{\textrm{max}}}})\\=(9;{\bf 1+2k})\&(1^-;\overline{{\bf o_{\textrm{max}}}}),\\k=0,\ldots,a-b-2\end{array}$&$\begin{array}{c}(9;{\bf 2+2k},\overline{{\bf 1}})\\=(9;{\bf 2+2k})\&(2^-;\overline{{\bf1}})\end{array}$&$(a-1-k)(a-1)\overline{0ba}$&$(a-1-k)(a-1)\overline{0ba}$&
\\\hline&&&&\\
$\begin{array}{c}(9;{\bf 2+2k},\overline{{\bf o_{\textrm{max}}}})\\=(9;{\bf 2+2k)})\&(2^-;\overline{{\bf o_{\textrm{max}}}}),\\k=0,\ldots,a-b-2\end{array}$&$\begin{array}{c}(9;{\bf 1+2(k+1)},\overline{{\bf 1}})\\=(9;{\bf 1+2(k+1)})\&(1^-;\overline{{\bf1}})\end{array}$&$(a-1-k)(a-1)\overline{(a-b)0a}$&$(a-2-k)\overline{(a-b)0a}$&(7)
\\\hline&&&&\\
$\begin{array}{c}(10;{\bf 1},\overline{{\bf o_{\textrm{max}}}})\\=(10;{\bf 1})\&(12^-;\overline{{\bf o_{\textrm{max}}}})\end{array}$&$\begin{array}{c}(10;{\bf 2},\overline{{\bf 1}})\\=(10;{\bf 2})\&(11^-;\overline{{\bf 1}})\end{array}$&$ba0\overline{a(a-b)0}$&$ba0\overline{a(a-b)0}$&
\\\hline&&&&\\
$\begin{array}{c}(12;{\bf 1},\overline{{\bf o_{\textrm{max}}}})\\=(12;{\bf 1})\&(5;\overline{{\bf o_{\textrm{max}}}})\end{array}$&$\begin{array}{c}(12;{\bf 2},\overline{{\bf 1}})\\=(12;{\bf 2})\&(6;\overline{{\bf 1}})\end{array}$&$(b-1)(a-1)\overline{0ba}$&$(b-1)a(b-1)\overline{a0b}$&(1)
\\\hline
\end{tabular}
\caption{Case $a\geq b+1,\;b\geq 2$, Condition~(\ref{Cond3}), $S=1,\ldots, 12$.}\label{table:abCont2A}
\end{center}
\end{table}

\begin{table}[h!]\scriptsize\begin{center}
\begin{tabular}{|c|c||c|c|c|}
\hline
 Walk 1& Walk 2& Sequence 1& Sequence 2 & Item\\
\hline &&&&\\
$\begin{array}{c}(1^-;{\bf 1+3k},\overline{{\bf o_{\textrm{max}}}})\\=(1^-;{\bf 1+3k})\&(7^-;\overline{{\bf o_{\textrm{max}}}}),\\k=0,\ldots,b-2\end{array}$&$\begin{array}{c}(1^-;{\bf 2+3k},\overline{{\bf 1}})\\=(1^-;{\bf 2+3k})\&(6^-;\overline{{\bf1}})\end{array}$&$(a-b+k)\overline{0ba}$&$(a-b+1+k)(b-1)\overline{a0b}$& (8)
\\\hline&&&&\\
$\begin{array}{c}(1^-;{\bf 2+3k},\overline{{\bf o_{\textrm{max}}}})\\=(1^-;{\bf 2+3k})\&(6^-;\overline{{\bf o_{\textrm{max}}}}),\\k=0,\ldots,b-2\end{array}$&$\begin{array}{c}(1^-;{\bf 3+3k},\overline{{\bf 1}})\\=(1^-;{\bf 3+3k})\&(5^-;\overline{{\bf1}})\end{array}$&$(a-b+1+k)(b-2)a\overline{0ba}$&$(a-b+1+k)(b-2)\overline{a0b}$&
\\\hline&&&&\\
$\begin{array}{c}(1^-;{\bf 3+3k},\overline{{\bf o_{\textrm{max}}}})\\=(1^-;{\bf 3+3k})\&(5^-;\overline{{\bf o_{\textrm{max}}}}),\\k=0,\ldots,b-2\end{array}$&$\begin{array}{c}(1^-;{\bf 1+3(k+1)},\overline{{\bf 1}})\\=(1^-;{\bf 1+3(k+1)})\&(7^-;\overline{{\bf1}})\end{array}$&$(a-b+1+k)0\overline{a(a-b)0}$&$(a-b+k+1)\overline{0a(a-b)}$&
\\\hline&&&&\\
$\begin{array}{c}(2^-;{\bf 1},\overline{{\bf o_{\textrm{max}}}})\\=(2^-;{\bf 1})\&(10^-;\overline{{\bf o_{\textrm{max}}}})\end{array}$&$\begin{array}{c}(2^-;{\bf 2},\overline{{\bf 1}})\\=(2^-;{\bf 2})\&(9^-;\overline{{\bf 1}})\end{array}$&$(a-1)0(b-1)a\overline{0ba}$&$(a-1)0(b-1)\overline{a0b}$&
\\\hline&&&&\\
$\begin{array}{c}(2^-;{\bf 2},\overline{{\bf o_{\textrm{max}}}})\\=(2^-;{\bf 2})\&(9^-;\overline{{\bf o_{\textrm{max}}}})\end{array}$&$\begin{array}{c}(2^-;{\bf 3},\overline{{\bf 1}})\\=(2^-;{\bf 3})\&(10^-;\overline{{\bf 1}})\end{array}$&$(a-1)(a-b-1)\overline{0a(a-b)}$&$(a-1)(a-b-1)\overline{0a(a-b)}$&
\\\hline&&&&\\
$\begin{array}{c}(5^-;{\bf 1+3k},\overline{{\bf o_{\textrm{max}}}})\\=(5^-;{\bf 1+3k})\&(7;\overline{{\bf o_{\textrm{max}}}}),\\k=0,\ldots,b-3\end{array}$&$\begin{array}{c}(5^-;{\bf 2+3k},\overline{{\bf 1}})\\=(5^-;{\bf 2+3k})\&(5;\overline{{\bf1}})\end{array}$&$(b-2-k)\overline{a(a-b)0}$&$(b-3-k)(a-b+1)\overline{0a(a-b)}$& (9)
\\\hline&&&&\\
$\begin{array}{c}(5^-;{\bf 2+3k},\overline{{\bf o_{\textrm{max}}}})\\=(5^-;{\bf 2+3k})\&(5;\overline{{\bf o_{\textrm{max}}}}),\\k=0,\ldots,b-3\end{array}$&$\begin{array}{c}(5^-;{\bf 3+3k},\overline{{\bf 1}})\\=(5^-;{\bf 3+3k})\&(6;\overline{{\bf1}})\end{array}$&$(b-3-k)(a-1)\overline{0ba}$&$(b-3-k)a(b-1)\overline{a0b}$&(1)
\\\hline&&&&\\
$\begin{array}{c}(5^-;{\bf 3+3k},\overline{{\bf o_{\textrm{max}}}})\\=(5^-;{\bf 3+3k})\&(6;\overline{{\bf o_{\textrm{max}}}}),\\k=0,\ldots,b-3\end{array}$&$\begin{array}{c}(5^-;{\bf 1+3(k+1)},\overline{{\bf 1}})\\=(5^-;{\bf 1+3(k+1)})\&(7;\overline{{\bf1}})\end{array}$&$(b-3-k)a\overline{0ba}$&$(b-3-k)\overline{a0b}$&
\\\hline&&&&\\
$\begin{array}{c}(6^-;{\bf 1},\overline{{\bf o_{\textrm{max}}}})\\=(6^-;{\bf 1})\&(7;\overline{{\bf o_{\textrm{max}}}})\end{array}$&$\begin{array}{c}(6^-;{\bf 2},\overline{{\bf 1}})\\=(6^-;{\bf 2})\&(5;\overline{{\bf 1}})\end{array}$&$(b-1)\overline{a(a-b)0}$&$(b-2)(a-b+1)\overline{0a(a-b)}$&(10)
\\\hline&&&&\\
$\begin{array}{c}(6^-;{\bf 2},\overline{{\bf o_{\textrm{max}}}})\\=(6^-;{\bf 2})\&(5;\overline{{\bf o_{\textrm{max}}}})\end{array}$&$\begin{array}{c}(6^-;{\bf 3},\overline{{\bf 1}})\\=(6^-;{\bf 3})\&(6;\overline{{\bf 1}})\end{array}$&$(b-2)(a-1)\overline{0ba}$&$(b-2)a(b-1)\overline{a0b}$&(1)
\\\hline&&&&\\
$\begin{array}{c}(7^-;{\bf 1},\overline{{\bf o_{\textrm{max}}}})\\=(7^-;{\bf 1})\&(8;\overline{{\bf o_{\textrm{max}}}})\end{array}$&$\begin{array}{c}(7^-;{\bf 2},\overline{{\bf 1}})\\=(7^-;{\bf 2})\&(9;\overline{{\bf 1}})\end{array}$&$0a(a-1)\overline{(a-b)0a}$&$0(a-1)\overline{(a-b)0a}$& (2)
\\\hline&&&&\\
$\begin{array}{c}(7^-;{\bf 2},\overline{{\bf o_{\textrm{max}}}})\\=(7^-;{\bf 2})\&(9;\overline{{\bf o_{\textrm{max}}}})\end{array}$&$\begin{array}{c}(7^-;{\bf 3},\overline{{\bf 1}})\\=(7^-;{\bf 3})\&(10;\overline{{\bf 1}})\end{array}$&$0b(a-1)\overline{0ba}$&$0ba(b-1)\overline{a0b}$& (1)
\\\hline&&&&\\
$\begin{array}{c}(8^-;{\bf 1},\overline{{\bf o_{\textrm{max}}}})\\=(8^-;{\bf 1})\&(2;\overline{{\bf o_{\textrm{max}}}})\end{array}$&$\begin{array}{c}(8^-;{\bf 2},\overline{{\bf 1}})\\=(8^-;{\bf 2})\&(1;\overline{{\bf 1}})\end{array}$&$(a-b-1)0\overline{ba0}$&$(a-b)(b-1)\overline{a0b}$&  (11)
\\\hline&&&&\\
$\begin{array}{c}(9^-;{\bf 1+2k},\overline{{\bf o_{\textrm{max}}}})\\=(9^-;{\bf 1+2k})\&(1;\overline{{\bf o_{\textrm{max}}}}),\\k=0,\ldots,a-b-2\end{array}$&$\begin{array}{c}(9^-;{\bf 2+2k},\overline{{\bf 1}})\\=(9^-;{\bf 2+2k})\&(2;\overline{{\bf1}})\end{array}$&$k\overline{0a(a-b)}$&$k\overline{0a(a-b)}$&
\\\hline&&&&\\
$\begin{array}{c}(9^-;{\bf 2+2k},\overline{{\bf o_{\textrm{max}}}})\\=(9^-;{\bf 2+2k)})\&(2;\overline{{\bf o_{\textrm{max}}}}),\\k=0,\ldots,a-b-2\end{array}$&$\begin{array}{c}(9^-;{\bf 1+2(k+1)},\overline{{\bf 1}})\\=(9^-;{\bf 1+2(k+1)})\&(1;\overline{{\bf1}})\end{array}$&$k0\overline{ba0}$&$(k+1)(b-1)\overline{a0b}$& (5)
\\\hline&&&&\\
$\begin{array}{c}(10^-;{\bf 1},\overline{{\bf o_{\textrm{max}}}})\\=(10^-;{\bf 1})\&(11;\overline{{\bf o_{\textrm{max}}}})\end{array}$&$\begin{array}{c}(10^-;{\bf 2},\overline{{\bf 1}})\\=(10^-;{\bf 2})\&(12;\overline{{\bf 1}})\end{array}$&$0b\overline{a(a-b)0}$&$0(b-1)(a-b+1)\overline{0a(a-b)}$&
\\\hline&&&&\\
$\begin{array}{c}(12^-;{\bf 1},\overline{{\bf o_{\textrm{max}}}})\\=(12^-;{\bf 1})\&(6^-;\overline{{\bf o_{\textrm{max}}}})\end{array}$&$\begin{array}{c}(12^-;{\bf 2},\overline{{\bf 1}})\\=(12^-;{\bf 2})\&(5^-;\overline{{\bf 1}})\end{array}$&$a(b-2)\overline{a0b}$&$a(b-2)\overline{a0b}$&
\\\hline
\end{tabular}
\caption{Case $a\geq b+1,\;b\geq 2$, Condition~(\ref{Cond3}), $S=1^-,\ldots, 12^{-}$.}\label{table:abCont2B}
\end{center}
\end{table}

\begin{table}[h!]\scriptsize\begin{center}
\begin{tabular}{|c|c||c|c|c|}
\hline
 Walk 1& Walk 2& Sequence 1& Sequence 2 &Checking via Lemma~\ref{CharacBound}: see Item\ldots\\
\hline 
\hline &&&&\\
$(4;\overline{{\bf o_{\textrm{max}}}})$&$(5;\overline{{\bf 1}})$&$(a-1)0\overline{1a0}$&$a0\overline{a01}$& (12)
\\\hline&&&&\\
$(5;\overline{{\bf o_{\textrm{max}}}})$&$(6;\overline{{\bf 1}})$&$a\overline{0a(a-1)}$&$a\overline{0a(a-1)}$&
\\\hline&&&&\\
$(9;\overline{{\bf o_{\textrm{max}}}})$&$(10;\overline{{\bf 1}})$&$1(a-1)\overline{01a}$&$1a0\overline{a01}$&(13)
\\\hline&&&&\\
$(11;\overline{{\bf o_{\textrm{max}}}})$&$(12;\overline{{\bf 1}})$&$1\overline{a(a-1)0}$&$0a\overline{0a(a-1)}$& (14)
\\\hline&&&&\\
$(12;\overline{{\bf o_{\textrm{max}}}})$&$(1;\overline{{\bf 1}})$&$0a\overline{01a}$&$0\overline{a01}$&
\\\hline&&&&\\
$\begin{array}{c}(2;{\bf 2},\overline{{\bf o_{\textrm{max}}}})\\=(2;{\bf 2})\&(9;\overline{{\bf o_{\textrm{max}}}})\end{array}$&$\begin{array}{c}(2;{\bf 3},\overline{{\bf 1}})\\=(2;{\bf 3})\&(10;\overline{{\bf 1}})\end{array}$&$01(a-1)\overline{01a}$&$01a0\overline{a01}$& (13)
\\\hline&&&&\\
$\begin{array}{c}(3;{\bf 1},\overline{{\bf o_{\textrm{max}}}})\\=(3;{\bf 1})\&(11;\overline{{\bf o_{\textrm{max}}}})\end{array}$&$\begin{array}{c}(3;{\bf 2},\overline{{\bf 1}})\\=(3;{\bf 2})\&(12;\overline{{\bf 1}})\end{array}$&$01\overline{a(a-1)0}$&$00a\overline{0a(a-1)}$& (14)
\\\hline&&&&\\
$\begin{array}{c}(3;{\bf 2},\overline{{\bf o_{\textrm{max}}}})\\=(3;{\bf 2})\&(12;\overline{{\bf o_{\textrm{max}}}})\end{array}$&$\begin{array}{c}(3;{\bf 3},\overline{{\bf 1}})\\=(3;{\bf 3})\&(1;\overline{{\bf 1}})\end{array}$&$00a\overline{01a}$&$00\overline{a01}$&
\\\hline&&&&\\
$\begin{array}{c}(10;{\bf 1},\overline{{\bf o_{\textrm{max}}}})\\=(10;{\bf 1})\&(12^-;\overline{{\bf o_{\textrm{max}}}})\end{array}$&$\begin{array}{c}(10;{\bf 2},\overline{{\bf 1}})\\=(10;{\bf 2})\&(11^-;\overline{{\bf 1}})\end{array}$&$1a0\overline{a(a-1)0}$&$1a\overline{0a(a-1)}$& 
\\\hline&&&&\\
$\begin{array}{c}(7^-;{\bf 2},\overline{{\bf o_{\textrm{max}}}})\\=(7^-;{\bf 2})\&(9;\overline{{\bf o_{\textrm{max}}}})\end{array}$&$\begin{array}{c}(7^-;{\bf 3},\overline{{\bf 1}})\\=(7^-;{\bf 3})\&(10;\overline{{\bf 1}})\end{array}$&$01(a-1)\overline{01a}$&$01a0\overline{a01}$& (13)
\\\hline
\end{tabular}
\caption{Case $a\geq 2,\;b=1$. See also Tables~\ref{table:abCont} to~\ref{table:abCont2B} with $b=1$.}\label{table:ab1Cont}
\end{center}
\end{table}

\begin{table}[h!]\scriptsize\begin{center}
\begin{tabular}{|c|c||c|c|c|}
\hline
 Walk 1& Walk 2& Sequence 1& Sequence 2 &Checking via Lemma~\ref{CharacBound}: see Item\ldots\\
\hline &&&&\\
$(1;\overline{{\bf o_{\textrm{max}}}})$&$(2;\overline{{\bf 1}})$&$\overline{0a0}$&$0\overline{a00}$&
\\\hline&&&&\\
$(2;\overline{{\bf o_{\textrm{max}}}})$&$(3;\overline{{\bf 1}})$&$0\overline{aa0}$&$0\overline{aa0}$&
\\\hline&&&&\\
$(3;\overline{{\bf o_{\textrm{max}}}})$&$(4;\overline{{\bf 1}})$&$0\overline{0a0}$&$0\overline{0a0}$&
\\\hline&&&&\\
$(4;\overline{{\bf o_{\textrm{max}}}})$&$(5;\overline{{\bf 1}})$&$1\overline{0a0}$&$1\overline{0a0}$&
\\\hline&&&&\\
$(5;\overline{{\bf o_{\textrm{max}}}})$&$(6;\overline{{\bf 1}})$&$(a-1)\overline{0aa}$&$a(a-1)\overline{a0a}$&(15)
\\\hline&&&&\\
$(6;\overline{{\bf o_{\textrm{max}}}})$&$(7;\overline{{\bf 1}})$&$a\overline{0aa}$&$\overline{a0a}$&
\\\hline&&&&\\
$(7;\overline{{\bf o_{\textrm{max}}}})$&$(8;\overline{{\bf 1}})$&$\overline{a00}$&$\overline{a00}$&
\\\hline&&&&\\
$(8;\overline{{\bf o_{\textrm{max}}}})$&$(9;\overline{{\bf 1}})$&$a(a-1)\overline{0aa}$&$aa(a-1)\overline{a0a}$&(15)
\\\hline&&&&\\
$(9;\overline{{\bf o_{\textrm{max}}}})$&$(10;\overline{{\bf 1}})$&$\overline{aa0}$&$\overline{aa0}$&
\\\hline&&&&\\
$(10;\overline{{\bf o_{\textrm{max}}}})$&$(11;\overline{{\bf 1}})$&$a\overline{a00}$&$(a-1)1\overline{0a0}$& (16)
\\\hline&&&&\\
$(11;\overline{{\bf o_{\textrm{max}}}})$&$(12;\overline{{\bf 1}})$&$(a-1)a\overline{0aa}$&$(a-1)\overline{a0a}$&
\\\hline
\end{tabular}

\caption{Case $a=b\geq 2$, Conditions~(\ref{Cond1}) and~(\ref{Cond2}).}\label{table:aeqbCont}
\end{center}
\end{table}

\begin{table}[h!]\scriptsize
\begin{center}
\begin{tabular}{|c|c||c|c|c|}
\hline
 Walk 1& Walk 2& Sequence 1& Sequence 2 & Item \\
\hline &&&&\\
$\begin{array}{c}(1;{\bf 1+3k},\overline{{\bf o_{\textrm{max}}}})\\=(1;{\bf 1+3k})\&(7;\overline{{\bf o_{\textrm{max}}}}),\\k=0,\ldots,a-2\end{array}$&$\begin{array}{c}(1;{\bf 2+3k},\overline{{\bf 1}})\\=(1;{\bf 2+3k})\&(5;\overline{{\bf1}})\end{array}$&$(a-1-k)\overline{a00}$&$(a-2-k)1\overline{0a0}$&(17)
\\\hline&&&&\\
$\begin{array}{c}(1;{\bf 2+3k},\overline{{\bf o_{\textrm{max}}}})\\=(1;{\bf 2+3k})\&(5;\overline{{\bf o_{\textrm{max}}}}),\\k=0,\ldots,a-2\end{array}$&$\begin{array}{c}(1;{\bf 3+3k},\overline{{\bf 1}})\\=(1;{\bf 3+3k})\&(6;\overline{{\bf1}})\end{array}$&$(a-2-k)(a-1)\overline{0aa}$&$(a-2-k)a(a-1)\overline{a0a}$& (15)
\\\hline&&&&\\
$\begin{array}{c}(1;{\bf 3+3k},\overline{{\bf o_{\textrm{max}}}})\\=(1;{\bf 3+3k})\&(6;\overline{{\bf o_{\textrm{max}}}}),\\k=0,\ldots,a-2\end{array}$&$\begin{array}{c}(1;{\bf 1+3(k+1)},\overline{{\bf 1}})\\=(1;{\bf 1+3(k+1)})\&(7;\overline{{\bf1}})\end{array}$&$(a-2-k)a\overline{0a}$&$(a-2-k)\overline{a0a}$&
\\\hline&&&&\\
$\begin{array}{c}(2;{\bf 1},\overline{{\bf o_{\textrm{max}}}})\\=(2;{\bf 1})\&(8;\overline{{\bf o_{\textrm{max}}}})\end{array}$&$\begin{array}{c}(2;{\bf 2},\overline{{\bf 1}})\\=(2;{\bf 2})\&(9;\overline{{\bf 1}})\end{array}$&$0a(a-1)\overline{0aa}$&$0aa(a-1)\overline{a0a}$&(15)
\\\hline&&&&\\
$\begin{array}{c}(3;{\bf 1},\overline{{\bf o_{\textrm{max}}}})\\=(3;{\bf 1})\&(10;\overline{{\bf o_{\textrm{max}}}})\end{array}$&$\begin{array}{c}(3;{\bf 2},\overline{{\bf 1}})\\=(3;{\bf 2})\&(11;\overline{{\bf 1}})\end{array}$&$0a\overline{a00}$&$0(a-1)1\overline{0a0}$& (16)
\\\hline&&&&\\
$\begin{array}{c}(3;{\bf 2},\overline{{\bf o_{\textrm{max}}}})\\=(3;{\bf 2})\&(11;\overline{{\bf o_{\textrm{max}}}})\end{array}$&$\begin{array}{c}(3;{\bf 3},\overline{{\bf 1}})\\=(3;{\bf 3})\&(1;\overline{{\bf 1}})\end{array}$&$0(a-1)a\overline{0aa}$&$0(a-1)\overline{a0a}$&
\\\hline&&&&\\
$\begin{array}{c}(4;{\bf 1},\overline{{\bf o_{\textrm{max}}}})\\=(4;{\bf 1})\&(2;\overline{{\bf o_{\textrm{max}}}})\end{array}$&$\begin{array}{c}(4;{\bf 2},\overline{{\bf 1}})\\=(4;{\bf 2})\&(1;\overline{{\bf 1}})\end{array}$&$00\overline{aa0}$&$1(a-1)\overline{a0a}$& (17)
\\\hline&&&&\\
$\begin{array}{c}(5;{\bf 1+3k},\overline{{\bf o_{\textrm{max}}}})\\=(5;{\bf 1+3k})\&(7^-;\overline{{\bf o_{\textrm{max}}}}),\\k=0,\ldots,a-3\end{array}$&$\begin{array}{c}(5;{\bf 2+3k},\overline{{\bf 1}})\\=(5;{\bf 2+3k})\&(6^-;\overline{{\bf1}})\end{array}$&$(k+1)\overline{0aa}$&$(k+2)(a-1)\overline{a0a}$& (17)
\\\hline&&&&\\
$\begin{array}{c}(5;{\bf 2+3k},\overline{{\bf o_{\textrm{max}}}})\\=(5;{\bf 2+3k})\&(6^-;\overline{{\bf o_{\textrm{max}}}}),\\k=0,\ldots,a-3\end{array}$&$\begin{array}{c}(5;{\bf 3+3k},\overline{{\bf 1}})\\=(5;{\bf 3+3k})\&(5^-;\overline{{\bf1}})\end{array}$&$(k+2)(a-2)\overline{a0a}$&$(k+2)(a-2)\overline{a0a}$&
\\\hline&&&&\\
$\begin{array}{c}(5;{\bf 3+3k},\overline{{\bf o_{\textrm{max}}}})\\=(5;{\bf 3+3k})\&(5^-;\overline{{\bf o_{\textrm{max}}}}),\\k=0,\ldots,a-3\end{array}$&$\begin{array}{c}(5;{\bf 1+3(k+1)},\overline{{\bf 1}})\\=(5;{\bf 1+3(k+1)})\&(7^-;\overline{{\bf1}})\end{array}$&$(k+2)\overline{0a0}$&$(k+2)\overline{0a0}$&
\\\hline&&&&\\
$\begin{array}{c}(6;{\bf 1},\overline{{\bf o_{\textrm{max}}}})\\=(6;{\bf 1})\&(6^-;\overline{{\bf o_{\textrm{max}}}})\end{array}$&$\begin{array}{c}(6;{\bf 2},\overline{{\bf 1}})\\=(6;{\bf 2})\&(5^-;\overline{{\bf 1}})\end{array}$&$a(a-2)\overline{a0a}$&$a(a-2)\overline{a0a}$&
\\\hline&&&&\\
$\begin{array}{c}(6;{\bf 2},\overline{{\bf o_{\textrm{max}}}})\\=(6;{\bf 2})\&(5^-;\overline{{\bf o_{\textrm{max}}}})\end{array}$&$\begin{array}{c}(6;{\bf 3},\overline{{\bf 1}})\\=(6;{\bf 3})\&(7^-;\overline{{\bf 1}})\end{array}$&$a\overline{0a0}$&$a\overline{0a0}$&
\\\hline&&&&\\
$\begin{array}{c}(7;{\bf 1},\overline{{\bf o_{\textrm{max}}}})\\=(7;{\bf 1})\&(9^-;\overline{{\bf o_{\textrm{max}}}})\end{array}$&$\begin{array}{c}(7;{\bf 2},\overline{{\bf 1}})\\=(7;{\bf 2})\&(8^-;\overline{{\bf 1}})\end{array}$&$a0(a-1)\overline{a0a}$&$a0(a-1)\overline{a0a}$&
\\\hline&&&&\\
$\begin{array}{c}(9;{\bf 1},\overline{{\bf o_{\textrm{max}}}})\\=(9;{\bf 1})\&(11^-;\overline{{\bf o_{\textrm{max}}}})\end{array}$&$\begin{array}{c}(9;{\bf 2},\overline{{\bf 1}})\\=(9;{\bf 2})\&(10^-;\overline{{\bf 1}})\end{array}$&$aa0\overline{a00}$&$aa0\overline{a00}$&
\\\hline&&&&\\
$\begin{array}{c}(11;{\bf 1},\overline{{\bf o_{\textrm{max}}}})\\=(11;{\bf 1})\&(5;\overline{{\bf o_{\textrm{max}}}})\end{array}$&$\begin{array}{c}(11;{\bf 2},\overline{{\bf 1}})\\=(11;{\bf 2})\&(6;\overline{{\bf 1}})\end{array}$&$(a-1)(a-1)\overline{0aa}$&$(a-1)a(a-1)\overline{a0a}$& (15)
\\\hline
\end{tabular}
\caption{Case $a= b\geq 2$, Condition~(\ref{Cond3}), $S=1,\ldots, 11$.}\label{table:aeqbCont2A}
\end{center}
\end{table}

\begin{table}[h!]\scriptsize\begin{center}
\begin{tabular}{|c|c||c|c|c|}
\hline
 Walk 1& Walk 2& Sequence 1& Sequence 2 & Item\\
\hline &&&&\\
$\begin{array}{c}(1^-;{\bf 1+3k},\overline{{\bf o_{\textrm{max}}}})\\=(1^-;{\bf 1+3k})\&(7^-;\overline{{\bf o_{\textrm{max}}}}),\\k=0,\ldots,a-2\end{array}$&$\begin{array}{c}(1^-;{\bf 2+3k},\overline{{\bf 1}})\\=(1^-;{\bf 2+3k})\&(6^-;\overline{{\bf1}})\end{array}$&$k\overline{0aa}$&$(k+1)(a-1)\overline{a0a}$& (18)
\\\hline&&&&\\
$\begin{array}{c}(1^-;{\bf 2+3k},\overline{{\bf o_{\textrm{max}}}})\\=(1^-;{\bf 2+3k})\&(6^-;\overline{{\bf o_{\textrm{max}}}}),\\k=0,\ldots,a-2\end{array}$&$\begin{array}{c}(1^-;{\bf 3+3k},\overline{{\bf 1}})\\=(1^-;{\bf 3+3k})\&(5^-;\overline{{\bf1}})\end{array}$&$(k+1)(a-2)a\overline{0aa}$&$(k+1)(a-2)\overline{a0a}$&
\\\hline&&&&\\
$\begin{array}{c}(1^-;{\bf 3+3k},\overline{{\bf o_{\textrm{max}}}})\\=(1^-;{\bf 3+3k})\&(5^-;\overline{{\bf o_{\textrm{max}}}}),\\k=0,\ldots,a-2\end{array}$&$\begin{array}{c}(1^-;{\bf 1+3(k+1)},\overline{{\bf 1}})\\=(1^-;{\bf 1+3(k+1)})\&(7^-;\overline{{\bf1}})\end{array}$&$(k+1)0\overline{a00}$&$(k+1)\overline{0a0}$&
\\\hline&&&&\\
$\begin{array}{c}(5^-;{\bf 1+3k},\overline{{\bf o_{\textrm{max}}}})\\=(5^-;{\bf 1+3k})\&(7;\overline{{\bf o_{\textrm{max}}}}),\\k=0,\ldots,a-3\end{array}$&$\begin{array}{c}(5^-;{\bf 2+3k},\overline{{\bf 1}})\\=(5^-;{\bf 2+3k})\&(5;\overline{{\bf1}})\end{array}$&$(a-2-k)\overline{a00}$&$(a-3-k)1\overline{0a0}$& (19)
\\\hline&&&&\\
$\begin{array}{c}(5^-;{\bf 2+3k},\overline{{\bf o_{\textrm{max}}}})\\=(5^-;{\bf 2+3k})\&(5;\overline{{\bf o_{\textrm{max}}}}),\\k=0,\ldots,a-3\end{array}$&$\begin{array}{c}(5^-;{\bf 3+3k},\overline{{\bf 1}})\\=(5^-;{\bf 3+3k})\&(6;\overline{{\bf1}})\end{array}$&$(a-3-k)(a-1)\overline{0aa}$&$(a-3-k)a(a-1)\overline{a0a}$& (15)
\\\hline&&&&\\
$\begin{array}{c}(5^-;{\bf 3+3k},\overline{{\bf o_{\textrm{max}}}})\\=(5^-;{\bf 3+3k})\&(6;\overline{{\bf o_{\textrm{max}}}}),\\k=0,\ldots,a-3\end{array}$&$\begin{array}{c}(5^-;{\bf 1+3(k+1)},\overline{{\bf 1}})\\=(5^-;{\bf 1+3(k+1)})\&(7;\overline{{\bf1}})\end{array}$&$(a-3-k)a\overline{0aa}$&$(a-3-k)\overline{a0a}$&
\\\hline&&&&\\
$\begin{array}{c}(6^-;{\bf 1},\overline{{\bf o_{\textrm{max}}}})\\=(6^-;{\bf 1})\&(7;\overline{{\bf o_{\textrm{max}}}})\end{array}$&$\begin{array}{c}(6^-;{\bf 2},\overline{{\bf 1}})\\=(6^-;{\bf 2})\&(5;\overline{{\bf 1}})\end{array}$&$(a-1)\overline{a00}$&$(a-2)1\overline{0a0}$& (20)
\\\hline&&&&\\
$\begin{array}{c}(6^-;{\bf 2},\overline{{\bf o_{\textrm{max}}}})\\=(6^-;{\bf 2})\&(5;\overline{{\bf o_{\textrm{max}}}})\end{array}$&$\begin{array}{c}(6^-;{\bf 3},\overline{{\bf 1}})\\=(6^-;{\bf 3})\&(6;\overline{{\bf 1}})\end{array}$&$(a-2)(a-1)\overline{0aa}$&$(a-2)a(a-1)\overline{a0a}$& (15)
\\\hline&&&&\\
$\begin{array}{c}(7^-;{\bf 1},\overline{{\bf o_{\textrm{max}}}})\\=(7^-;{\bf 1})\&(8;\overline{{\bf o_{\textrm{max}}}})\end{array}$&$\begin{array}{c}(7^-;{\bf 2},\overline{{\bf 1}})\\=(7^-;{\bf 2})\&(9;\overline{{\bf 1}})\end{array}$&$0a(a-1)\overline{0aa}$&$0aa(a-1)\overline{a0a}$& (15)
\\\hline&&&&\\
$\begin{array}{c}(9^-;{\bf 1},\overline{{\bf o_{\textrm{max}}}})\\=(9^-;{\bf 1})\&(10;\overline{{\bf o_{\textrm{max}}}})\end{array}$&$\begin{array}{c}(9^-;{\bf 2},\overline{{\bf 1}})\\=(9^-;{\bf 2})\&(11;\overline{{\bf 1}})\end{array}$&$0a\overline{a00}$&$0(a-1)1\overline{0a0}$& (16)
\\\hline&&&&\\
$\begin{array}{c}(11^-;{\bf 1},\overline{{\bf o_{\textrm{max}}}})\\=(11^-;{\bf 1})\&(6^-;\overline{{\bf o_{\textrm{max}}}})\end{array}$&$\begin{array}{c}(11^-;{\bf 2},\overline{{\bf 1}})\\=(11^-;{\bf 2})\&(5^-;\overline{{\bf 1}})\end{array}$&$a(a-2)a\overline{0aa}$&$a(a-2)\overline{a0a}$&
\\\hline
\end{tabular}
\caption{Case $a=b\geq 2$, Condition~(\ref{Cond3}), $S=1^-,\ldots, 12^{-}$.}\label{table:aeqbCont2B}
\end{center}
\end{table}

\begin{table}[h!]\scriptsize\begin{center}
\begin{tabular}{|c|c||c|c|c|}
\hline
 Walk 1& Walk 2& Sequence 1& Sequence 2 & Checking via Lemma~\ref{CharacBound}: see Item\ldots\\
\hline 
$(4;\overline{{\bf o_{\textrm{max}}}})$&$(5;\overline{{\bf 1}})$&$00\overline{110}$&$10\overline{101}$&(22)
\\\hline&&&&\\
$(5;\overline{{\bf o_{\textrm{max}}}})$&$(6;\overline{{\bf 1}})$&$1\overline{010}$&$1\overline{010}$&
\\\hline&&&&\\
$(10;\overline{{\bf o_{\textrm{max}}}})$&$(11;\overline{{\bf 1}})$&$1\overline{100}$&$01\overline{010}$& (23)
\\\hline&&&&\\
$\begin{array}{c}(3;{\bf 1},\overline{{\bf o_{\textrm{max}}}})\\=(3;{\bf 1})\&(10;\overline{{\bf o_{\textrm{max}}}})\end{array}$&$\begin{array}{c}(3;{\bf 2},\overline{{\bf 1}})\\=(3;{\bf 2})\&(11;\overline{{\bf 1}})\end{array}$&$01\overline{100}$&$001\overline{010}$& (23)
\\\hline&&&&\\
$\begin{array}{c}(9;{\bf 1},\overline{{\bf o_{\textrm{max}}}})\\=(9;{\bf 1})\&(11^-;\overline{{\bf o_{\textrm{max}}}})\end{array}$&$\begin{array}{c}(9;{\bf 2},\overline{{\bf 1}})\\=(9;{\bf 2})\&(10^-;\overline{{\bf 1}})\end{array}$&$110\overline{100}$&$11\overline{010}$&
\\\hline&&&&\\
$\begin{array}{c}(9^-;{\bf 1},\overline{{\bf o_{\textrm{max}}}})\\=(9^-;{\bf 1})\&(10;\overline{{\bf o_{\textrm{max}}}})\end{array}$&$\begin{array}{c}(9^-;{\bf 2},\overline{{\bf 1}})\\=(9^-;{\bf 2})\&(11;\overline{{\bf 1}})\end{array}$&$01\overline{100}$&$001\overline{010}$& (23)
\\\hline
\end{tabular}
\caption{Case $a=b=1$.} \label{table:a1b1Cont}
\end{center}
\end{table}

\pagebreak

\section{Details for the proof of Theorem~\ref{DNDTheo} (disk-like tiles).}\label{app:TheoDisk}
We depict the automata $\mathcal{A}^\psi$ and $\mathcal{A}^{sl}$ for the remaining cases:
\begin{itemize}
\item $b=1,a\geq 2$, Figures~\ref{fig:APSIab1} and~\ref{fig:ASLab1};
\item $a=b\leq3$ (recall that $2b-a\leq 3$), Figures~\ref{fig:APSIaeqb} and~\ref{fig:ASLaeqb}, for $a\in\{2,3\}$, as well as Figures~\ref{fig:APSIa1b1} and~\ref{fig:ASLa1b1} for $a=b=1$. 
 \end{itemize}
Again, in $\mathcal{A}^{\psi}$ as well as in $\mathcal{A}^{sl}$, no more pairs $(w,w')$ than the pairs given in~(\ref{identifpairs}) are found. Thus the same conclusion as in the core of the proof of Theorem~\ref{DNDTheo} (disk-like tiles) applies.

\begin{figure}[h!]
\begin{center}
\includegraphics[width=145mm,height=130mm]{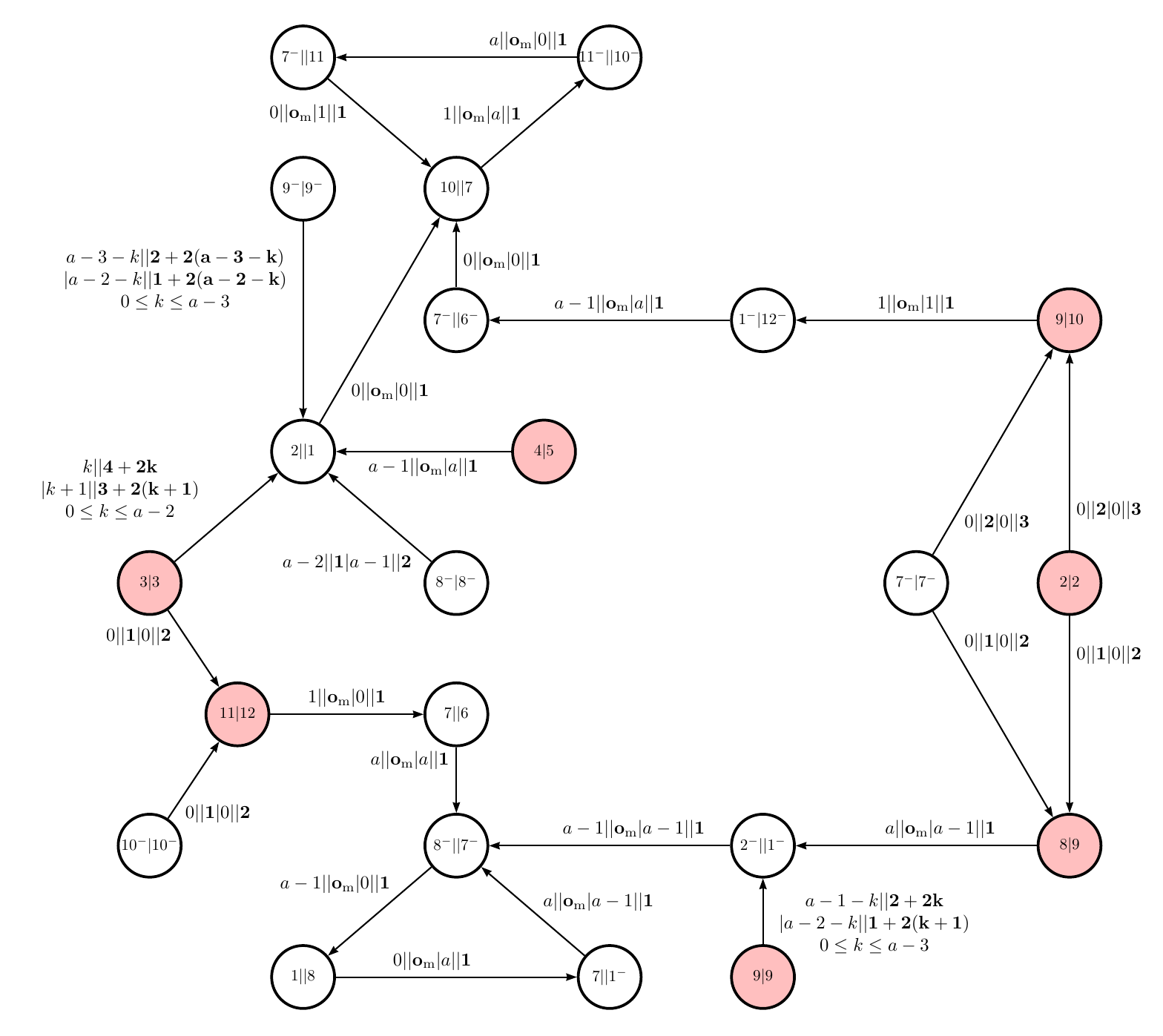}\end{center}
\caption{$\mathcal{A}^\psi$ for $b=1, a\geq 2$.}\label{fig:APSIab1}
\end{figure}

\begin{figure}[h!]
\begin{center}
\includegraphics[width=140mm,height=115mm]{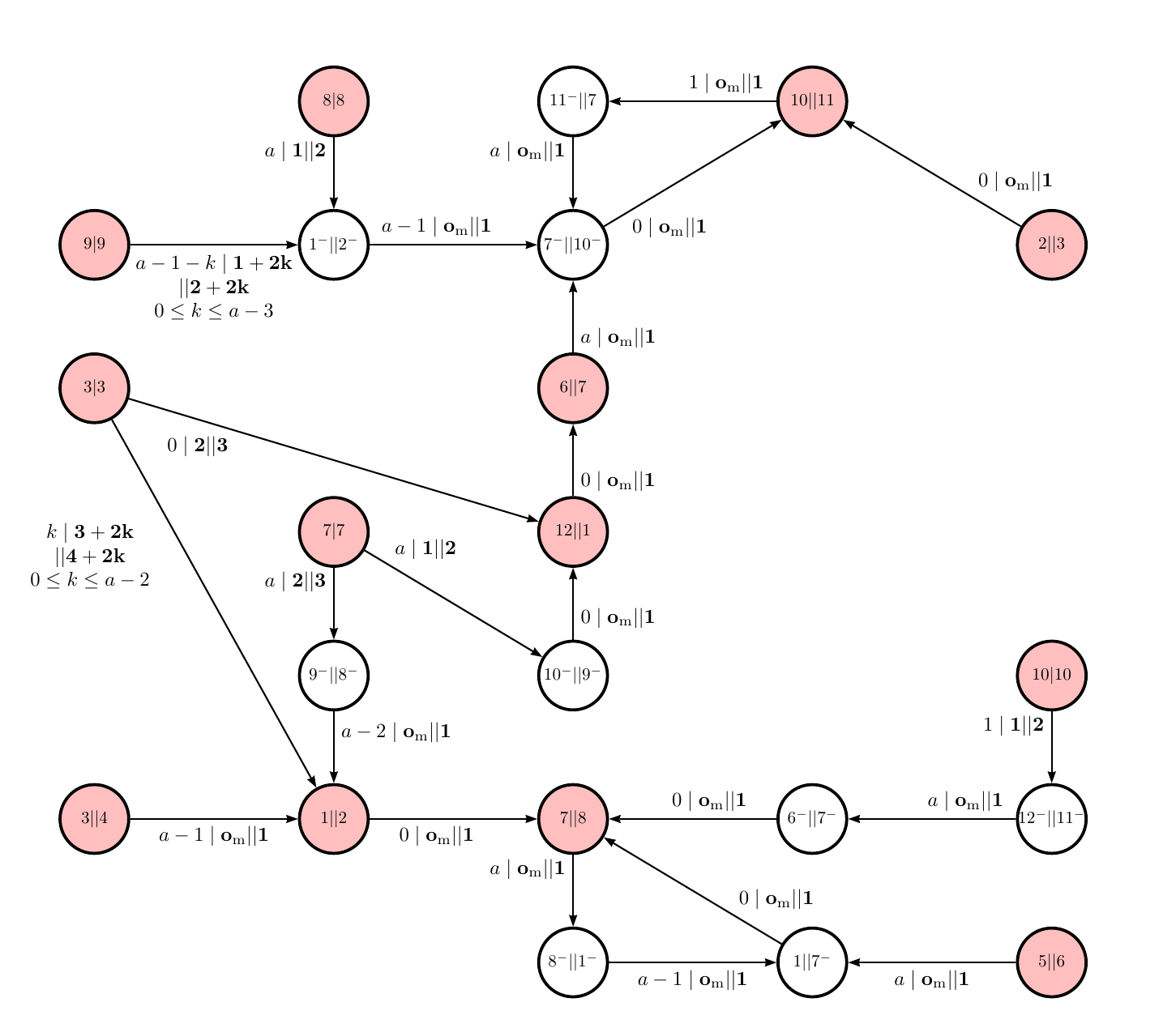}
\end{center}
\caption{$\mathcal{A}^{sl}$ for $b=1, a\geq 2$.}\label{fig:ASLab1}
\end{figure}

\begin{figure}[h!]
\begin{center}
\includegraphics[width=150mm,height=125mm]{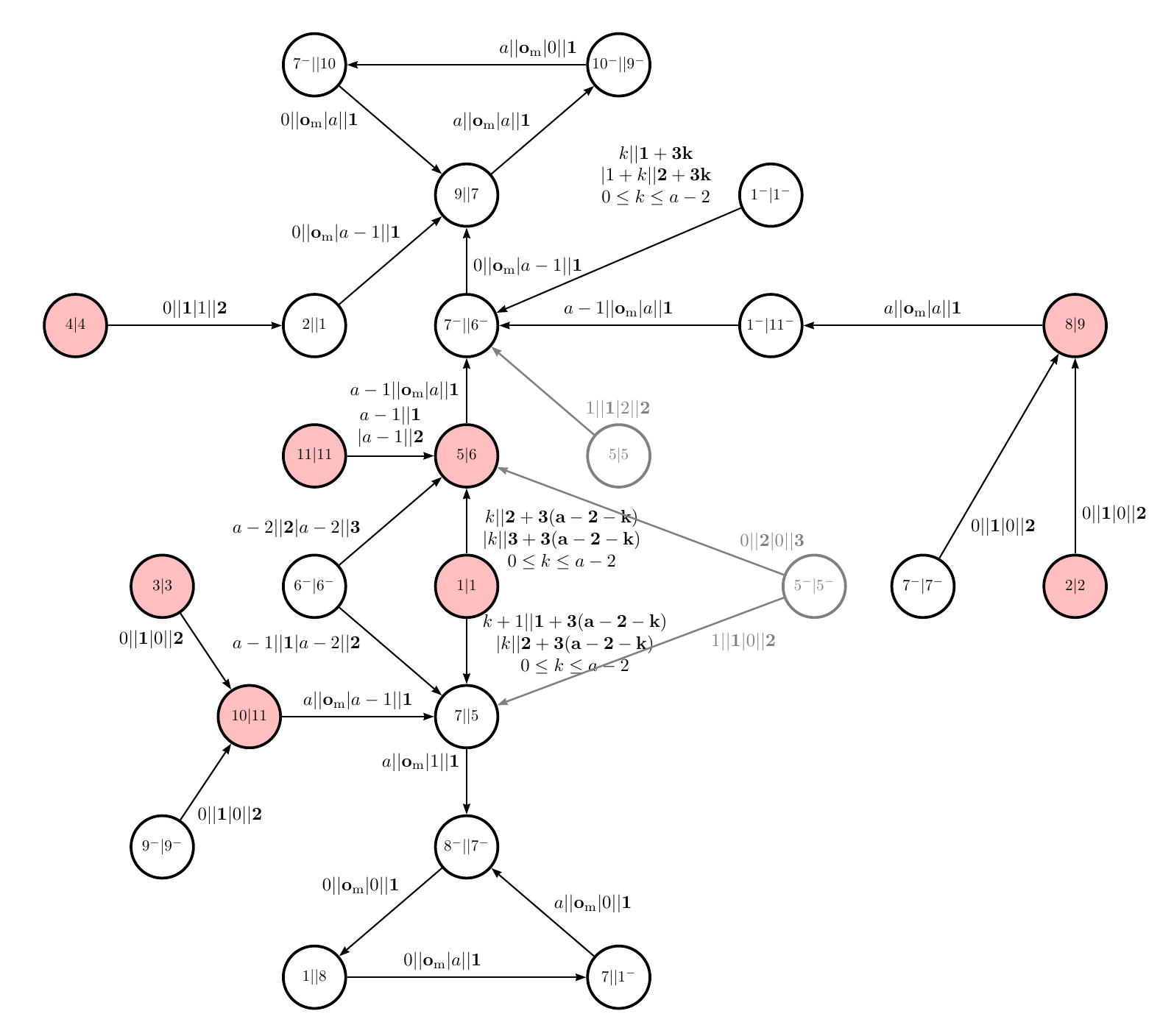}
\end{center}
\caption{$\mathcal{A}^\psi$ for $2\leq a=b\leq 3$ (dimmed states and edges only for $a=3$).}\label{fig:APSIaeqb}
\end{figure}

\begin{figure}[h!]
\begin{center}

\includegraphics[width=140mm,height=125mm]{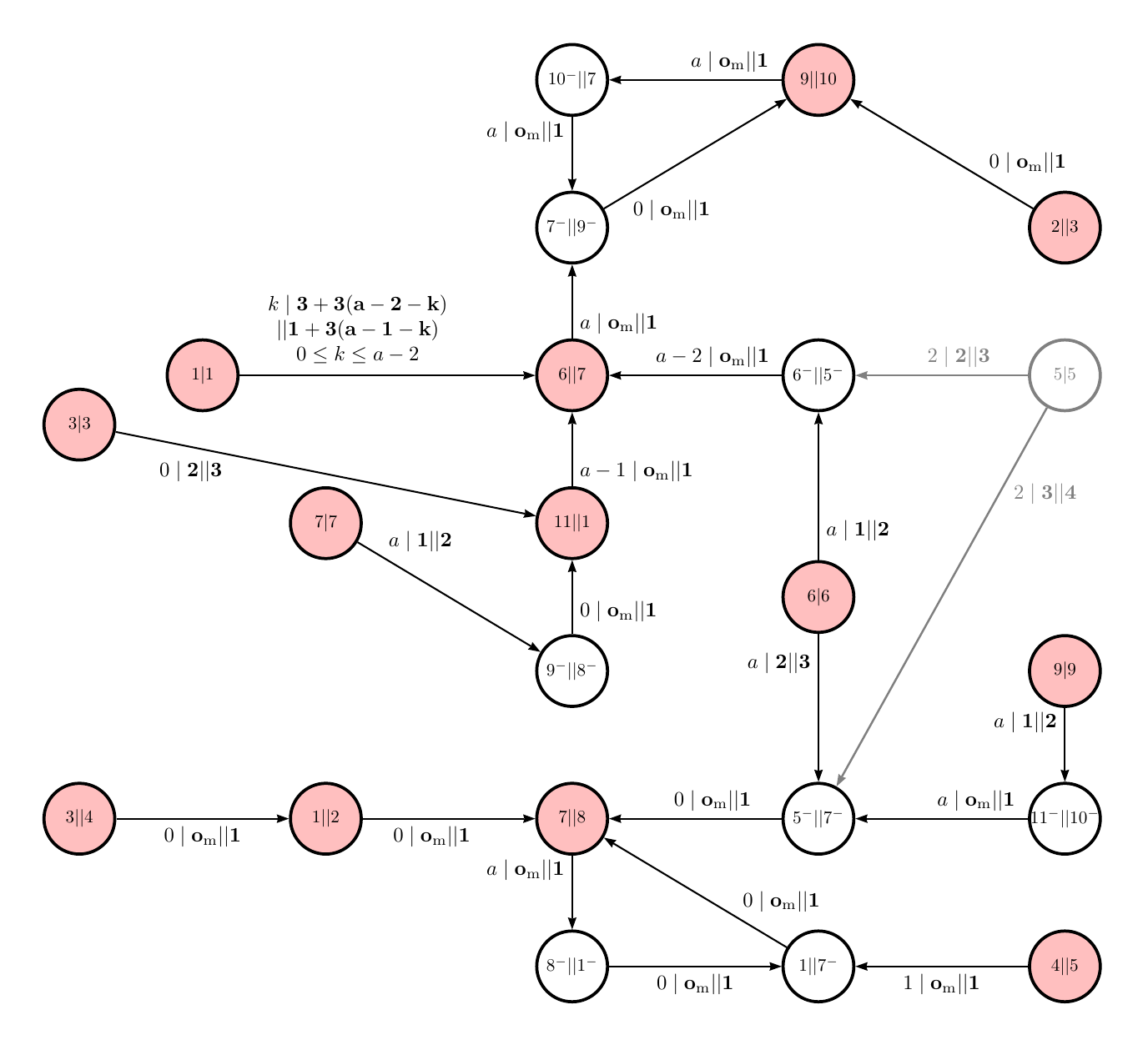}
\end{center}
\caption{$\mathcal{A}^{sl}$ for $2\leq a=b\leq 3$ (dimmed state and edges only for $a=3$).}\label{fig:ASLaeqb}
\end{figure}

\begin{figure}[h!]
\begin{center}
\includegraphics[width=130mm,height=50mm]{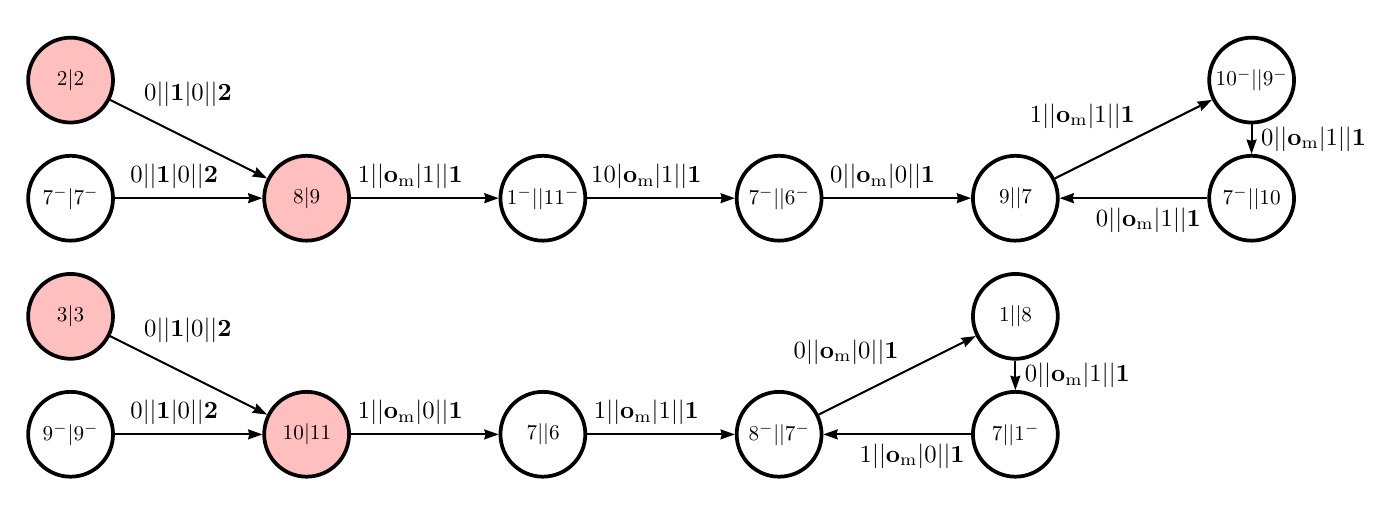}
\end{center}
\caption{$\mathcal{A}^\psi$   for $a=b=1$.}\label{fig:APSIa1b1}
\end{figure}

\begin{figure}[h!]
\begin{center}
\includegraphics[width=125mm,height=50mm]{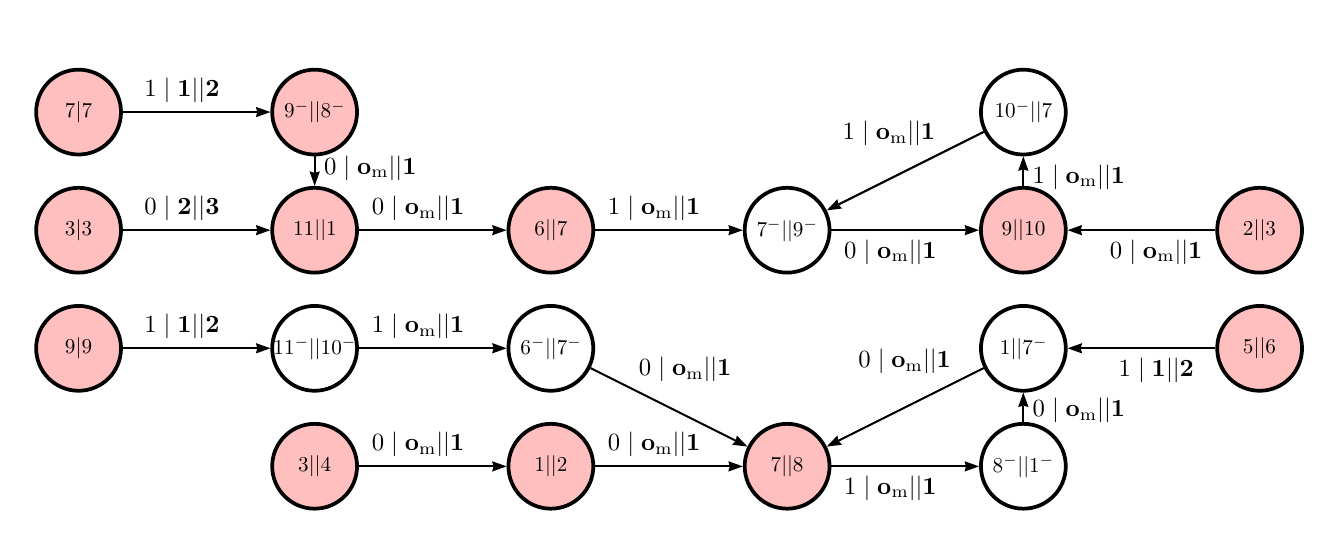}\end{center}

\caption{$\mathcal{A}^{sl}$  for $a=b=1$.}\label{fig:ASLa1b1}
\end{figure}

\bibliographystyle{plain}
\bibliography{biblio}

\end{document}